\documentclass[11pt]{article}
\usepackage[margin=1in]{geometry}
\usepackage{amsmath,amssymb,amsthm}
\usepackage[T1]{fontenc}
\usepackage{tikz}
\usepackage{pgfplots}
\usepackage{float}
\pgfplotsset{compat=1.18}
\newtheorem{theorem}{Theorem}[section]
\newtheorem{proposition}[theorem]{Proposition}
\newtheorem{lemma}[theorem]{Lemma}
\newtheorem{definition}[theorem]{Definition}
\newtheorem{corollary}[theorem]{Corollary}
\newtheorem{conjecture}[theorem]{Conjecture}
\theoremstyle{remark}
\newtheorem{remark}[theorem]{Remark}
\title{Spectral duality structures and the Fisher--Rao geometry of
reset distributions}
\author{Juan Antonio Vega Coso\\[2pt]
\normalsize Instituto Universitario de F\'isica Fundamental y
Matem\'aticas (IUFFyM),\\
\normalsize Universidad de Salamanca, Plaza de la Merced s/n,
E-37008 Salamanca, Spain}
\date{}
\begin{document}
\maketitle
\begin{abstract}
We study the geometry that spectral duality induces on the simplex
of reset distributions for absorbed Markov processes with geometric
resetting. The simplex carries a canonical Riemannian structure ---
the Fisher--Rao metric, selected by \v{C}encov's theorem --- and we
show that in this geometry the separatrix $\Sigma$ of reset-neutral
distributions is rigid: under the square-root embedding
$\pi\mapsto2\sqrt\pi$ its interior becomes a totally geodesic
subsphere of the positive octant, and is itself a Fisher--Rao
simplex of lower dimension. We then reduce the reset response to a finite structure.
The response functionals $\psi(\gamma)$ span a subspace $V$ whose
dimension $r$ equals the number of active orbits of the duality
involution, and the local separatrix is its annihilator; at the
vertices of the simplex we prove a sign theorem valid for every $r$,
which recovers the two-zone phenomenon of the first paper of this
series as the trace of the geometry along the vertex diagonal.
Finally, the invariant $r$ settles the question of global orientation
raised in Paper~III. For $r=1$ all response functionals lie on
a single line and, under a scalar sign condition satisfied by the
canonical realisation, the sign of the response is fixed once and
for all by the side of $\Sigma$; this is a sign law referred to a
fixed director, and it reproduces the pointwise identification of
Paper~III only when the two orientations agree, which the canonical
realisation guarantees but the abstract class does not. For
$r\ge2$ the span has dimension at least two and the orientation
cannot stay constant; we give the criterion under which no such
global law can hold, and record the counterexamples, due to
T.~Newton, that show it is met in the abstract class.
The biased random walk with multi-site geometric resetting realises
the whole picture explicitly. This is the fifth paper in a program
connecting stochastic resetting with spectral theory and information
geometry.
\end{abstract}

\section{Introduction}
\label{sec:intro}

First-passage problems in confined domains sit at the intersection
of probability, statistical physics, and the spectral theory of the
generators that govern them \cite{Redner2001,Bray2013,Schuss2015}.
The gambler's ruin problem --- the
probability that a biased random walk on a finite interval is
absorbed at one end before the other \cite{Feller1968} --- is the
canonical example,
and the effect of \emph{stochastic resetting} on such problems has
been studied intensively in recent years
\cite{EvMaj2011,EvMaj2020,Reuveni2016,PalReuveni2017,VillarroelMonteroVega2021}:
at random times the
process is returned to a prescribed distribution, and the absorption
statistics change in ways that are often far from obvious. Among the
phenomena uncovered by this programme, one has proved especially
persistent. For the biased walk with geometric resetting there exist
\emph{reset-neutral} initial data --- distributions from which the
ruin probability does not depend on the resetting rate at all
\cite{PaperI,PaperII}. The first two papers of this series
identified this invariance and computed the neutral distributions
explicitly; what they did not explain is \emph{why} it occurs, or
what governs the way the ruin probability responds to resetting away
from neutrality. Paper~III \cite{PaperIII} took the first step
towards an explanation: it isolated the structural conditions under
which the phenomenon arises, identified the critical scalar
$C^\ast$ and the set of reset-neutral distributions it cuts out,
and asked whether a single linear functional orients the response
across the whole simplex. That question is the point of departure
here.

The purpose of the present paper is to answer that question by
changing the object of study. Rather than track the ruin probability
from a single starting site, we consider the whole simplex
$\Delta_{m-1}$ of reset distributions and the functional
$C(\pi,\gamma)$ that a distribution $\pi$ and a resetting rate
$\gamma$ produce. The reset-neutral data of the earlier papers are
then no longer isolated points but a submanifold --- a
\emph{separatrix} $\Sigma$ --- and the response of $C$ to $\gamma$
acquires an orientation: on one side of $\Sigma$ it increases, on
the other it decreases, and on $\Sigma$ it is constant. The
two-zone picture of Paper~I \cite{PaperI} is recovered as the trace
of this geometry along the vertices of the simplex, where resetting to a
site and starting from it coincide; the neutral starting point
becomes the vertex at which the separatrix meets the boundary. The
pointwise phenomenon and the global geometry are, in this sense, the
same fact seen at two resolutions.

Read this way, the programme is seen to carry two geometries at
different levels, and of different kinds. The first is a
Hilbert-space structure on the state space: the spectral
decomposition of the absorbed generator, weighted by the reversible
measure, in which the modes and the duality of the earlier papers
reside. The second is a Riemannian structure on the parameter
space: the simplex carries the Fisher--Rao metric, the canonical
Riemannian structure on a space of probability distributions, and
it is here that the separatrix and the response geometry live. The
two are not the same object, and there is no global identification
between them; but they meet, and what joins them is a square root.
Under the embedding $\pi\mapsto2\sqrt\pi$, which realises the
Fisher--Rao geometry isometrically as the round geometry of a
spherical octant, the interior of the separatrix becomes a totally
geodesic subsphere. Both constructions are governed by the same square-root
transformation: in the canonical realisation, the square root of
the reversible measure that symmetrises the generator on the state
space is also, once
normalised, a reset-neutral distribution --- the distinguished one
at which every orbit carries equal weight in the symmetrised
coordinates.

The paper makes three contributions, each resting on the one before.
First, it reduces the entire reset response to a finite structure:
the spectral duality forces the mode-by-mode constraints to decouple
along the orbits of an involution, and at the distinguished
invariant value $C^\ast=1/(1+\sqrt K)$ each orbit collapses
algebraically to a single direction, so that the whole response
span has dimension equal to the number of active orbits. Second, it identifies
Fisher--Rao as the geometry intrinsic to the problem --- not imposed
for convenience but selected, by \v{C}encov's theorem, as the only
metric compatible with the statistical structure --- and shows that
the separatrix is cut out by linear constraints, is totally
geodesic in this metric, and is itself a Fisher--Rao simplex of
lower dimension. Third, it isolates
the invariant $r$, the number of active orbits, as the integer that
governs the global response: when $r=1$ the response direction
cannot rotate, and under a scalar sign condition across the reset
sites --- which the canonical realisation satisfies --- the sign of
the response is fixed once and for all by the side of the
separatrix, whereas for $r\ge2$ the response span has dimension at
least two, the orientation cannot stay constant, and the simple
two-zone picture gives way to a higher-dimensional one. The biased random walk realises all of this
concretely, and provides the figures and the closed-form checks
throughout.

The paper is organised as follows. Section~\ref{sec:inherited}
sets out the abstract framework and the four structural conditions.
Section~\ref{sec:fisher-rao} introduces the Fisher--Rao geometry and
the square-root embedding. Section~\ref{sec:finite-reduction}
establishes the finite reduction and the invariant $r$.
Section~\ref{sec:geometry} develops the geometry of the critical set
and proves the separatrix rigid and totally geodesic.
Section~\ref{sec:sign-theorem} proves the sign theorem at the
starting sites, valid for every $r$, recovers the two zones of
Paper~I as a corollary, and settles the global orientation
principle: a sign law for $r=1$, and, for $r\ge2$, a criterion for
its failure together with the counterexamples that meet it.
Section~\ref{sec:illustrations} illustrates the whole picture on the
biased random walk. Section~\ref{sec:beyond} indicates the
directions --- combinatorial, projective, and representational ---
in which the construction continues. An appendix collects the
spectral realisation and its connection to the resolvent methods of
Paper~IV \cite{PaperIV}.

\section{The structural framework}
\label{sec:inherited}

This section collects the objects and results from
Papers~I--III \cite{PaperI,PaperII,PaperIII} on which the present
work is built. Everything in this section is pre-metric in the
intrinsic sense: no inner product, no notion of distance or angle,
is imposed on the simplex. \textup{(}An auxiliary inner product on
a representation space appears in Remark~\ref{rem:bridge}; it is a
device for writing the coupling, not a geometry on
$\Delta_{m-1}$.\textup{)} The framework
consists of a functional on a simplex, four structural conditions,
an invariant scalar, and a distinguished family of distributions.
The geometric structure that is the subject of this paper enters
only in Section~\ref{sec:fisher-rao}.

\subsection{Reset sites and the coupling functional}
\label{subsec:coupling}

Let $m\ge 2$ and let $\{z_1,\dots,z_m\}$ be a finite set of
distinguished \emph{reset sites}. Probability distributions supported
on these sites form the simplex
\[
\Delta_{m-1}=\Bigl\{\pi\in\mathbb R^m:\ \pi_i\ge 0,\
\textstyle\sum_{i=1}^m\pi_i=1\Bigr\},
\]
with relative interior $\Delta_{m-1}^\circ$ (all $\pi_i>0$). The
parameter $\gamma\in(0,1)$ is the resetting rate. The object of
study is a smooth \emph{coupling functional}
\[
C:\Delta_{m-1}\times(0,1)\longrightarrow\mathbb R,
\]
which in concrete models is a coupling constant associated with
absorption probabilities under resetting; in the canonical
realisation (the biased random walk of Papers~I--II, revisited in
Appendix~\ref{app:spectral}) it takes the form
$C(\pi,\gamma)=\bar u_\pi(\gamma)/\bar s_\pi(\gamma)$.

Following Paper~III \cite{PaperIII}, we do not specify the underlying Markov
process; instead we assume four structural conditions on $C$ and
its spectral data --- conditions which, in this paper, are promoted
from working hypotheses to the defining axioms of an abstract
structure (Definition~\ref{def:sds} below).

\begin{itemize}
\item[\textbf{(S1)}] \emph{Regularity.} For each fixed
$\gamma\in(0,1)$, the map $\pi\mapsto C(\pi,\gamma)$ is $\mathcal
C^1$ on $\Delta_{m-1}^\circ$ and extends continuously to the
boundary.
\item[\textbf{(S2)}] \emph{Spectral representability.} There exist
functions $u(z;\gamma)$, $s(z;\gamma)$ with
\[
C(\pi,\gamma)=\frac{\sum_i\pi_i\,u(z_i;\gamma)}
{\sum_i\pi_i\,s(z_i;\gamma)},
\]
the denominator strictly positive on $\Delta_{m-1}\times(0,1)$,
admitting finite spectral decompositions
\[
u(z;\gamma)=\sum_{\nu=1}^N f_\nu(\gamma)A_\nu(z),\qquad
s(z;\gamma)=\sum_{\nu=1}^N f_\nu(\gamma)\bigl(A_\nu(z)+B_\nu(z)\bigr),
\]
where the $\{f_\nu\}_{\nu=1}^N$ are real-analytic and linearly
independent on $(0,1)$ and the $m$ vectors
$\{(A_\nu(z_i))_{\nu=1}^N\}_{i=1}^m$ are linearly independent in
$\mathbb R^N$; equivalently, the matrix
$(A_\nu(z_i))\in\mathbb R^{N\times m}$ has full column rank $m\le N$.
Thus all dependence on the reset parameter is carried by the scalar
functions $f_\nu(\gamma)$, while the spatial information is encoded
in the fixed coefficients $A_\nu,B_\nu$; this separation of
variables is used constantly in what follows. The full column rank
guarantees that no reset site is spectrally redundant. This
non-degeneracy is one of the ingredients --- together with the
duality \textup{(S3)}--\textup{(S4)} --- from which the dimension of
the response span is computed in
Section~\ref{sec:finite-reduction}. The analyticity of the $f_\nu$
costs nothing in practice: in any resolvent realisation they are
rational functions of $\gamma$ with non-vanishing denominator, as
for the biased walk, where
$f_\nu(\gamma)=(1-\gamma)/(1-\lambda_\nu(1-\gamma))$ with
$|\lambda_\nu|<1$. It is used only to rule out degeneracies
concentrated on a set of resetting rates without interior
(Lemma~\ref{lem:rigidity}).
\item[\textbf{(S3)}] \emph{Spectral duality.} There exist an
involution $\sigma$ on the reset sites ($\sigma^2=\mathrm{id}$,
with at least one non-trivial pair) and positive weights
$\kappa(z)>0$, \emph{independent of the mode index $\nu$}, such
that
\[
B_\nu(z)=\kappa(z)\,A_\nu(\sigma(z))\qquad\text{for all }\nu,z.
\]
\item[\textbf{(S4)}] \emph{Compatibility.} There exists $K>0$,
the same for all sites, with
\[
\kappa(z)\,\kappa(\sigma(z))=K\qquad\text{for all }z.
\]
\end{itemize}

\begin{definition}[Spectral duality structure]
\label{def:sds}
A \emph{spectral duality structure} on the sites
$\{z_1,\dots,z_m\}$ is a pair $(\Delta_{m-1},C)$ --- the simplex
of distributions together with a coupling functional --- admitting
a \emph{spectral representation}: data
$\bigl(\sigma,\kappa,\{f_\nu\},\{A_\nu\}\bigr)$ satisfying
\textup{(S1)}--\textup{(S4)}, the dual coefficients being not
independent data but determined by \textup{(S3)} through
$B_\nu(z)=\kappa(z)A_\nu(\sigma(z))$.
\end{definition}

The definition places the pair $(\Delta_{m-1},C)$ first and the
spectral data second, and the proposition below shows that this is
the correct hierarchy: the duality data $(\sigma,\kappa)$ are not
an additional choice but invariants of the pair. The following lemma isolates
the one degeneracy that has to be excluded first --- the resetting
rates at which the coupling is blind to $\pi$ altogether.

\begin{lemma}[Non-degenerate rates]
\label{lem:rigidity}
Let $(\Delta_{m-1},C)$ admit a spectral representation and set
\[
G:=\{\gamma\in(0,1):\ C(\cdot,\gamma)\ \text{is not constant on}\
\Delta_{m-1}\}.
\]
Then $G$ is open and dense in $(0,1)$, its complement being
discrete; and for $\gamma\in G$ any two spectral representations of
$C$ satisfy
\[
u'(\cdot;\gamma)=\lambda(\gamma)\,u(\cdot;\gamma),\qquad
s'(\cdot;\gamma)=\lambda(\gamma)\,s(\cdot;\gamma),
\qquad\lambda(\gamma)>0 .
\]
\end{lemma}

\begin{proof}
We first record, since it is used twice here and again below, that
the functions $\{U(z,\cdot)\}_z$, $U(z,\cdot):=u(z;\cdot)$, are
linearly independent: a vanishing combination
$\sum_z c_zU(z,\cdot)\equiv0$ expands as
$\sum_\nu f_\nu(\gamma)\sum_z c_zA_\nu(z)=0$, and the independence
of the $f_\nu$ followed by the full column rank of
$(A_\nu(z_i))$ in \textup{(S2)} forces $c\equiv0$. In particular no
two distinct $U(z,\cdot)$ are proportional.

Taking $\pi=\delta_z$ in the positivity of \textup{(S2)} gives
$s(z;\gamma)>0$ at every site, so $s(\cdot;\gamma)$ has strictly
positive entries. Consequently $u(\cdot;\gamma)\propto
s(\cdot;\gamma)$ if and only if $C(\cdot,\gamma)$ is constant: if
$u=cs$ then $C\equiv c$, and conversely $C\equiv c$ gives
$\langle\pi,u-cs\rangle=0$ for every $\pi\in\Delta_{m-1}$, hence
$u=cs$. Thus $G$ is exactly the set of $\gamma$ at which
$u(\cdot;\gamma)$ and $s(\cdot;\gamma)$ are non-proportional.
Moreover $C(\cdot,\gamma)$ is constant on $\Delta_{m-1}$ if and only
if its values at the vertices coincide: if
$u(z;\gamma)/s(z;\gamma)=c$ for every $z$ then $u=cs$ and
$C\equiv c$, and the converse is immediate. Hence
\[
G=\bigl\{\gamma:\ \exists\,z,z'\ \text{with}\
u(z;\gamma)s(z';\gamma)\neq u(z';\gamma)s(z;\gamma)\bigr\},
\]
a union of finitely many sets on which a continuous function of
$\gamma$ is non-zero, and therefore open.

Fix $\gamma\in G$. The identity
$\langle\pi,u\rangle\langle\pi,s'\rangle
=\langle\pi,u'\rangle\langle\pi,s\rangle$ on the open simplex
extends to $\mathbb R^m$ by homogeneity, and unique factorization
of products of linear forms leaves two pairings; the one giving
$u\propto s$ is excluded because $\gamma\in G$. Hence
$u'=\alpha u$ and $s'=\beta s$, and $C=C'$ forces
$\alpha=\beta=:\lambda(\gamma)$, positive since $s,s'>0$.

Finally, the complement of $G$ is discrete. For
$\gamma\notin G$ one has $u(\cdot;\gamma)=c(\gamma)s(\cdot;\gamma)$,
whence $B=s-u=(1-c)s$, while \textup{(S3)} gives
$B(z,\cdot)=\kappa(z)u(\sigma(z),\cdot)
=c\,\kappa(z)s(\sigma(z),\cdot)$. Comparing these at $z$ and at
$\sigma(z)$,
\[
(1-c)\,s(z,\gamma)=c\,\kappa(z)\,s(\sigma(z),\gamma),\qquad
(1-c)\,s(\sigma(z),\gamma)=c\,\kappa(\sigma(z))\,s(z,\gamma),
\]
and multiplying, the strictly positive factor
$s(z,\gamma)s(\sigma(z),\gamma)$ cancels and
$\kappa(z)\kappa(\sigma(z))=K$ leaves the quadratic identity
\[
(1-c)^2=c^2K ,
\]
obtained without dividing by $c$. \textup{(}In fact $c\neq0$: were
$u(\cdot;\gamma)=0$ then $B=0$ by \textup{(S3)} and hence $s=0$,
contradicting \textup{(S2)}.\textup{)} Its solutions are
$c(1\pm\sqrt K)=1$: the branch with $+$ always gives
$c=1/(1+\sqrt K)$, while the branch with $-$ gives
$c=1/(1-\sqrt K)$ when $K\neq1$ and is impossible when $K=1$, where
it reads $0=1$. So $c(\gamma)$ takes at most two values.
\textup{(}The second is retained where it exists: \textup{(S2)}
constrains only the denominator, so $C$ need not be positive in the
abstract setting and that branch cannot be discarded on sign
grounds.\textup{)} For each such fixed value $c$
the map
$\gamma\mapsto u(z;\gamma)-c\,s(z;\gamma)$ is real-analytic, so
its zero set is either discrete or all of $(0,1)$. In the latter
case $u\equiv c\,s$, whence
$U(\sigma(z),\cdot)=\bigl((1-c)/(c\,\kappa(z))\bigr)
U(z,\cdot)$ as functions, making $U(\sigma(z),\cdot)$ and $U(z,\cdot)$ proportional on any
non-trivial pair, which the independence recorded above excludes.
Each
of the two zero sets is therefore discrete, and so is their
union $(0,1)\setminus G$.
\end{proof}

\begin{proposition}[The duality data are invariants]
\label{prop:invariants}
The involution $\sigma$ and the weights $\kappa$ are uniquely
determined by the coupling functional: any two spectral
representations of the same pair $(\Delta_{m-1},C)$ share the same
$(\sigma,\kappa)$. Consequently every
object built from $(\sigma,\kappa)$ --- orbits, the constant $K$,
the invariant value $C^\ast$, the separatrix, the response span
--- is an invariant of the pair.
\end{proposition}

\begin{proof}
Write $B(z,\cdot)=s(z;\cdot)-u(z;\cdot)$, and recall from the proof
of Lemma~\ref{lem:rigidity} that the family $\{U(z,\cdot)\}_z$ is
linearly independent.

By Lemma~\ref{lem:rigidity} there is a dense open
$G\subseteq(0,1)$ such that, for each $\gamma\in G$,
$u'(\cdot;\gamma)=\lambda(\gamma)u(\cdot;\gamma)$ and
$s'(\cdot;\gamma)=\lambda(\gamma)s(\cdot;\gamma)$ with
$\lambda(\gamma)>0$. The proportionality holds at each such
$\gamma$ separately, but that is all we need: subtracting the two
identities gives $B'(z,\gamma)=\lambda(\gamma)B(z,\gamma)$ and
$U'(z,\gamma)=\lambda(\gamma)U(z,\gamma)$ for every site and every
$\gamma\in G$, so the ratio $B'(z,\gamma)/U'(w,\gamma)
=B(z,\gamma)/U(w,\gamma)$ is free of $\lambda$ throughout $G$.
Now \textup{(S3)}, applied to each representation, reads
\[
B(z,\cdot)=\kappa(z)\,U(\sigma(z),\cdot),\qquad
B'(z,\cdot)=\kappa'(z)\,U'(\sigma'(z),\cdot)
\]
as identities of functions on $(0,1)$. Dividing the second by
$\lambda$ on $G$ and comparing with the first,
\[
\kappa(z)\,U(\sigma(z),\cdot)=\kappa'(z)\,U(\sigma'(z),\cdot)
\qquad\text{on } G,
\]
and both sides are real-analytic, so the identity holds on all of
$(0,1)$. Since $\kappa(z),\kappa'(z)>0$, this makes
$U(\sigma(z),\cdot)$ and $U(\sigma'(z),\cdot)$ proportional; but no
two distinct members of a linearly independent family are
proportional, so $\sigma'(z)=\sigma(z)$ for every $z$, and then
$\kappa'(z)=\kappa(z)$. The involution and the
weights are therefore the same for both representations.
\end{proof}

Throughout the paper we work with a fixed spectral duality
structure, described through some spectral representation; by
Proposition~\ref{prop:invariants}, nothing depends on this choice.
Except where a realisation is invoked explicitly --- as in the
vertex condition \textup{(V)} of Section~\ref{sec:sign-theorem} and
in Appendix~\ref{app:spectral} --- the arguments below use only the
axioms, and no property of an underlying stochastic process.
The canonical realisation --- the biased random walk with geometric
resetting of Papers~I--II --- is described in
Appendix~\ref{app:spectral}, and every statement below applies to
it as a particular instance.

\begin{remark}[Terminology]
\label{rem:terminology}
We retain throughout the vocabulary of the problem in which the
structure was discovered --- \emph{reset sites}, \emph{reset
parameter}, \emph{reset landscape} --- as names for the abstract
objects, in the same way that Riemannian geometry retains
\emph{geodesic} from geodesy: no probabilistic content is implied
by these words, and none is used. When a statement does require an
underlying process, the hypothesis is stated explicitly (as in
Remark~\ref{rem:pi-mu} below).
\end{remark}

An \emph{orbit} of $\sigma$ is a pair $\{z,\sigma(z)\}$ with
$\sigma(z)\neq z$; a \emph{neutral site} is a fixed point
$z_0=\sigma(z_0)$, for which (S4) forces $\kappa(z_0)=\sqrt K$. We
write $n_{\mathrm{pair}}$ for the number of orbits and $n_0$ for
the number of neutral sites, so that $m=2n_{\mathrm{pair}}+n_0$.
Throughout the paper, the orbits of $\sigma$ are the fundamental
structural units, and every object introduced below is organised
orbitwise: the compatibility \textup{(S4)} is one constraint per
orbit; \textup{(S4)} assigns the same product $K$ to every orbit,
and combined with the neutrality condition this makes the invariant
scalar $C^\ast=1/(1+\sqrt K)$ the same whichever orbit it is computed
from;
the ratios characterising the separatrix are fixed orbit by orbit;
the response dimension $r$ counts the orbits
(Section~\ref{sec:finite-reduction}); and the Fisher--Rao geometry
of Sections~\ref{sec:fisher-rao}--\ref{sec:geometry} will
decompose orbitwise as well. Although introduced historically in a
different order, these objects organise themselves naturally
around the orbit structure induced by the involution.

\subsection{The invariant scalar and the separatrix}
\label{subsec:invariant}

A distribution $\pi^\ast\in\Delta_{m-1}$ is \emph{reset-neutral} if
$C(\pi^\ast,\gamma)$ is independent of $\gamma$. The set of
reset-neutral distributions,
\[
\Sigma:=\{\pi\in\Delta_{m-1}:\ C(\pi,\gamma)\ \text{is constant in}\
\gamma\},
\]
is the \emph{separatrix}. The central existence result of
Paper~III is:

\begin{theorem}[{Paper~III, Theorem~3.3}]
\label{thm:paper3}
Under \textup{(S1)--(S4)} the separatrix is non-empty, and every
$\pi^\ast\in\Sigma$ satisfies $C(\pi^\ast,\gamma)=C^\ast$ for all $\gamma$,
with
\[
C^\ast=\frac{1}{1+\sqrt K}.
\]
Moreover:
\textup{(i)} on every orbit carrying positive mass,
$\pi^\ast_z/\pi^\ast_{\sigma(z)}=\sqrt{\kappa(\sigma(z))/\kappa(z)}$ ---
in particular throughout $\Sigma^\circ$;
\textup{(ii)} the weight of each neutral site is free, subject
only to normalization; \textup{(iii)} $C^\ast$ depends only on $K$.
\end{theorem}

Two features of Theorem~\ref{thm:paper3} organise everything that
follows. First, the invariant value $C^\ast$ is a function of the
compatibility constant $K$ alone: it is a structural scalar, not a
model-dependent quantity. Second, the characterisation \textup{(i)}
fixes only the \emph{ratio} of $\pi^\ast$ within each orbit; the
relative weight assigned to different orbits (and to neutral sites)
is unconstrained. The interior part of the separatrix is therefore a family of
positive dimension, not a single point:
\[
\dim\bigl(\Sigma^\circ\bigr)
=n_{\mathrm{pair}}+n_0-1
\]
\textup{(}Paper~III, Proposition~5.1\textup{)}. This is the
principal stratum; $\Sigma$ also meets the boundary, in strata of
lower dimension where some orbits carry no mass, and the three
statements $\Sigma\neq\emptyset$,
$\Sigma^\circ\neq\emptyset$ and
$\dim(\Sigma^\circ)=n_{\mathrm{pair}}+n_0-1$ are
kept distinct throughout. Both features admit
a natural geometric interpretation, developed in
Sections~\ref{sec:fisher-rao} and~\ref{sec:geometry}: the scalar
$C^\ast$ is the distinguished level whose level set is carried to a
totally geodesic submanifold of a round sphere, while the orbitwise
freedom of $\pi^\ast$ is precisely the intrinsic dimension of that
submanifold.

Paper~III established the sufficiency of \textup{(S4)}. The
proposition below shows that, within the present abstract
framework, compatibility is also \emph{necessary}: it is exactly
the global synchronisation condition for an interior reset-neutral
distribution. The compatibility axiom deserves this closer look: it is not an
assumption of convenience but the exact synchronisation condition
for the phenomenon.

\begin{proposition}[Sharpness of the compatibility axiom]
\label{prop:sharp}
Assume \textup{(S1)--(S3)}. Then the interior separatrix is
non-empty, $\Sigma^\circ\neq\emptyset$, if and
only if \textup{(S4)} holds. Moreover the constant and the
invariant value are forced: if $\pi\in\Sigma^\circ$
has constant value $C$, then
$\kappa(z)\kappa(\sigma(z))=\bigl((1-C)/C\bigr)^2$ for every $z$,
i.e.\ $K=\bigl((1-C)/C\bigr)^2$ and $C=1/(1+\sqrt K)$.
\end{proposition}

\begin{proof}
Sufficiency is Theorem~\ref{thm:paper3}. For necessity, let
$\pi\in\Sigma^\circ$ with $C(\pi,\gamma)\equiv C$.
Positivity of the denominator \textup{(S2)} gives
$\langle\pi,u(\cdot;\gamma)-C\,s(\cdot;\gamma)\rangle=0$ for all
$\gamma$; expanding by \textup{(S2)}--\textup{(S3)}, using the
linear independence of the $f_\nu$ and reindexing the dual term by
$z\mapsto\sigma(z)$,
\[
\sum_{z}A_\nu(z)\Bigl[(1-C)\,\pi_z
-C\,\kappa(\sigma(z))\,\pi_{\sigma(z)}\Bigr]=0
\qquad\text{for all }\nu.
\]
Full column rank \textup{(S2)} forces the bracket to vanish at
every site. Evaluating it at $z$ and at $\sigma(z)$ and multiplying
the two identities, $\pi_z\pi_{\sigma(z)}>0$ yields
$\kappa(z)\kappa(\sigma(z))=((1-C)/C)^2$ for every $z$, which is
\textup{(S4)}; at a neutral site the vanishing bracket gives
directly $\kappa(z_0)=(1-C)/C=\sqrt K$. Without \textup{(S4)},
reset-neutral distributions may still exist, but only on proper
faces: on any
collection of orbits whose products $\kappa(z)\kappa(\sigma(z))$
share a common value $K_0$, the orbitwise critical ratios define a
distribution that is reset-neutral with value
$1/(1+\sqrt{K_0})$, supported on that collection. An \emph{interior} reset-neutral distribution --- one lying in
$\Delta^\circ_{m-1}$ --- therefore exists exactly when all orbit
products agree, which is precisely \textup{(S4)}: compatibility is
the global synchronisation of the orbitwise critical values.
\end{proof}

\begin{remark}[Logical structure of the axioms]
\label{rem:axioms}
The system \textup{(S1)--(S4)} is a tower rather than a flat list:
each axiom is formulated in the vocabulary of the previous ones.
Within it, \textup{(S1)} is in fact implied by \textup{(S2)} --- a
ratio of linear forms with strictly positive denominator extends
real-analytically to an open neighbourhood of the closed simplex,
the denominator remaining non-zero there --- and is retained only to
delimit the ambient class of functionals. None of the three structural
axioms is redundant within the tower: representability without duality
(\textup{(S2)}$\,\wedge\,\neg$\textup{(S3)}: choose the $B_\nu$
with $\nu$-dependent ratios) and duality without compatibility
(\textup{(S2)}--\textup{(S3)}$\,\wedge\,\neg$\textup{(S4)}:
choose $\kappa$ with orbit-dependent products) are both trivially
realisable. Proposition~\ref{prop:sharp} completes the picture:
given \textup{(S1)--(S3)}, the compatibility of the weights is
\emph{equivalent} to the existence of interior reset-neutral
distributions.
\end{remark}

\begin{remark}[The coupling as a derived observable]
\label{rem:bridge}
The axioms admit an alternative, representation-theoretic reading
in which the object placed first is not the functional $C$ but a
representation of the simplex. It is a useful reading rather than a
canonical equivalence: Proposition~\ref{prop:invariants} makes
$(\sigma,\kappa)$ invariants of the pair, but does not make the
representation below unique. No
new structure is needed for it: let $\mathcal H=\mathbb R^N$ carry
its standard inner product and orthonormal basis
$\{e_\nu\}_{\nu=1}^N$, send each site to its column of spectral
coefficients, $\delta_z\mapsto\sum_\nu A_\nu(z)e_\nu$, and extend
by linearity from $\mathbb R^m$. Restricted to the simplex this
gives the affine embedding
\[
\mathcal R:\Delta_{m-1}\longrightarrow\mathcal H,\qquad
\mathcal R(\pi)
=\sum_\nu\Bigl(\sum_z\pi_z A_\nu(z)\Bigr)e_\nu ,
\]
injective by the full column rank of \textup{(S2)}. Write
$f(\gamma)=(f_1(\gamma),\dots,f_N(\gamma))\in\mathcal H$ for the
vector of spectral scalars and
$b(\pi)=\sum_z\pi_z\bigl(B_1(z),\dots,B_N(z)\bigr)\in\mathcal H$
for the dual contribution. Then
$\sum_z\pi_zu(z;\gamma)=\langle\mathcal R(\pi),f(\gamma)\rangle$
and
$\sum_z\pi_zs(z;\gamma)
=\langle\mathcal R(\pi)+b(\pi),f(\gamma)\rangle$, so that
\[
C(\pi,\gamma)
=\frac{\langle \mathcal R(\pi),\,f(\gamma)\rangle}
      {\langle \mathcal R(\pi)+b(\pi),\,f(\gamma)\rangle}
\]
--- a form in which the duality appears as the shift
$\mathcal R(\pi)\mapsto\mathcal R(\pi)+b(\pi)$ ---
is a $\gamma$-family of projective ratios of linear functionals on
$\mathcal R(\Delta_{m-1})$: the
coupling is an observable \emph{induced} by the representation. In
this reading, \textup{(S2)} states that the representation is
spectrally nondegenerate, with positive pairing against $s$;
\textup{(S3)} that it is compatible with an involutive duality; and
\textup{(S4)} that the duality carries a constant compatibility
factor. The regularity \textup{(S1)} is then automatic --- $\mathcal
R$ is linear and each projection is a ratio of linear forms with
nonvanishing denominator --- in agreement with
Remark~\ref{rem:axioms}. The simplex thus carries two natural maps
into a common representation space: the affine embedding
$\mathcal R$, through which the algebra acts, and the square-root
map of Section~\ref{sec:fisher-rao}, through which the Fisher--Rao
geometry is measured; the results of this paper describe how structures
transported by the former become geodesic objects under the latter.
A systematic development of this viewpoint is deferred to the
sequel.
\end{remark}

\begin{remark}[Critical distributions and the reversible measure]
\label{rem:pi-mu}
The axioms \textup{(S1)--(S4)} make no reference to any measure,
and not every spectral duality structure arises from a Markov
chain: in the abstract setting, the weights $\kappa$ are the
primitive object. When the structure \emph{is} realised by a chain
in the $(\sigma,\kappa)$-reversible class of Paper~IV --- as the
canonical realisation is --- and is reversible with respect to a
measure $\mu$, the latter is a derived object: detailed balance
combined with $(\sigma,\kappa)$-reversibility gives
$\mu(\sigma(z))\propto\mu(z)\kappa(z)^2$ (Paper~IV), equivalently,
using \textup{(S4)}, the product $\mu\kappa$ is $\sigma$-invariant,
so that $\mu(\sigma(z))/\mu(z)=\kappa(z)/\kappa(\sigma(z))$ on every
orbit. \textup{(}For the biased walk, $\mu(z)\propto(p/q)^z
\propto\kappa(z)^{-1}$ globally; in general $\mu\kappa$ need only be
constant on each orbit.\textup{)} Combining this with the ratio
constraint \textup{(i)} of Theorem~\ref{thm:paper3} gives
\[
\pi^\ast_z\propto\sqrt{\mu(z)}\qquad\text{on each orbit of }\sigma:
\]
within each orbit, the critical distributions are proportional to
the square root of the reversible measure. The exponent $1/2$ is
not incidental: it is dictated by the critical ratios
$\sqrt{\kappa(\sigma(z))/\kappa(z)}$ themselves. In the reversible realisation, therefore, the critical
ratios may be viewed not merely as the solution of the neutrality
equations but as the image of the reversible measure under a fixed
transformation, $\pi^\ast=\mathcal F(\mu)$ orbitwise with
$\mathcal F(x)=\sqrt x$ --- a reading available in that class only,
since the abstract axioms refer to no measure at all.
This is a first indication of the interplay between the spectral
weighting of Papers~I--II and the geometric structure developed
from Section~\ref{sec:fisher-rao} onwards; the same square root
will reappear there as the isometry $\Phi$.
\end{remark}

\subsection{The orientation field}
\label{subsec:orientation}

Fix a reference point $\pi^\ast\in\Sigma^\circ$. The
\emph{orientation field} is the family of vectors
$\psi(\gamma)\in\mathbb R^m$ with components
\begin{equation}
\psi_z(\gamma):=u(z;\gamma)-C^\ast\,s(z;\gamma),
\qquad z\in\{z_1,\dots,z_m\},
\label{eq:psi-def}
\end{equation}
and the associated \emph{orientation functional}
\begin{equation}
L_\gamma(\pi):=\langle\pi-\pi^\ast,\psi(\gamma)\rangle
=\sum_{i=1}^m(\pi_i-\pi^\ast_i)\,\psi_{z_i}(\gamma),
\label{eq:L-def}
\end{equation}
where $\langle\cdot,\cdot\rangle$ denotes the standard pairing of
$\mathbb R^m$ (no metric structure is invoked; $\psi(\gamma)$ is a
covector).

Since $\pi^\ast\in\Sigma$ and $C(\pi^\ast,\gamma)=C^\ast$, the numerator of
$C(\pi^\ast,\gamma)-C^\ast$ vanishes, that is
$\langle\pi^\ast,\psi(\gamma)\rangle=0$, and therefore
\[
L_\gamma(\pi)=\langle\pi,\psi(\gamma)\rangle .
\]
Moreover, for every $\pi\in\Delta_{m-1}$,
\begin{equation}
C(\pi,\gamma)-C^\ast
=\frac{\langle\pi,\,u(\cdot;\gamma)-C^\ast s(\cdot;\gamma)\rangle}
      {\langle\pi,s(\cdot;\gamma)\rangle}
=\frac{L_\gamma(\pi)}{\langle\pi,s(\cdot;\gamma)\rangle}.
\label{eq:C-minus-Cstar}
\end{equation}
The denominator is strictly positive by \textup{(S2)}, so
$L_\gamma$ and $C-C^\ast$ carry the same sign everywhere. This
identity is used repeatedly below: it is what makes the
orientation functional an exact record of the deviation of the
coupling from its invariant value.

\begin{remark}[Relation to the linear response functional of
Paper~III]
\label{rem:normalisation}
Paper~III (Definition~5.2) works with the differential of
$C(\cdot,\gamma)$ at $\pi^\ast$, whose components are
$\bar s_{\pi^\ast}(\gamma)^{-1}\,\psi_{z_i}(\gamma)$. The two objects
differ by the scalar factor $\bar s_{\pi^\ast}(\gamma)^{-1}$, which is
strictly positive by \textup{(S2)}. Consequently they have the same
kernels, span the same subspace as $\gamma$ varies, and
$\operatorname{sgn}L_\gamma(\pi)$ is the same for both. We adopt
the unnormalised version \eqref{eq:psi-def} for algebraic
transparency: all statements below about spans, kernels and signs
are insensitive to the choice.
\end{remark}

The family $\{\psi(\gamma)\}_{\gamma\in(0,1)}$ concentrates the
entire dependence of the reset response on the parameter $\gamma$;
understanding its structure --- the task of
Section~\ref{sec:finite-reduction} --- is therefore tantamount to
understanding the reset landscape itself.

Following Paper~III (Definition~5.3), define the \emph{response
span} and its dimension
\[
V:=\operatorname{span}\{\psi(\gamma):\gamma\in(0,1)\}
\subseteq\mathbb R^m,
\qquad r:=\dim V.
\]
Paper~III (Theorem~5.4) shows that the kernels
$\bigl\{\ker\bigl(\psi(\gamma)|_{T_{\pi^\ast}\Delta_{m-1}}\bigr)\bigr\}$
form a fan of hyperplanes in
$T_{\pi^\ast}\Delta_{m-1}$ whose common intersection --- the local
separatrix $\Sigma_{\mathrm{loc}}$, the directions along which $C$
is insensitive to $\gamma$ at first order --- is exactly the
annihilator of $V$, of dimension $(m-1)-r$. The dimension $r$ is
thus the fundamental finite-dimensional complexity parameter of
the response landscape.
Its exact value under \textup{(S2)--(S4)} is not determined there;
Section~\ref{sec:finite-reduction} \emph{will show} that, under
\textup{(S2)--(S4)}, compatibility collapses each non-trivial orbit
to a single response direction, that the directions associated with
distinct orbits are independent, and that every non-trivial orbit is
active. Consequently $r=n_{\mathrm{pair}}$, neutral sites
contributing none. The parameter $r$ controls the codimension of
the local separatrix and, with it, the complexity of the landscape:
for $r=1$ a single linear functional governs the response, while
for $r>1$ several functionals compete, producing the transition
region analysed in Section~\ref{sec:beyond}.

\subsection{The Global Orientation Conjecture}
\label{subsec:conjecture}

The local theory of Paper~III leaves open the global question:
does the sign of the response $\partial_\gamma C$ away from
$\pi^\ast$ remain governed by the linear functional $L_\gamma$? This
is the content of:

\begin{conjecture}[Global orientation principle]
\label{conj:global}
Under \textup{(S1)--(S4)},
\[
\operatorname{sgn}\bigl(\partial_\gamma C(\pi,\gamma)\bigr)
=\operatorname{sgn}\,L_\gamma(\pi)
\qquad\text{for all }\pi\in\Delta_{m-1}^\circ,\
\gamma\in(0,1).
\]
By \eqref{eq:C-minus-Cstar}, $\operatorname{sgn}L_\gamma(\pi)
=\operatorname{sgn}\bigl(C(\pi,\gamma)-C^\ast\bigr)$ identically,
so the conjectured identity reads
$\operatorname{sgn}\partial_\gamma C
=\operatorname{sgn}(C-C^\ast)$ whence, wherever $C(\pi,\gamma)\neq C^\ast$,
\[
\frac{\partial}{\partial\gamma}\bigl|C(\pi,\gamma)-C^\ast\bigr|
=\operatorname{sgn}\bigl(C(\pi,\gamma)-C^\ast\bigr)\,
\partial_\gamma C(\pi,\gamma)>0 .
\]
Thus $|C(\pi,\cdot)-C^\ast|$ is non-decreasing on $(0,1)$ and
strictly increasing on any interval on which it does not vanish.
\textup{(}The qualification matters: $\pi\notin\Sigma$ only says
that $C(\pi,\cdot)$ is non-constant, and $L_\gamma(\pi)$ may still
vanish at isolated rates.\textup{)} The response would then be
monotone in $\gamma$ for every interior $\pi$. The geometric reading as a two-sided
boundary, however, requires care, because three distinct objects
are in play. The instantaneous zero set at a given $\gamma$ is the
hyperplane
$\ker\bigl(\psi(\gamma)|_{T_{\pi^\ast}\Delta_{m-1}}\bigr)$, of
codimension one, which moves with $\gamma$. Its intersection over
all $\gamma$ is the local separatrix $\Sigma_{\mathrm{loc}}=V^\perp$,
of codimension $r$ in the tangent space. And $\Sigma$ itself is the
global object, with
$\dim(\Sigma^\circ)=n_{\mathrm{pair}}+n_0-1$, so
that it too has codimension $r$. The three coincide only when
$r=1$; and only then does a single hyperplane, fixed as $\gamma$
varies, separate the simplex into two sides. For $r\ge2$ each
instantaneous hyperplane strictly contains $\Sigma$, no one of them
is $\Sigma$, and a set of codimension $r\ge2$ does not separate at
all.
\end{conjecture}

By Remark~\ref{rem:normalisation} the statement is unchanged if
$\psi$ is replaced by the normalised response functional of
Paper~III. This paper settles the conjecture in its pointwise form
at the \emph{starting sites}: distributions supported on neutral
sites are exactly reset-neutral and, under the vertex condition
\textup{(V)} of Section~\ref{sec:sign-theorem} --- satisfied by
the canonical realisation --- the response at
every vertex of the simplex is strictly monotone in $\gamma$, with
the two members of each orbit responding in opposite directions
(Section~\ref{sec:sign-theorem}). Over the full simplex the behaviour is
governed by the invariant $r$ together with a sign condition, and
the two are independent; this is the content of the following
remark.

\begin{remark}[Rotation versus reversal]
\label{rem:rotation-reversal}
Two distinct things can prevent a $\gamma$-independent orientation,
and they should not be conflated.
\begin{itemize}
\item[\textup{(i)}] \emph{Rotation.} If $r\ge2$ the response field
is not confined to a line and the direction of $\psi(\gamma)$ may
turn as $\gamma$ varies. This is a genuine possibility for $r\ge2$,
not a necessity: rotation need not occur in every such structure.
\item[\textup{(ii)}] \emph{Reversal.} If $r=1$ no rotation is
possible, since all $\psi(\gamma)$ span the same line. But
$\psi(\gamma)$ may still change sign \emph{within} that line, which
reverses the orientation without moving the hyperplane
$\ker\bigl(\psi(\gamma)|_{T_{\pi^\ast}\Delta_{m-1}}\bigr)$. It is the
vertex condition \textup{(V)} of
Section~\ref{sec:sign-theorem} that excludes this.
\item[\textup{(iii)}] \emph{Miscalibration.} Even with $r=1$ and
\textup{(V)}, a third obstacle remains. Writing
$\psi(\gamma)=\varphi(\gamma)\mathbf n$ with $\mathbf n$ a fixed
director, \textup{(V)} forces the Wronskians
$\mathcal W[\varphi,s(z;\cdot)]$ to share a strict sign
$\varepsilon$, and hence forbids $\varphi$ from changing sign; but
it does not force that constant sign to \emph{agree} with
$\varepsilon$. The sign law proved in
Section~\ref{sec:sign-theorem} is referred to $\varepsilon$ and
$\mathbf n$, whereas Conjecture~\ref{conj:global} is referred to
$\psi(\gamma)$, and the two coincide precisely when
$\varepsilon=\operatorname{sgn}\varphi$. This calibration is not
implied by \textup{(S1)--(S4)}: that it can fail --- and hence that
it is a genuine third requirement rather than a consequence of the
first two --- was established by T.~Newton \cite{Newton}, who
exhibited two-site structures satisfying \textup{(S1)--(S4)} and
\textup{(V)} in which $\varepsilon$ and $\operatorname{sgn}\varphi$
are opposite, so that the identification of
Conjecture~\ref{conj:global} holds with the reversed sign
throughout the interior. In the equivalent formulation given there,
the calibration is the positivity of the product of the Wronskian
governing the rate response and the determinant governing the
distribution response.
\end{itemize}
The kernel of a covector is unchanged by multiplication by a
negative scalar, so the projective field $[\psi]$ carries strictly
less information than the oriented field $\psi$: the structural
axioms organise the projective data, whereas \textup{(V)} supplies
the additional sign information needed to orient it. Accordingly
$r=1$ alone does not suffice, and neither does $r=1$ together with
\textup{(V)}: within \textup{(S1)--(S4)} there are structures
satisfying both for which the identification of
Conjecture~\ref{conj:global} nevertheless fails, by
\textup{(iii)} above. What $r=1$ and \textup{(V)} do give is the
sign law of Section~\ref{sec:sign-theorem}, stated with respect to
$\varepsilon$ and $\mathbf n$; the identification as conjectured
requires in addition the calibration
$\varepsilon=\operatorname{sgn}\varphi$, which the canonical
realisation satisfies.
\end{remark} The geometry of the rotating case is
combinatorial and is deferred to the sequel
(Section~\ref{sec:beyond}).

\section{Fisher--Rao geometry and the spherical representation}
\label{sec:fisher-rao}

The framework of Section~\ref{sec:inherited} is purely algebraic: a
functional, a simplex, an involution. We now introduce the single
new structure of this paper --- a Riemannian metric on the simplex
of reset distributions --- and its canonical realisation as
elementary spherical geometry. The metric is not a modelling
choice: it is the canonical metric singled out by \v{C}encov's
theorem, and its canonical status makes it a natural intrinsic
metric for the reset simplex rather than an arbitrary modelling
choice.

\subsection{The Fisher--Rao metric}
\label{subsec:FR}

On the open simplex $\Delta_{m-1}^\circ$, the \emph{Fisher--Rao
metric} is the Riemannian metric
\begin{equation}
g^{\mathrm{FR}}_\pi(h,k)\;=\;\sum_{i=1}^m\frac{h_i k_i}{\pi_i},
\qquad h,k\in T_\pi\Delta_{m-1}
=\Bigl\{v\in\mathbb R^m:\textstyle\sum_i v_i=0\Bigr\}.
\label{eq:FR}
\end{equation}
At each point it is a weighted inner product on the tangent space,
with weights $1/\pi_i$ depending on the base point. Its canonical
status is classical: by \v{C}encov's theorem
\cite{Cencov,AmariNagaoka} (for the Riemannian background used
below --- geodesics on the sphere and totally geodesic submanifolds
--- see \cite{Sakai,Helgason}), $g^{\mathrm{FR}}$ is, up to a constant
factor, the canonical monotone metric on the simplex --- the one
invariant under Markov embeddings. Whatever geometry the reset
landscape possesses in this metric is therefore a property of the
statistical structure itself and not of a particular
parametrisation.

\begin{remark}[The weighted structure at the critical point]
\label{rem:two-weights}
At a critical distribution $\pi^\ast\in\Sigma$, the metric
\eqref{eq:FR} freezes into a fixed weighted inner product with
weights $1/\pi^\ast_i$ on the reset sites --- the direct analogue, on
the space of reset parameters, of the weighted inner product on
the state space that diagonalises the generator in Papers~I--II.
For structures in the reversible class of
Remark~\ref{rem:pi-mu}, the two weightings are intertwined
orbitwise through a square root: on each orbit,
$1/\pi^\ast_z\propto\mu(z)^{-1/2}$. The same square root reappears,
in a different role, in Proposition~\ref{prop:isometry} below, and
the coincidence is not accidental. It first emerges from the
spectral duality determining the critical distributions, and
reappears as the coordinate change that turns the Fisher--Rao
metric into the round metric of the sphere. The two maps are not
literally the same --- one symmetrises an operator on the state
space, the other embeds the parameter simplex in a sphere --- but
they are governed by the same square-root transformation, and in
the reversible class the connection is concrete rather than
analogical: the square root of the reversible measure determines
the same orbitwise ratios as the critical distributions, so that
after normalisation it selects a distinguished point of the
interior separatrix. \textup{(}The inter-orbit weights remain
free, consistently with
$\dim(\Sigma^\circ)=n_{\mathrm{pair}}+n_0-1$.\textup{)} Their interaction becomes concrete in
Section~\ref{sec:geometry}.
\end{remark}

\subsection{The spherical representation}
\label{subsec:sphere}

Define
\begin{equation}
\Phi:\Delta_{m-1}^\circ\longrightarrow\mathbb R^m,
\qquad
\Phi(\pi)=2\sqrt{\pi}:=\bigl(2\sqrt{\pi_1},\dots,2\sqrt{\pi_m}\bigr).
\label{eq:Phi}
\end{equation}

The same formula defines a continuous map
$\bar\Phi:\Delta_{m-1}\to S^{m-1}(2)$ on the closed simplex, onto the
closed positive octant; $\bar\Phi$ is used only when a boundary
stratum has to be named, and every metric statement below refers to
$\Phi$ on $\Delta^\circ_{m-1}$, where alone the Fisher--Rao metric
is defined.

\begin{proposition}[Spherical realisation]
\label{prop:isometry}
$\Phi$ is an isometry from $(\Delta_{m-1}^\circ,g^{\mathrm{FR}})$
onto the open positive octant of the round sphere of radius $2$,
\[
S^{m-1}_+(2)\;=\;\Bigl\{x\in\mathbb R^m:\ \textstyle\sum_i
x_i^2=4,\ x_i>0\Bigr\},
\]
equipped with the metric induced from the Euclidean inner product
of $\mathbb R^m$.
\end{proposition}

\begin{proof}
Since $\sum_i\Phi(\pi)_i^2=4\sum_i\pi_i=4$ and $\Phi(\pi)_i>0$ on
the open simplex, $\Phi$ maps into $S^{m-1}_+(2)$; it is a
bijection onto it with inverse $x\mapsto(x_i^2/4)_i$. Its
differential is $d\Phi_\pi(h)=h/\sqrt{\pi}$ componentwise, which
is tangent to the sphere at $\Phi(\pi)$ because
$\langle\Phi(\pi),d\Phi_\pi(h)\rangle=2\sum_ih_i=0$. Finally,
\[
\bigl\langle d\Phi_\pi(h),\,d\Phi_\pi(k)\bigr\rangle
=\sum_{i=1}^m\frac{h_ik_i}{\pi_i}
=g^{\mathrm{FR}}_\pi(h,k),
\]
so the Euclidean inner product pulls back exactly to
\eqref{eq:FR}.
\end{proof}

\begin{corollary}[Explicit spherical geometry]
\label{cor:geodesics}
In the Fisher--Rao geometry of the simplex, geodesics are the
preimages of great-circle arcs, and the geodesic distance is the
Bhattacharyya angle \cite{Bhattacharyya1943}:
\[
d_{\mathrm{FR}}(\pi,\rho)
=2\arccos\Bigl(\sum_{i=1}^m\sqrt{\pi_i\rho_i}\Bigr).
\]
\end{corollary}

The practical content of Proposition~\ref{prop:isometry} is that
the point-dependent inner product \eqref{eq:FR} becomes, in the
coordinates $x=\Phi(\pi)$, the standard dot product of
$\mathbb R^m$, independent of the base point. The weighted Hilbert
structure of the reset parameters thus admits a global,
unweighted Euclidean realisation, and all metric complexity is
absorbed into the coordinate change $\Phi$. This simplification
is the geometric engine of the paper: from here on, every metric
statement about the reset landscape is a statement of standard
Riemannian geometry on a round sphere, and the principal metric
objects used below --- geodesics, distances, angles, reflections
--- are computable in closed form.
Moreover, $\Phi$ is an isometry of Riemannian manifolds, not
merely of metric spaces: the Levi-Civita connection, the constant
sectional curvature $1/4$, the exponential and logarithm maps,
Jacobi fields --- every local differential-geometric construction
on the simplex is transported from the round sphere through
$\Phi$. We shall use only a small part of this inheritance, but
all of it is available.

\begin{remark}[The octant]
\label{rem:octant}
The image of $\Phi$ is the \emph{octant}, not the full sphere:
the positivity of probabilities is encoded as the positivity of
all coordinates. The boundary of the octant corresponds to
degenerate reset distributions (some $\pi_i=0$). Two elementary
consequences will be used repeatedly. First, any two points of
the closed octant subtend an angle at most $\pi/2$, since
nonnegative unit vectors have nonnegative inner product: the
octant sits entirely within the ``short-distance'' regime of
spherical geometry. In particular, no two reset distributions are
antipodal, and the cut locus of every point of the octant lies
outside the octant: several non-uniqueness phenomena of spherical
geometry, originating in antipodal configurations, are simply
absent from the positive simplex. Second --- and less obviously
--- the
positivity constraint is a geometric \emph{resource} rather than
a nuisance: several exact statements of
Section~\ref{sec:geometry} hold precisely because the ambient
space is the octant and not the whole sphere.
\end{remark}

\subsection{Gradients, covectors and normals}
\label{subsec:gradients}

The orientation field $\psi(\gamma)$ of
Section~\ref{sec:inherited} was introduced as a covector: it acts
on displacements through the standard pairing, with no metric
involved. A metric converts covectors into vectors, and the
conversion does real work in this paper, so we record it
explicitly. For a $\mathcal C^1$ functional
$F:\Delta_{m-1}^\circ\to\mathbb R$ with ambient differential
$\nabla F$, the Fisher--Rao gradient is characterised by
$g^{\mathrm{FR}}_\pi(\operatorname{grad}^{\mathrm{FR}}F,h)=
\langle\nabla F,h\rangle$ for all tangent $h$, which gives
\begin{equation}
\operatorname{grad}^{\mathrm{FR}}_\pi F
=\Pi_\pi\bigl(\pi\odot\nabla F\bigr),
\qquad
(\pi\odot\nabla F)_i=\pi_i\,\partial_iF,
\label{eq:grad}
\end{equation}
where $\Pi_\pi$ is the $g^{\mathrm{FR}}$-orthogonal projection
onto $T_\pi\Delta_{m-1}$. In particular, a covector $\mathbf n$
(defining a hyperplane through the Euclidean pairing, as in
Section~\ref{sec:inherited}) and its geometric realisation
$\Pi_\pi(\pi\odot\mathbf n)$ --- the Fisher--Rao normal to that
hyperplane \emph{within} $T_\pi\Delta_{m-1}$, the projection being
needed because $\pi\odot\mathbf n$ is not tangent in general ---
are \emph{different} objects, related by the metric. Conflating
them produces wrong angles on the sphere; distinguishing them is
what will allow the algebraic normals of
Section~\ref{sec:finite-reduction} to become genuine geometric
normals in Section~\ref{sec:geometry}.

\medskip
A closing observation frames what follows, and we state it as a
principle. \emph{The Fisher--Rao manifold is universal: it depends
only on the number $m$ of reset sites. Every reset model with $m$
sites therefore lives on exactly the same Riemannian manifold.
What distinguishes one model from another is not the ambient
geometry but the distinguished algebraic objects embedded in it
--- the separatrix $\Sigma$, the response span $V$, the involution
$\sigma$ and its orbit decomposition.} The geometry developed in
the remainder of the paper is precisely the interaction between
these universal and model-dependent structures: the algebraic
skeleton of Section~\ref{sec:finite-reduction} placed on the stage
built here (Section~\ref{sec:geometry}). The direction of
explanation is thereby inverted: the model no longer generates the
structure --- it instantiates it.

\begin{remark}[Affine triviality of the tangent bundle]
\label{rem:affine-triviality}
The simplex is an open subset of the affine hyperplane
$\{x:\langle x,\mathbf 1\rangle=1\}$, so its tangent bundle is
trivial: every tangent space is canonically the zero-sum hyperplane
\[
T_0:=\{v\in\mathbb R^m:\langle v,\mathbf 1\rangle=0\}
\]
of mass-preserving
perturbations, and for any two distributions the difference
$\rho-\pi$ lies in $T_0$. The Fisher--Rao geometry may therefore be
regarded as a smoothly varying inner product on a single fixed
vector space: the curvature resides entirely in the
$\pi$-dependence of $g_\pi$, not in the topology of the bundle.
Equivalently, the rescaling $v\mapsto v/\pi$ identifies
$(T_0,g_\pi)$ isometrically with the space
$L^2_0(\pi)=\{f:\sum_i f_i\pi_i=0\}$ of zero-mean observables with
the $L^2(\pi)$ inner product --- a fixed-form inner product on a
$\pi$-dependent space --- providing a natural bridge with the
weighted Hilbert-space formulation of the spectral theory in
Paper~I. Note finally the dual role of the response data: tangent
vectors are zero-sum, while the orientation field $\psi(\gamma)$
and the orbit directors act on the simplex as covectors, well
defined modulo the constant vector $\mathbf 1$ \emph{when paired
with tangent directions only}. The qualification is essential: the
functional $L_\gamma(\pi)=\langle\pi,\psi(\gamma)\rangle$ of
Section~\ref{sec:inherited} pairs $\psi$ with a point of the
simplex rather than with a tangent vector, and there the shift is
not free --- $\langle\pi,\psi+c\mathbf 1\rangle=L_\gamma(\pi)+c$.
The representative of $\psi$ fixed by
$\psi=u-C^\ast s$ is the one used throughout.
\end{remark}

\begin{remark}[Projective interpretation: simplex and sphere as
gauges]
\label{rem:projective}
The coupling functional is homogeneous of degree zero,
$C(\lambda\pi,\gamma)=C(\pi,\gamma)$ for every $\lambda>0$,
since numerator and denominator are linear in $\pi$. The homogeneity of $C$ therefore also permits a projective
reading, on the space of rays of the positive cone
$\mathbb R^m_{>0}$; the statistical manifold on which the
Fisher--Rao geometry lives remains the normalised simplex. The simplex is its
affine section $\sum_i\pi_i=1$; the spherical octant of
Proposition~\ref{prop:isometry} is its metric section
$\sum_ix_i^2=4$, $x=2\sqrt\pi$; and $\Phi$ is not a change of
state space but a change of section. The Fisher--Rao metric does
not create the spherical picture: under the square-root
representation it identifies the simplex section isometrically with
the spherical section of the same cone, making manifest a structure
already implicit in the homogeneity of $C$. The projective viewpoint is not needed in
this paper, but it underlies the topological questions raised in
the concluding section.
\end{remark}

\section{Orbit collapse and finite reduction}
\label{sec:finite-reduction}

This section establishes the structural core of the paper. The
orientation field $\psi(\gamma)$, a priori an infinite family indexed
by the reset parameter $\gamma\in(0,1)$, spans a subspace $V$ of
dimension equal to the number of orbits of the involution $\sigma$.
The mechanism is not spectral: it is an exact algebraic
collapse of each orbit, induced by the critical value
$C^\ast=1/(1+\sqrt K)$ of Theorem~\ref{thm:paper3}. The three structural
conditions each play one role: (S2) provides the spectral expansion,
(S3) builds the normal directions, and (S4) makes the orbit matrix
singular. We isolate the collapse as the fundamental phenomenon
(Lemma~\ref{lem:orbit-collapse}); finite reduction is its corollary
(Theorem~\ref{thm:fsr}).

\subsection{Inherited structure}
\label{subsec:inherited-structure}

We restate the structural conditions of
Section~\ref{sec:inherited} in the form in which they are used
here; they are the axioms of Definition~\ref{def:sds}, and nothing
below is assumed beyond them. The kernel data $u(z;\gamma)$,
$s(z;\gamma)$ on the reset sites $\{z_1,\dots,z_m\}$ admit:

\begin{itemize}
\item[\textbf{(S2)}] a representation
$C(\pi,\gamma)=\langle\pi,u(\cdot;\gamma)\rangle/
\langle\pi,s(\cdot;\gamma)\rangle$ whose denominator is strictly
positive on \emph{all} of $\Delta_{m-1}\times(0,1)$, boundary
included --- a clause of the axiom, not a consequence of the
expansion --- together with a spectral expansion
$u(z;\gamma)=\sum_{\nu=1}^N f_\nu(\gamma)A_\nu(z)$ and
$s(z;\gamma)=\sum_\nu f_\nu(\gamma)\big(A_\nu(z)+B_\nu(z)\big)$, with the
$f_\nu$ linearly independent as functions on $(0,1)$, and with the
$m$ \emph{site vectors}
\[
\mathbf A(z):=\bigl(A_\nu(z)\bigr)_{\nu=1}^N\in\mathbb R^N,
\qquad z\in\{z_1,\dots,z_m\},
\]
linearly independent --- equivalently, the matrix
$\big(A_\nu(z_i)\big)\in\mathbb R^{N\times m}$, whose \emph{columns}
are the site vectors $\mathbf A(z_i)$, has full column rank
$m\le N$. It is the independence of the $\mathbf A(z)$, one vector
per reset site, that is used below;
\item[\textbf{(S3)}] a twisted symmetry
$B_\nu(z)=\kappa(z)\,A_\nu(\sigma(z))$, with weights $\kappa(z)>0$
independent of $\nu$;
\item[\textbf{(S4)}] a compatibility relation
$\kappa(z)\,\kappa(\sigma(z))=K$.
\end{itemize}

The orientation field \eqref{eq:psi-def} has components
$\psi_z(\gamma)=u(z;\gamma)-C^\ast s(z;\gamma)$, where
$C^\ast=1/(1+\sqrt K)$ is the invariant scalar of
Theorem~\ref{thm:paper3}, and
$V=\operatorname{span}\{\psi(\gamma):\gamma\in(0,1)\}$ is the
response span, of dimension $r$. We recall from
Section~\ref{sec:inherited} that an \emph{orbit} of $\sigma$ is a
pair $O=\{z,\sigma(z)\}$ with $z\neq\sigma(z)$, that a
\emph{neutral site} is a fixed point $\sigma(z_0)=z_0$, for which
(S4) gives $\kappa(z_0)^2=K$ and hence, the weights being positive
by (S3), $\kappa(z_0)=+\sqrt K$; and that
$m=2n_{\mathrm{pair}}+n_0$. In each orbit we fix once and for all a
representative $z_O$, and order the two coordinates of the orbit
plane as $(z_O,\sigma(z_O))$; the choice affects the sign of the
director $\mathbf d_O$ below and nothing else. The content of this
section is that $r=n_{\mathrm{pair}}$.

\subsection{The normal directions}
\label{subsec:normals}

Substituting (S3) into the definition of $\psi$ and using the
expansion (S2),
\begin{equation}
\psi_z(\gamma)
=\sum_\nu f_\nu(\gamma)\Big[A_\nu(z)(1-C^\ast)-C^\ast\kappa(z)A_\nu(\sigma(z))\Big]
=\sum_\nu f_\nu(\gamma)\,[\mathbf n_\nu]_z,
\label{eq:decomp}
\end{equation}
where the \emph{normal directions} are defined by
\begin{equation}
[\mathbf n_\nu]_z := A_\nu(z)(1-C^\ast)-C^\ast\kappa(z)\,A_\nu(\sigma(z)).
\label{eq:normals}
\end{equation}
The decomposition \eqref{eq:decomp} is immediate from (S2)+(S3): the
$\gamma$-dependence is confined to the scalars $f_\nu$, and the
$\mathbf n_\nu$ are fixed.

\begin{lemma}
\label{lem:span}
$V=\operatorname{span}\{\mathbf n_\nu\}_\nu$.
\end{lemma}

\begin{proof}
The inclusion $V\subseteq\operatorname{span}\{\mathbf n_\nu\}$ is
\eqref{eq:decomp}. Conversely, let $L$ be a linear functional on
$\mathbb R^m$ vanishing on $V$. Then
$\sum_\nu f_\nu(\gamma)\,L(\mathbf n_\nu)=0$ for every
$\gamma\in(0,1)$, and the linear independence of the $f_\nu$
\textup{(S2)} forces $L(\mathbf n_\nu)=0$ for every $\nu$. Thus
every functional annihilating $V$ annihilates each $\mathbf n_\nu$,
whence $\operatorname{span}\{\mathbf n_\nu\}\subseteq V$.
\end{proof}

The problem is therefore reduced to computing the span of a finite
set of fixed vectors.

\subsection{The orbit collapse lemma}
\label{subsec:collapse}

The key phenomenon is that each orbit contributes a single direction.

\begin{lemma}[Orbit collapse]
\label{lem:orbit-collapse}
Under \textup{(S3)--(S4)}, for each orbit $O=\{z,\sigma(z)\}$ the
matrix
\[
M_O=\begin{pmatrix}1-C^\ast & -C^\ast\kappa(z)\\[2pt]
-C^\ast\kappa(\sigma(z)) & 1-C^\ast\end{pmatrix}
\]
has rank exactly one, with
\[
\operatorname{im}M_O=\mathbb R\,\mathbf d_O,\qquad
\mathbf d_O:=\bigl(\sqrt{\kappa(z)},\,-\sqrt{\kappa(\sigma(z))}\bigr),
\qquad
\ker M_O=\mathbb R\bigl(\sqrt{\kappa(z)},\,\sqrt{\kappa(\sigma(z))}\bigr).
\]
Consequently, for every $\nu$ the pair
$\big([\mathbf n_\nu]_z,[\mathbf n_\nu]_{\sigma(z)}\big)$ is a scalar
multiple of the single fixed vector $\mathbf d_O$: the restriction of
every normal direction to a given orbit collapses onto one line.
\end{lemma}

\begin{proof}
On the orbit $O$, \eqref{eq:normals} reads
\[
\begin{pmatrix}[\mathbf n_\nu]_z\\[2pt][\mathbf n_\nu]_{\sigma(z)}\end{pmatrix}
= M_O\begin{pmatrix}A_\nu(z)\\[2pt]A_\nu(\sigma(z))\end{pmatrix}.
\]
Its determinant is
$\det M_O=(1-C^\ast)^2-C^{\ast2}\kappa(z)\kappa(\sigma(z))
=(1-C^\ast)^2-C^{\ast2}K$ by (S4). With $C^\ast=1/(1+\sqrt K)$,
\[
(1-C^\ast)^2=\frac{K}{(1+\sqrt K)^2},\qquad
C^{\ast2}K=\frac{K}{(1+\sqrt K)^2},
\]
so $\det M_O=0$. The rank is exactly one and not zero, since the
diagonal entry $1-C^\ast=\sqrt K/(1+\sqrt K)$ is strictly positive.

For the image, note first the identity $1-C^\ast=\sqrt K\,C^\ast$,
and that (S4) gives $\sqrt K=\sqrt{\kappa(z)}\sqrt{\kappa(\sigma(z))}$.
The first column of $M_O$ is therefore
\[
\bigl(1-C^\ast,\,-C^\ast\kappa(\sigma(z))\bigr)
=C^\ast\bigl(\sqrt K,\,-\kappa(\sigma(z))\bigr)
=C^\ast\sqrt{\kappa(\sigma(z))}\;\mathbf d_O ,
\]
a non-zero multiple of $\mathbf d_O$; since the rank is one, the
image is spanned by it. For the kernel, $M_O(x,y)^{\!\top}=0$ reads
$(1-C^\ast)x=C^\ast\kappa(z)\,y$, that is
$x/y=\kappa(z)/\sqrt K=\sqrt{\kappa(z)/\kappa(\sigma(z))}$, which is
the stated line.

A warning is in order, since the two lines are easy to conflate.
The image is spanned by the \emph{columns}, as computed above; it is
\emph{not} the orthogonal complement of the kernel. That identity
holds for symmetric matrices, and $M_O$ is symmetric only when
$\kappa(z)=\kappa(\sigma(z))$. Here
\[
\ker M_O=\mathbb R\bigl(\sqrt{\kappa(z)},\sqrt{\kappa(\sigma(z))}\bigr),
\qquad
(\ker M_O)^{\perp}
=\mathbb R\bigl(\sqrt{\kappa(\sigma(z))},-\sqrt{\kappa(z)}\bigr)
\neq\operatorname{im}M_O ,
\]
the two differing by the coordinate swap of the orbit plane; it is
$(\operatorname{im}M_O)^{\perp}$, not $\ker M_O$, that carries the
critical ratio of Theorem~\ref{thm:paper3}\textup{(i)}. See
Remark~\ref{rem:cstar}. The final assertion is that the left-hand column
above lies in $\operatorname{im}M_O$.
\end{proof}

\begin{remark}
The value $C^\ast=1/(1+\sqrt K)$ is precisely the one making
$\det M_O=0$: it is the unique root of $(1-C)^2=C^2K$ in $(0,1)$.
The scalar that Paper~III singled out by a dynamical requirement
(reset neutrality) is the same scalar that here produces an
algebraic rank collapse. We return to this in
Remark~\ref{rem:cstar}.
\end{remark}

\subsection{Finite reduction}
\label{subsec:reduction}

\begin{theorem}[Finite reduction]
\label{thm:fsr}
Assume \textup{(S2)--(S4)}, and let $\mathbf d_O$ be the orbit
directors of Lemma~\ref{lem:orbit-collapse}. Then:
\begin{enumerate}
\item[\textup{(i)}] \emph{Orbitwise decomposition.} Every normal
direction decomposes as
\[
\mathbf n_\nu=\sum_{O} w_\nu^{(O)}\,\mathbf d_O,
\qquad
w^{(O)}_\nu=\frac{[\mathbf n_\nu]_{z_O}}{\sqrt{\kappa(z_O)}},
\]
each $\mathbf d_O$ being supported on its own orbit. In particular
every $\mathbf n_\nu$ vanishes at every neutral site.
\item[\textup{(ii)}] \emph{Activation and independence.} The
coefficient matrix $\bigl(w^{(O)}_\nu\bigr)_{\nu,O}$ has full column
rank $n_{\mathrm{pair}}$; in particular no orbit is inactive.
\item[\textup{(iii)}] \emph{Finite reduction.}
$V=\bigoplus_{O}\mathbb R\,\mathbf d_O$, so that
$r=\dim V=n_{\mathrm{pair}}$, and
\begin{equation}
\psi(\gamma)=\sum_{O}\varphi_O(\gamma)\,\mathbf d_O,
\qquad
\varphi_O(\gamma):=\sum_\nu f_\nu(\gamma)\,w^{(O)}_\nu .
\label{eq:orbit-response}
\end{equation}
\end{enumerate}
Each orbit of $\sigma$ therefore contributes exactly one fixed
direction to $V$, neutral sites contribute none, and all
$\gamma$-dependence is carried by the orbit response scalars
$\varphi_O$.
\end{theorem}

\begin{proof}
\textup{(i)} By Lemma~\ref{lem:orbit-collapse} the restriction of
$\mathbf n_\nu$ to an orbit $O$ is a multiple of $\mathbf d_O$;
write it $w^{(O)}_\nu\mathbf d_O$. Since
$[\mathbf d_O]_{z_O}=\sqrt{\kappa(z_O)}\neq0$, the coefficient is
recovered from the representative site as
$w^{(O)}_\nu=[\mathbf n_\nu]_{z_O}/\sqrt{\kappa(z_O)}$. At a neutral
site $z_0$ one has $\sigma(z_0)=z_0$, so (S4) reads
$\kappa(z_0)^2=K$ and the positivity of $\kappa$ in (S3) selects the
root $\kappa(z_0)=+\sqrt K$; then \eqref{eq:normals} gives
\[
[\mathbf n_\nu]_{z_0}=A_\nu(z_0)\bigl[(1-C^\ast)-C^\ast\sqrt K\bigr]=0,
\]
again because $C^\ast=1/(1+\sqrt K)$. As the orbits are pairwise
disjoint and their union is the complement of the neutral sites,
$\mathbf n_\nu=\sum_O w^{(O)}_\nu\mathbf d_O$.

\textup{(ii)} Suppose $\sum_O c_O\,w^{(O)}_\nu=0$ for every $\nu$.
Substituting the expression for $w^{(O)}_\nu$ and
\eqref{eq:normals}, and reading the resulting identity as one in
$\mathbb R^N$ --- that is, collecting the $N$ scalar identities
indexed by $\nu$ into a single vector identity among the site
vectors $\mathbf A(z)$ of \textup{(S2)} --- we obtain
\begin{equation}
\sum_{O}\frac{c_O}{\sqrt{\kappa(z_O)}}
\Bigl[(1-C^\ast)\,\mathbf A(z_O)
-C^\ast\kappa(z_O)\,\mathbf A(\sigma(z_O))\Bigr]=\mathbf 0 .
\label{eq:site-combination}
\end{equation}
This is a vanishing linear combination of the $2n_{\mathrm{pair}}$
site vectors $\{\mathbf A(z_O),\mathbf A(\sigma(z_O))\}_O$, which
are pairwise distinct members of the family
$\{\mathbf A(z_i)\}_{i=1}^m$, the orbits being disjoint. That family
is linearly independent by \textup{(S2)}, and so is any subfamily,
so every coefficient in \eqref{eq:site-combination} vanishes; the
coefficient of $\mathbf A(z_O)$ is $c_O(1-C^\ast)/\sqrt{\kappa(z_O)}$,
and $1-C^\ast=\sqrt K/(1+\sqrt K)\neq0$, so $c_O=0$ for each $O$. The
$n_{\mathrm{pair}}$ columns of $\bigl(w^{(O)}_\nu\bigr)$ are thus
independent; in particular none of them is zero, that is, no orbit
is inactive.

The same conclusion can be read off directly, and is worth
recording. An orbit is inactive precisely when
$[\mathbf n_\nu]_{z}=[\mathbf n_\nu]_{\sigma(z)}=0$ for all $\nu$,
which by Lemma~\ref{lem:orbit-collapse} means that
$(A_\nu(z),A_\nu(\sigma(z)))\in\ker M_O$ for every $\nu$; that is,
the two site vectors $\mathbf A(z)$ and $\mathbf A(\sigma(z))$ would
be proportional, with ratio $\rho=\sqrt{\kappa(z)/\kappa(\sigma(z))}$.
Two proportional members of a linearly independent family being
impossible, \textup{(S2)} excludes this.

\textup{(iii)} The directors $\mathbf d_O$ are non-zero with
pairwise disjoint supports, hence linearly independent, and the map
$T:(c_O)_O\mapsto\sum_O c_O\mathbf d_O$ is injective. By
\textup{(i)}, $\mathbf n_\nu=T(w^{(\cdot)}_\nu)$, so
$\operatorname{span}\{\mathbf n_\nu\}_\nu$ is the image under $T$ of
the row space of $\bigl(w^{(O)}_\nu\bigr)$; by \textup{(ii)} that
row space is all of $\mathbb R^{n_{\mathrm{pair}}}$. Hence
$\operatorname{span}\{\mathbf n_\nu\}_\nu=\bigoplus_O\mathbb R\,\mathbf d_O$,
of dimension $n_{\mathrm{pair}}$, and this equals $V$ by
Lemma~\ref{lem:span}. Substituting \textup{(i)} into
\eqref{eq:decomp} and collecting the spectral scalars orbitwise
gives \eqref{eq:orbit-response}.
\end{proof}

\begin{remark}[The three roles of $C^\ast$]
\label{rem:cstar}
The critical value $C^\ast=1/(1+\sqrt K)$ appears in three distinct
places: as the value eliminating the $\gamma$-dependence of the ruin
functional (Paper~III, a dynamical condition); as the value making the
orbit matrix $M_O$ singular (Lemma~\ref{lem:orbit-collapse}, an algebraic
collapse); and as the value fixing the two distinguished lines of the
orbit plane, $\operatorname{im}M_O=\mathbb R\,\mathbf d_O$ and
$\ker M_O$. The latter is the \emph{inactivity direction}: an orbit
would fail to contribute exactly when its two site vectors
$\mathbf A(z)$, $\mathbf A(\sigma(z))$ sat in the ratio $\rho=C^\ast\kappa(z)/(1-C^\ast)=\sqrt{\kappa(z)/\kappa(\sigma(z))}$,
which is the reciprocal of the ratio
$\pi^\ast_z/\pi^\ast_{\sigma(z)}=\sqrt{\kappa(\sigma(z))/\kappa(z)}$
characterising the invariant distributions
(Theorem~\ref{thm:paper3}\textup{(i)}); the two lines are exchanged
by the coordinate swap of the orbit plane. These are not
coincidences: the collapse is the algebraic shadow of the dynamical
neutrality, both governed by the compatibility (S4) between the
twisted duality and the critical value.
\end{remark}

\begin{remark}[Positivity, model-specific]
\label{rem:positivity}
The scalars $f_\nu(\gamma)$ are positive on $(0,1)$ for the biased
random walk, where $f_\nu(\gamma)=(1-\gamma)/(1-\lambda_\nu(1-\gamma))$
with $|\lambda_\nu|<1$; this drives the orientation theorem of
Section~\ref{sec:sign-theorem}. Positivity is not needed for the
reduction itself, which uses only (S2)--(S4).
\end{remark}

\begin{remark}[Hypotheses]
\label{rem:hypotheses}
The reduction uses each structural condition once: (S2) gives the
expansion and, through the independence of the site vectors, the
activation of every orbit;
(S3) builds the normal directions; (S4) makes each orbit matrix
singular. It does not use the resetting mechanism, nor any spectral
representation of a generator. A spectral realisation for the random
walk, connecting the collapse to the resolvent structure of Paper~IV,
is given in Appendix~\ref{app:spectral}.
\end{remark}

\section{Rigidity of the separatrix}
\label{sec:geometry}

The two preceding threads now meet. Section~\ref{sec:finite-reduction}
established, by purely algebraic means, that the reset landscape is
governed by a finite family of directions, one per orbit of the
involution; Section~\ref{sec:fisher-rao} built the universal stage
on which reset distributions live. This section places the algebra
on the stage. The outcome is a complete geometric description of
the separatrix --- its dimension, an explicit linear model,
intrinsic coordinates, the isometric reflection fixing it, the
nearest-point projection and distance, and the canonical quadratic
functional detecting it --- and its
character is rigidity: the direction of every orbit is forced. The
only remaining freedom, the distribution of mass among orbits and
neutral sites, is not a failure of rigidity but the intrinsic
coordinate system of the separatrix itself
(Corollary~\ref{cor:intrinsic}). None of these linear
relations is created by the metric; they are entirely encoded in
the singular orbit matrix of Lemma~\ref{lem:orbit-collapse}. The
metric merely reveals their intrinsic geometric meaning. Nor is
the rigidity an ornament: the great subsphere identified below is
the common zero set of the transverse response functionals, and is
the geometric input for the orientation theorem of
Section~\ref{sec:sign-theorem}. Its separating character, however,
is a codimension-one phenomenon: when $r=1$ it is a hypersurface and
does bound the two regions of definite sign, whereas for $r\ge2$ it
has codimension $r$ and does not separate the sphere at all --- it
is then the common codimension-$r$ zero set of the transverse
coordinates. This is the geometric face of the dichotomy of
Section~\ref{sec:sign-theorem}: two zones when $r=1$, a response
field free to rotate when $r\ge2$. The name \emph{separatrix} is
retained throughout in its established sense of the neutral set,
and no topological separation property is claimed beyond $r=1$.

\subsection{One matrix, two shadows}
\label{subsec:principle}

Recall from Lemma~\ref{lem:orbit-collapse} the orbit matrix
\[
M_O=\begin{pmatrix}1-C^\ast & -C^\ast\kappa(z)\\[2pt]
-C^\ast\kappa(\sigma(z)) & 1-C^\ast\end{pmatrix},
\qquad \det M_O=0\ \text{ at }\ C^\ast=\frac1{1+\sqrt K},
\]
whose rank-one degeneration is the fundamental phenomenon of the
paper. It casts two shadows. The \emph{algebraic} shadow was the
subject of Section~\ref{sec:finite-reduction}: on each orbit, all
normal directions collapse onto the single director spanning
$\operatorname{im}M_O$. The \emph{geometric} shadow is
the subject of this section: on the sphere, the quadric cut out by
each orbit condition degenerates into a pair of hyperplanes, of
which the octant retains exactly one. One object, two readings.

We first record the two elementary properties of the director that
drive everything.

\begin{lemma}[Orthogonality of the orbit director]
\label{lem:director}
For each orbit $O=\{z,\sigma(z)\}$, write
$\mathbf v_O=(\sqrt{\kappa(\sigma(z))},\sqrt{\kappa(z)})$ for the
\emph{critical vector}, whose direction realises the critical
ratio of Theorem~\ref{thm:paper3}\textup{(i)}. Then the director
of Lemma~\ref{lem:orbit-collapse}, which spans
$\operatorname{im}M_O$, is
\[
\mathbf d_O:=\bigl(\sqrt{\kappa(z)},\ -\sqrt{\kappa(\sigma(z))}\bigr)
\;=\;R_{-\pi/2}\,\mathbf v_O,
\]
where $R_{-\pi/2}\in SO(2)$ is the rotation by $-\pi/2$ of the
Euclidean orbit plane. In particular $\mathbf d_O$ and
$\mathbf v_O$ are orthogonal \emph{and of equal Euclidean norm}
$\sqrt{\kappa(z)+\kappa(\sigma(z))}$: the director is not a new
object --- it is the critical vector viewed from the orthogonal
direction, with no deformation whatsoever. Consequently:
\begin{enumerate}
\item[\textup{(i)}] the two components of $\mathbf d_O$ have
opposite signs;
\item[\textup{(ii)}] $\langle\pi|_O,\mathbf d_O\rangle=0$ for every
$\pi$ satisfying the critical ratio.
\end{enumerate}
\end{lemma}

\begin{proof}
That $\operatorname{im}M_O=\mathbb R\,\mathbf d_O$ is
Lemma~\ref{lem:orbit-collapse}; only the geometric reading remains.
The map $(a,b)\mapsto(b,-a)$ is the rotation by
$-\pi/2$, and a rotation preserves norms and is orthogonal to its
argument: orthogonality and equality of norms follow, and any
$\pi$ with the critical ratio is a \emph{nonnegative} multiple of
$\mathbf v_O$ on $O$ --- the zero multiple corresponding to an orbit
of zero mass, which is why the statement survives on the boundary
strata used in Proposition~\ref{prop:sigma-exact} and
Corollary~\ref{cor:cone}. Note that \textup{(S4)} is doing double
duty: it makes $M_O$ singular (Lemma~\ref{lem:orbit-collapse}),
\emph{and} it rotates the director exactly orthogonal to the
critical vector.
\end{proof}

In retrospect, Lemma~\ref{lem:director} identifies two objects of
Paper~III that had appeared unrelated: the direction defining the
invariant distributions and the direction governing the response
are the same vector, viewed at right angles.

\subsection{Exact characterisation and linearisation}
\label{subsec:linearisation}

The global structure of the separatrix was identified in
Paper~III (Proposition~5.1): it is the affine subset of the
simplex cut out by one ratio constraint per orbit. We now
re-derive this characterisation from the orbit collapse ---
exhibiting its mechanism --- and recast it in the form adapted to
the sphere.

\begin{proposition}[Exact linearity of the separatrix; cf.\
Paper~III, Prop.~5.1]
\label{prop:sigma-exact}
Under \textup{(S1)--(S4)}, a distribution lies in the separatrix
if and only if it annihilates every normal direction:
\[
\pi\in\Sigma\iff\langle\pi,\mathbf n_\nu\rangle=0\ \ \forall\nu
\iff\langle\pi|_O,\mathbf d_O\rangle=0\ \ \text{for every orbit }O,
\]
and explicitly, on the whole simplex,
\[
\Sigma
=\Bigl\{\pi\in\Delta_{m-1}:\
\sqrt{\kappa(z)}\,\pi_z=\sqrt{\kappa(\sigma(z))}\,\pi_{\sigma(z)}\
\text{for every orbit}\Bigr\},
\]
with the weights of neutral sites unconstrained. Note that the
linear form of the condition already covers the boundary strata: if
one of $\pi_z,\pi_{\sigma(z)}$ vanishes so does the other, so no
orbit can carry mass at a single site --- which is how Paper~III's
ratio condition, stated only on orbits of positive mass, is
subsumed here. Equivalently:
$\Sigma$ is the intersection of the simplex with a linear subspace
of $\mathbb R^m$ of codimension $r=n_{\mathrm{pair}}$; the
interior part $\Sigma^\circ$, which the spherical
picture below uses, is the corresponding relatively open subset. We
write
\[
\Sigma^\circ:=\Sigma\cap\Delta^\circ_{m-1}
\]
throughout, and keep the distinction firmly: the \emph{algebraic}
statements of this section hold on $\Sigma\subseteq\Delta_{m-1}$,
closed simplex included, while every \emph{Riemannian} statement ---
distances, geodesics, isometries, Hessians --- is made on
$\Sigma^\circ\subseteq\Delta^\circ_{m-1}$, where the Fisher--Rao
metric is defined.
This recovers, through the orbit collapse, the affine characterisation
of the separatrix obtained in Paper~III, and prepares it for the
spherical representation.
\end{proposition}

\begin{proof}
By Theorem~\ref{thm:paper3}, any $\pi$ with $C(\pi,\cdot)$ constant
has constant value $C^\ast$; the separatrix is defined on the
closed simplex, and Theorem~\ref{thm:paper3} covers it as stated.
The positivity clause of \textup{(S2)} places
$\langle\pi,s(\cdot;\gamma)\rangle>0$ on all of
$\Delta_{m-1}\times(0,1)$, boundary included --- assumed there, not
derived from the expansion --- so that no step below requires $\pi$
to be interior. Hence
\[
C(\pi,\gamma)=C^\ast\ \forall\gamma
\iff
\langle\pi,u(\cdot;\gamma)-C^\ast s(\cdot;\gamma)\rangle=0\
\forall\gamma
\iff
\sum_\nu f_\nu(\gamma)\,\langle\pi,\mathbf n_\nu\rangle=0\
\forall\gamma,
\]
using the decomposition \eqref{eq:normals}. Linear independence of
the $f_\nu$ \textup{(S2)} makes this equivalent to
$\langle\pi,\mathbf n_\nu\rangle=0$ for all $\nu$. By the orbit
collapse (Lemma~\ref{lem:orbit-collapse}) and the disjointness of
orbit supports, $\mathbf n_\nu=\sum_O w^{(O)}_\nu\mathbf d_O$ with
the coefficient matrix $(w^{(O)}_\nu)$ of full column rank
$n_{\mathrm{pair}}$. Writing $t_O:=\langle\pi|_O,\mathbf d_O\rangle$,
the conditions read $\sum_O w^{(O)}_\nu t_O=0$ for every $\nu$; it is
the injectivity of $(w^{(O)}_\nu)$ that forces $t_O=0$ for every
orbit. The point deserves emphasis: the orbit collapse alone gives
only the inclusion $V\subseteq\operatorname{span}\{\mathbf d_O\}_O$,
and the equality of the annihilator of the $\mathbf n_\nu$ with the
orbitwise kernel is precisely where the full-column-rank
\emph{activation} of Theorem~\ref{thm:fsr}\textup{(ii)} enters, not
merely the orbitwise collapse. \textup{(}The converse implication
needs only the decomposition.\textup{)} By
Lemma~\ref{lem:director}, $t_O=0$ is the stated ratio condition. Neutral
sites carry no condition since every $\mathbf n_\nu$ vanishes
there (Theorem~\ref{thm:fsr}). The conditions are linear in $\pi$
and independent across orbits (disjoint supports), giving
codimension $n_{\mathrm{pair}}$.
\end{proof}

On the sphere, the exact linearity of
Proposition~\ref{prop:sigma-exact} becomes total geodesy --- but
this requires an argument, because $\Phi$ is quadratic: the image
of an affine set under $\pi\mapsto2\sqrt\pi$ is, generically, a
curved quadric. The mechanism that prevents this is the octant,
through the following observation.

\begin{lemma}[Scalar octant lemma]
\label{lem:scalar}
For $\alpha,\beta>0$,
\[
\bigl\{x\in\mathbb R^2_{\ge0}:\ \alpha x_1^2=\beta x_2^2\bigr\}
=\bigl\{x\in\mathbb R^2_{\ge0}:\
\sqrt\alpha\,x_1=\sqrt\beta\,x_2\bigr\}.
\]
\end{lemma}

\begin{proof}
Two nonnegative numbers with equal squares are equal.
\end{proof}

\begin{theorem}[Linearisation: the separatrix is a great subsphere]
\label{thm:great-subsphere}
Let $L\subset\mathbb R^m$ be the linear subspace
\[
L=\bigl\{x:\ \kappa(z)^{1/4}x_z=\kappa(\sigma(z))^{1/4}\,x_{\sigma(z)}\
\text{for every orbit}\bigr\},\qquad \dim L=m-r.
\]
Then
\[
\Phi\bigl(\Sigma^\circ\bigr)
=L\cap S^{m-1}_+(2).
\]
In particular, the image of the separatrix is an open piece of a
great subsphere of dimension $m-r-1$; hence $\Sigma^\circ$ is a
totally geodesic submanifold of $(\Delta^\circ_{m-1},
g^{\mathrm{FR}})$, of codimension exactly $r$. \textup{(}This
statement was obtained independently, and by the same route, by
T.~Newton \cite{Newton}.\textup{)} The closed set
$\Sigma$ remains available as the algebraic object of
Proposition~\ref{prop:sigma-exact}; ``submanifold'' is asserted of
$\Sigma^\circ$ only.
\end{theorem}

\begin{proof}
Substituting $\pi=x^2/4$ into the characterisation of
Proposition~\ref{prop:sigma-exact}, each orbit condition reads
$\sqrt{\kappa(z)}\,x_z^2=\sqrt{\kappa(\sigma(z))}\,x_{\sigma(z)}^2$
with both coefficients positive by \textup{(S3)}. On the
octant,
Lemma~\ref{lem:scalar} replaces each such quadric by the single
linear equation
$\kappa(z)^{1/4}x_z=\kappa(\sigma(z))^{1/4}x_{\sigma(z)}$.
One independent equation per orbit (disjoint supports) gives
$\dim L=m-n_{\mathrm{pair}}=m-r$; neutral coordinates are
unconstrained. Finally, since $L$ is a linear subspace, the
orthogonal reflection of $\mathbb R^m$ across $L$ restricts to an
isometry of the sphere whose fixed-point set is exactly $L\cap
S^{m-1}(2)$; and this fixed-point set is totally geodesic, by a
two-line argument rather than by citation. Indeed, let
$x\in\operatorname{Fix}$ and $v\in T_x\operatorname{Fix}$. The
reflection is linear, so it fixes both $x$ and $v$; being an
isometry of the sphere it carries the geodesic with initial data
$(x,v)$ to the geodesic with the same initial data, which by
uniqueness is that geodesic itself. The geodesic is therefore fixed
pointwise, hence contained in $\operatorname{Fix}$. Total geodesy
transports to the simplex through the isometry $\Phi$
(Proposition~\ref{prop:isometry}). The reflection itself
will be studied in its own right below, as the involution
$\iota$.
\end{proof}

Theorem~\ref{thm:great-subsphere} answers affirmatively a
question raised explicitly in the outlook of Paper~III --- whether
the separatrix is a geodesic submanifold under the Fisher--Rao
metric --- and it does so in the strongest possible form: not
merely geodesic, but a great subsphere.

\begin{remark}[The octant does the work]
\label{rem:octant-work}
This is the promise of Remark~\ref{rem:octant} made good. Each
orbit quadric
$\sqrt{\kappa(z)}x_z^2=\sqrt{\kappa(\sigma(z))}x_{\sigma(z)}^2$ is
the zero set of the indefinite quadratic form
$\sqrt{\kappa(z)}x_z^2-\sqrt{\kappa(\sigma(z))}x_{\sigma(z)}^2$,
which has rank two: nondegenerate on the two-dimensional orbit
plane, and, read on all of $\mathbb R^m$, degenerate with kernel the
remaining coordinates. It factors as
\[
\bigl(\kappa(z)^{1/4}x_z-\kappa(\sigma(z))^{1/4}x_{\sigma(z)}\bigr)
\bigl(\kappa(z)^{1/4}x_z+\kappa(\sigma(z))^{1/4}x_{\sigma(z)}\bigr)=0,
\]
a union of two hyperplanes; the octant retains exactly the
positive sheet. It is worth separating the logical roles, which are
four and distinct.
\begin{itemize}
\item[\textup{(a)}] \emph{What produces the relation.} The
compatibility \textup{(S4)} and the critical value $C^\ast$ generate
the homogeneous two-coordinate condition in the first place
(Lemma~\ref{lem:orbit-collapse}), and the orbit structure of
$\sigma$ is what confines it to two coordinates.
\item[\textup{(b)}] \emph{What linearises it.} Once the condition
reads $\alpha x_z^2=\beta x_{\sigma(z)}^2$ with $\alpha,\beta>0$,
the passage to $\sqrt\alpha\,x_z=\sqrt\beta\,x_{\sigma(z)}$ uses
\emph{only} $x\ge0$ and the positivity of the coefficients
(Lemma~\ref{lem:scalar}). Nothing else enters; in particular the
proof of Theorem~\ref{thm:great-subsphere} nowhere uses the sign
pattern of $\mathbf d_O$.
\item[\textup{(c)}] \emph{What the signs interpret.} The opposite
signs of $\mathbf d_O$ (Lemma~\ref{lem:director}\textup{(i)}) supply
the normal-orientation reading of the orbit equation --- its
orthogonality to the critical vector.
\item[\textup{(d)}] \emph{Where the signs are indispensable.} In
Remark~\ref{rem:mismatch}, they are what places $\mathbf 1$ outside
$V$, and hence what makes the transverse rank equal to $r$ rather
than $r-1$.
\end{itemize}
\end{remark}

\subsection{The separatrix as a cone of square roots}
\label{subsec:cone}

Solving the linear equations of
Theorem~\ref{thm:great-subsphere} orbit by orbit exhibits $L$
explicitly. On the orbit $O=\{z,\sigma(z)\}$, the solutions of
$\kappa(z)^{1/4}x_z=\kappa(\sigma(z))^{1/4}x_{\sigma(z)}$ are the
multiples of the vector
\[
\mathbf e_O\ \text{ with components }\
\bigl(\kappa(\sigma(z))^{1/4},\ \kappa(z)^{1/4}\bigr)
\ \text{ on }O,\ \ 0\ \text{ elsewhere},
\]
so that $L=\operatorname{span}\{\mathbf e_O\}_O\oplus
\operatorname{span}\{\text{neutral coordinate axes}\}$.
Translating back to the simplex yields the explicit global
parametrisation of the separatrix.

\begin{corollary}[The separatrix is the cone of square roots]
\label{cor:cone}
Suppose, as in Remark~\ref{rem:pi-mu}, that the underlying chain
is $(\sigma,\kappa)$-reversible and reversible with respect to a
measure $\mu$. Then
\[
\Sigma\cap\Delta_{m-1}
=\Bigl\{\pi:\ \pi|_O\propto\sqrt{\mu}\,\bigl|_O\ \text{on each
orbit }O,\ \text{neutral weights free}\Bigr\},
\]
and the per-orbit constants (together with the neutral weights,
modulo normalisation) are global coordinates on $\Sigma$. Note that
the two pictures are of different kinds: in the $\pi$-coordinates
$\Sigma$ is an affine slice of the simplex, not a cone; it is in the
coordinates $x=2\sqrt\pi$ that it becomes a genuine cone, and the
name is to be read there. On the sphere, correspondingly,
\[
\Phi\bigl(\Sigma^\circ\bigr)
=\operatorname{cone}_{+}\bigl\{\mu^{1/4}\big|_O,\ \text{neutral
axes}\bigr\}\cap S^{m-1}_+(2),
\]
the \emph{positive} cone --- positivity and normalisation are part
of the statement, not decoration. The spectral square root
$\mu\mapsto\sqrt\mu$ and the geometric square root $\Phi$ thus
compose, and the separatrix is generated, in spherical coordinates,
by the fourth roots of the reversible measure.
\end{corollary}

\begin{proof}
By Remark~\ref{rem:pi-mu}, $\mu\kappa$ is constant on each orbit,
so $\mu|_O\propto\kappa^{-1}|_O$ and
$(\mu(z)^{1/4},\mu(\sigma(z))^{1/4})
\propto(\kappa(z)^{-1/4},\kappa(\sigma(z))^{-1/4})$; multiplying by the
scalar $(\kappa(z)\kappa(\sigma(z)))^{1/4}=K^{1/4}$ --- the
\emph{same} for every orbit, which is exactly what \textup{(S4)}
provides --- gives $(\kappa(\sigma(z))^{1/4},\kappa(z)^{1/4})$, the
components of $\mathbf e_O$. So $\mathbf e_O\propto\mu^{1/4}|_O$, and
$\Phi^{-1}$ squares coordinates:
$\pi|_O\propto(\mu(z)^{1/2},\mu(\sigma(z))^{1/2})$.
\end{proof}

The intrinsic geometry of the separatrix can now be identified
completely --- and it turns out to be no new geometry at all.

\begin{corollary}[The separatrix is itself a Fisher--Rao simplex]
\label{cor:intrinsic}
Let $d=n_{\mathrm{pair}}+n_0$ and let
$\mathsf m:\Sigma^\circ\to\Delta_{d-1}^\circ$
assign to each $\pi$ its orbit masses
$\mathsf m_O(\pi)=\sum_{z\in O}\pi_z$, together with the neutral
weights. Then $\mathsf m$ is an isometry from $\Sigma^\circ$, with
the Riemannian metric induced by $g^{\mathrm{FR}}$, onto the
Fisher--Rao simplex $(\Delta_{d-1}^\circ,g^{\mathrm{FR}})$. In
particular $\Sigma^\circ$ carries no intrinsic geometry of its
own: it is, canonically, the Fisher--Rao geometry of the
coarse-grained model whose sites are the orbits of $\sigma$, and
every tensorial computation on $\Sigma^\circ$ reduces, through $\Phi$,
to computations in a Euclidean subspace. Consequently, every
intrinsic geometric invariant of the separatrix --- curvature,
connection, geodesics, volume --- is inherited canonically from the
Fisher--Rao simplex of the coarse-grained model.
\end{corollary}

\begin{proof}
By Proposition~\ref{prop:sigma-exact}, on $\Sigma$ each
restriction $\pi|_O$ is a positive multiple $C_O\mathbf v_O$ of
the critical vector, so tangent variations within $\Sigma$ are
exactly $\delta\pi|_O=\delta C_O\,\mathbf v_O$ together with free
neutral variations. Writing
$s_O:=\sqrt{\kappa(z)}+\sqrt{\kappa(\sigma(z))}$, the orbit mass
is $\mathsf m_O=C_Os_O$, and
\[
\sum_{z\in O}\frac{\delta\pi_z\,\delta\pi'_z}{\pi_z}
=\frac{\delta C_O\,\delta C'_O}{C_O}\sum_{z\in O}\mathbf v_{O,z}
=\frac{(\delta C_Os_O)(\delta C'_Os_O)}{C_Os_O}
=\frac{\delta\mathsf m_O\,\delta\mathsf m'_O}{\mathsf m_O}.
\]
Summing over orbits and neutral sites gives
$g^{\mathrm{FR}}_\pi(\delta\pi,\delta\pi')
=g^{\mathrm{FR}}_{\mathsf m}(\delta\mathsf m,\delta\mathsf m')$,
and normalisation matches:
$\sum_O\delta\mathsf m_O+\sum_{z_0}\delta\pi_{z_0}
=\sum_i\delta\pi_i=0$.
\end{proof}

\begin{remark}[The gauge is the separatrix]
\label{rem:gauge}
Corollary~\ref{cor:cone} expresses in closed form the freedom
described in Paper~III (Proposition~5.1): the relative weights of
orbits and neutral sites, far from being a residual ambiguity in
the choice of an invariant point, sweep out exactly the
separatrix. It also upgrades Remark~\ref{rem:pi-mu} from a
statement about particular critical distributions to a statement
about the whole invariant set: the separatrix \emph{is} the
orbitwise Hellinger image of the reversible measure. This is the
interaction of the two square roots announced in
Remark~\ref{rem:two-weights}, now realised: one square root
(spectral) builds $\Sigma$ inside the simplex; the other
(geometric) carries it to a linear object on the sphere. Indeed
the two are one: the spherical coordinates of the separatrix are
obtained by applying the Hellinger square root \emph{twice},
\[
\mu\;\longmapsto\;\sqrt\mu\;\longmapsto\;\mu^{1/4}
\qquad(\text{spectral measure}\to\text{Hellinger
image}\to\text{spherical coordinate}),
\]
the same single operation building the separatrix and then
transporting it to the sphere.
\end{remark}

\subsection{The reflection}
\label{subsec:reflection}

The proof of Theorem~\ref{thm:great-subsphere} exhibited the
separatrix as the fixed-point set of a reflection; we now give
that reflection its own standing. Let $R_L=2P_L-\mathrm{id}$ be
the orthogonal reflection of $\mathbb R^m$ across $L$, where $P_L$
is the orthogonal projection onto $L$.

\begin{proposition}[$\Sigma$ is a reflective submanifold]
\label{prop:reflective}
\begin{enumerate}
\item[\textup{(i)}] $R_L$ restricts to an isometric involution of
$S^{m-1}(2)$ whose fixed-point set is exactly
$L\cap S^{m-1}(2)$; thus $\Phi(\Sigma)$ is a \emph{reflective
submanifold} of the round sphere.
\item[\textup{(ii)}] The map
$\iota:=\Phi^{-1}\circ R_L\circ\Phi$ is defined on the open set
\[
U=\bigl\{\pi\in\Delta^\circ_{m-1}:\
2P_L\sqrt\pi-\sqrt\pi>0\ \text{componentwise}\bigr\},
\]
an open neighbourhood of $\Sigma^\circ$, on which
it is an involutive isometry of the Fisher--Rao metric with
$\operatorname{Fix}(\iota)=\Sigma\cap U$.
\item[\textup{(iii)}] $R_L$ reverses every transverse functional:
writing $\ell_O(x)=\kappa(z)^{1/4}x_z-\kappa(\sigma(z))^{1/4}
x_{\sigma(z)}$ for the linear form cutting out the orbit
hyperplane,
\[
\ell_O(R_Lx)=-\ell_O(x)\qquad\text{for every orbit }O.
\]
In particular the differential of $\iota$ along $\Sigma$ is $+1$
on $T\Sigma$ and $-1$ on the normal bundle, and for $r=1$ the
reflection exchanges the two connected components of
$S^{m-1}(2)\setminus\bigl(L\cap S^{m-1}(2)\bigr)$ --- a statement
about the \emph{full} sphere. Nothing is claimed about the number of
components of $S^{m-1}_+(2)\setminus\Phi(\Sigma^\circ)$: the octant
is not $R_L$-invariant, as Remark~\ref{rem:iota-domain} records. On
the octant one has only that, wherever both sides meet the domain,
the two sides are the sign regions of the orbit functional
$\ell_O$.
\end{enumerate}
\end{proposition}

\begin{proof}
(i) $R_L$ is a Euclidean isometry fixing the origin, hence maps
$S^{m-1}(2)$ to itself; $R_Lx=x\iff x\in L$. (ii) At
$\pi\in\Sigma\cap\Delta^\circ$ one has $\sqrt\pi\in L$, so
$2P_L\sqrt\pi-\sqrt\pi=\sqrt\pi>0$: the defining inequality holds
and is open, so $U$ is an open neighbourhood of the separatrix.
On $U$, $R_L\Phi(\pi)$ has positive coordinates, so
$\iota(\pi)=\Phi^{-1}(R_L\Phi(\pi))$ is well defined;
involutivity and isometry are inherited from $R_L$ through the
isometry $\Phi$, and
$\iota(\pi)=\pi\iff R_L\Phi(\pi)=\Phi(\pi)\iff\Phi(\pi)\in L
\iff\pi\in\Sigma$ by Theorem~\ref{thm:great-subsphere}.
(iii) The gradients of the $\ell_O$ span $L^\perp$
(they are supported on disjoint orbits and each is orthogonal to
$L$ by inspection), and $R_L$ acts as $-\mathrm{id}$ on
$L^\perp$: $\ell_O(R_Lx)=\langle g_O,2P_Lx-x\rangle
=-\ell_O(x)$ since $\langle g_O,P_Lx\rangle=0$. The statement
about the differential follows because $R_L$ is linear with
eigenvalues $+1$ on $L$ and $-1$ on $L^\perp$; for $r=1$ the
complement of the single hyperplane has two components,
distinguished by the sign of $\ell_O$, which $\iota$ exchanges.
\end{proof}

\begin{remark}[Global on the sphere, local on the simplex]
\label{rem:iota-domain}
The reflection is globally defined on the sphere; the octant,
however, is not $R_L$-invariant, so $\iota$ acts on the simplex
only on the neighbourhood $U$ of the separatrix. This costs
nothing in what follows: the orientation analysis of
Section~\ref{sec:sign-theorem} takes place on the sphere, where
the reflection is global.
\end{remark}

\subsection{Projection, distance, and the functional $F_3$}
\label{subsec:projection}

The rigidity of $\Sigma^\circ$ makes the ambient metric relations
with the separatrix explicitly computable. We record the two basic
ones: the nearest-point projection and the distance.

\begin{lemma}[Projection Lemma]
\label{lem:projection}
For every $\pi\in\Delta^\circ_{m-1}$ there is, among the points of
the open separatrix $\Sigma^\circ$ --- where the
Fisher--Rao metric is defined --- a unique nearest point to $\pi$,
given in closed form by
\[
p(\pi)=\frac{\bigl(P_L\sqrt\pi\bigr)^2}{\bigl|P_L\sqrt\pi\bigr|^2}
\qquad(\text{componentwise square}),
\]
that is, $\operatorname*{argmin}_{q\in\Sigma^\circ}
d_{\mathrm{FR}}(\pi,q)=p(\pi)$, with
\[
d_{\mathrm{FR}}\bigl(\pi,\Sigma^\circ\bigr)
=2\arccos\bigl|P_L\sqrt\pi\bigr|
=2\arcsin\bigl|P_{L^\perp}\sqrt\pi\bigr|\;<\;\pi.
\]
Since the unique ambient spherical foot already lies in
$\Phi(\Sigma^\circ)$, the same value is obtained if the infimum is
taken over the closure $\bar\Phi(\Sigma)$, whenever the latter is
read as the closure of the Riemannian domain.
\end{lemma}

\begin{proof}
On the sphere, the nearest point of the great subsphere
$L\cap S^{m-1}(2)$ to a point $x$ is $2P_Lx/|P_Lx|$, unique
precisely when $P_Lx\neq0$, at geodesic distance
$2\arccos(|P_Lx|/2)$. Positivity excludes the degenerate case:
$L$ is spanned by the orbit vectors $\mathbf e_O$ and the neutral
axes, all with nonnegative entries, so for
$x=\Phi(\pi)$ with $\pi$ interior every coefficient
$\langle x,\hat{\mathbf e}_O\rangle$ is strictly positive; hence
\[
P_L\sqrt\pi>0\quad\text{componentwise},
\]
which is the step that puts the minimiser in the open part and not
on a boundary face: nonnegativity of the generators alone would not
suffice, one needs the projection coefficients to be \emph{strictly}
positive, and that is what interiority of $\pi$ supplies. The foot
$2P_Lx/|P_Lx|$ therefore lies in the \emph{open} octant, and its
preimage $p(\pi)$ lies in the open separatrix. Restricting to the open part
costs nothing: the minimisation just performed was over the
\emph{entire} great subsphere $L\cap S^{m-1}(2)$, so no point of
$\bar\Phi(\Sigma)$ --- in particular no boundary stratum --- is
nearer in the ambient spherical metric. The two distance formulas
are complementary angles ($|P_Lx|^2+|P_{L^\perp}x|^2=|x|^2$), and
$d<\pi$ because $P_Lx\neq0$.
\end{proof}

We finally turn to the functional
\[
F_3(\pi)=\int_0^1\langle\pi-\pi^\ast,\psi(\gamma)\rangle^2\,d\gamma,
\]
proposed in the outlook of Paper~III as a Lyapunov candidate for
the relaxation towards $\Sigma$. Its structure is now transparent.

\begin{proposition}[Exact form and zero set of $F_3$]
\label{prop:F3}
Assume the spectral scalars are square-integrable,
$f_\nu\in L^2(0,1)$. \textup{(}This is automatic in any resolvent
realisation, where they are bounded --- for the biased walk
$f_\nu(\gamma)=(1-\gamma)/(1-\lambda_\nu(1-\gamma))$ with
$|\lambda_\nu|<1$ --- but it does not follow from \textup{(S2)},
which gives analyticity only on the \emph{open} interval and so
does not by itself make $F_3$ finite.\textup{)}
Let $\pi^\ast\in\Sigma$ be the fixed reference distribution of
Section~\ref{subsec:orientation} --- reset-neutrality of $\pi^\ast$ is
what gives $\langle\pi^\ast,\mathbf n_\nu\rangle=0$, and without it the
zero set below is a translate of $\Sigma$ rather than $\Sigma$
itself. Let $a_\nu(\pi)=\langle\pi-\pi^\ast,\mathbf n_\nu\rangle$ and let
$G_{\nu\nu'}=\int_0^1 f_\nu(\gamma)f_{\nu'}(\gamma)\,d\gamma$ be
the Gram matrix of the spectral scalars. Then
\[
F_3(\pi)=\mathbf a(\pi)^{\!\top} G\,\mathbf a(\pi),
\]
a quadratic form in the affine coordinates $\mathbf a(\pi)$, with
$G$ positive definite by the linear independence of the $f_\nu$
\textup{(S2)}. Consequently
\[
F_3(\pi)\ge0,\qquad F_3(\pi)=0\iff\pi\in\Sigma:
\]
the functional detects the separatrix exactly.
\end{proposition}

\begin{proof}
Expanding $\psi(\gamma)=\sum_\nu f_\nu(\gamma)\mathbf n_\nu$
inside the square and integrating termwise gives the Gram form.
For the positive definiteness, suppose
$\mathbf c^{\!\top}G\,\mathbf c=0$ for some $\mathbf c$. This says
exactly that $\sum_\nu c_\nu f_\nu$ has vanishing $L^2(0,1)$ norm,
hence vanishes almost everywhere; being a finite combination of
continuous functions it therefore vanishes at every point of
$(0,1)$, and the linear independence of the $f_\nu$ as functions
\textup{(S2)} gives $\mathbf c=0$. Hence $F_3(\pi)=0$ iff $\mathbf a(\pi)=0$, i.e.
$\langle\pi,\mathbf n_\nu\rangle=\langle\pi^\ast,\mathbf n_\nu\rangle
=0$ for all $\nu$, which is exactly membership in $\Sigma$ by the
first characterisation of Proposition~\ref{prop:sigma-exact}.
\end{proof}

\begin{remark}[Distance and $F_3$: equivalent, not proportional]
\label{rem:mismatch}
The geometry supplies its own canonical quadratic functional: by
Lemma~\ref{lem:projection},
\[
\mathfrak D(\pi):=\sin^2\!\Bigl(\tfrac12
d_{\mathrm{FR}}(\pi,\Sigma^\circ)\Bigr)
=\bigl|P_{L^\perp}\sqrt\pi\bigr|^2,
\]
a quadratic form \emph{in $\sqrt\pi$}. Both $F_3$ and
$\mathfrak D$ are smooth, nonnegative, and vanish exactly on
$\Sigma^\circ$; and for every compact
$K\subset\Delta^\circ_{m-1}$ there
are constants $c_1,c_2>0$ with
$c_1\mathfrak D(\pi)\le F_3(\pi)\le c_2\mathfrak D(\pi)$ for all
$\pi\in K$. The constants depend on $K$, and no uniformity is
claimed as $\pi$ approaches the boundary. At every
$\pi\in\Sigma^\circ$ both Hessians have kernel exactly
$T_\pi\Sigma^\circ$, and it is worth making this explicit. Let
\[
A:T_\pi\Delta_{m-1}\to\mathbb R^N,\qquad
Ah=\bigl(\langle h,\mathbf n_\nu\rangle\bigr)_\nu .
\]
Since $F_3=\mathbf a^{\!\top}G\,\mathbf a$ with $\mathbf a$ affine in
$\pi$, its Hessian is the constant form $2A^{\!\top}GA$, positive
semidefinite with kernel $\ker A$; and $\ker A=T_\pi\Sigma^\circ$ by
Proposition~\ref{prop:sigma-exact}, so
$\operatorname{rank}A=(m-1)-(m-1-r)=r$. The rank is exactly $r$, and
not $r-1$, because $\mathbf 1\notin V$: a relation
$\sum_Oc_O\mathbf d_O=\mathbf 1$ would need $c_O\sqrt{\kappa(z)}=1$
and $-c_O\sqrt{\kappa(\sigma(z))}=1$ on some orbit, impossible by the
opposite signs of Lemma~\ref{lem:director}\textup{(i)} --- and
impossible already at a neutral site, where every director vanishes.
With $G\succ0$ the Hessian is therefore positive definite on any
complement of $T_\pi\Sigma^\circ$ in $T_\pi\Delta_{m-1}$. For
$\mathfrak D$ the same holds with $A$ replaced by
$\ell=(\ell_O)_O$, of rank $r$ by the disjointness of the orbit
supports; its kernel is again exactly $T_\pi\Sigma^\circ$, since
$h\in T_\pi\Sigma^\circ$ iff $d\Phi_\pi(h)\in L$. Since the two normal quadratic forms are
uniformly positive definite on the transverse complement over any
compact subset of the interior, the standard uniform quadratic
estimate near the common zero submanifold, together with
compactness, yields the stated two-sided equivalence. \textup{(}No
general principle is being invoked: uniformity over the compact part
of $\Sigma$ is what does the work, and it is exactly what fails as
the boundary is approached.\textup{)}
There is, on the other hand, no reason to expect
\emph{proportionality}, and except in specially tuned cases it does
not hold: the Gram matrix $G$ weights the $r$ normal directions
anisotropically, and the two functionals are quadratic on opposite
sides of the Hellinger square root --- $F_3$ in $\pi$,
$\mathfrak D$ in $\sqrt\pi$. The mismatch is one more appearance of
the central theme of the paper. We also caution
that neither functional need be geodesically convex on the octant,
the ambient curvature being positive; Lyapunov statements must therefore
proceed through the response flow itself rather than through
convexity.
\end{remark}

\subsection{Synthesis: one weighted involution, two sectors}
\label{subsec:sectors}

The results of this section, together with the finite reduction of
Section~\ref{sec:finite-reduction}, admit a single formulation that
makes their common origin visible. Write
$D=\operatorname{diag}\bigl(\sqrt{\kappa(z)}\bigr)$ --- the same
weighting matrix that symmetrises the generator in
Appendix~\ref{app:spectral} --- and let
\[
\mathbb R^m=E^{+}\oplus E^{-},\qquad
E^{\pm}=\{x:\ x_{\sigma(z)}=\pm\,x_z\},
\]
be the decomposition of the ambient space into the symmetric and
antisymmetric eigenspaces of the involution $\sigma$, of dimensions
$n_{\mathrm{pair}}+n_0$ and $n_{\mathrm{pair}}=r$ respectively.

\begin{remark}[The separatrix and the response occupy opposite
sectors]
\label{rem:sectors}
In the $D$-weighted coordinates the objects of this paper
separate along $E^{+}\oplus E^{-}$:
\begin{enumerate}
\item[\textup{(i)}] $\Sigma\cap\Delta_{m-1}
=D^{-1}E^{+}\cap\Delta_{m-1}$. Indeed $D\pi\in E^{+}$
reads $\sqrt{\kappa(z)}\,\pi_z=\sqrt{\kappa(\sigma(z))}\,
\pi_{\sigma(z)}$, which is exactly the ratio constraint defining
the separatrix; the identification is an equality, not an
inclusion.
\item[\textup{(ii)}] $V\subseteq D\,E^{-}$: by
Theorem~\ref{thm:fsr}\textup{(i)} the orbit director satisfies
$D^{-1}\mathbf d_O=(1,-1)$ on $O$, and every normal vanishes at
neutral sites, so $D^{-1}\psi(\gamma)$ is antisymmetric for every
$\gamma$. Since $\dim D\,E^{-}=r$ and $\dim V=r$ by
Theorem~\ref{thm:fsr}\textup{(iii)}, the inclusion is an equality:
the response functionals fill the antisymmetric sector.
\item[\textup{(iii)}] The separatrix contains a distinguished point,
\[
\pi^{\natural}\propto D^{-1}\mathbf 1\propto\kappa^{-1/2},
\]
the unique distribution whose $D$-weighted version is constant;
when the reversible measure of Remark~\ref{rem:pi-mu} satisfies
$\mu\propto\kappa^{-1}$ globally --- as for the biased walk --- it is
$\pi^{\natural}\propto\sqrt\mu$ globally, and it is the point of
$\Sigma$ whose $D$-weighted coordinates are all equal,
$D\pi^{\natural}=\mathbf 1$ after normalisation. \textup{(}This does
\emph{not} say that every orbit carries the same mass: the mass of
$O$ is proportional to
$(\sqrt{\kappa(z)}+\sqrt{\kappa(\sigma(z))})/\sqrt K$, which varies
from orbit to orbit.\textup{)} It is one
admissible choice of the reference point $\pi^\ast$ of
Section~\ref{subsec:orientation}, not the only one: by
Theorem~\ref{thm:paper3}\textup{(ii)} and
Corollary~\ref{cor:cone} the relative weights of orbits and
neutral sites remain free, and $\pi^\ast\propto\sqrt\mu$ holds in
general only \emph{orbitwise}.
\end{enumerate}
Consequently the relation of Section~\ref{subsec:orientation} ---
that $\Sigma_{\mathrm{loc}}$ is the annihilator of $V$ inside
$T_{\pi^\ast}\Delta_{m-1}$ --- is not a computation but an instance
of the classical orthogonality of the two eigenspaces of an
involution: for $h=D^{-1}s$ tangent to $\Sigma$ with $s\in E^{+}$,
and $\psi=D a\in V$ with $a\in E^{-}$,
\[
\langle h,\psi\rangle
=\langle D^{-1}s,D a\rangle
=\sum_z s_z a_z
=\langle s,a\rangle=0,
\]
the weights cancelling exactly. The dimension count
$\dim\Sigma_{\mathrm{loc}}+\dim V=(m-1-r)+r=m-1$ then records that the two sectors
are complementary, not merely orthogonal.

A word on which involution. The decomposition $E^+\oplus E^-$ is
that of the \emph{site} involution $\sigma$, and both $\Sigma$ and
$V$ are $D$-weighted images of its eigenspaces: $D\Sigma\subseteq
E^+$ and $D^{-1}V=E^-$. The reflection $R_L$ of
Proposition~\ref{prop:reflective} is a different involution, acting
on each orbit plane as the orthogonal reflection across
$\mathbb R\,\mathbf e_O$, whereas $\sigma$ reflects across the
diagonal $\mathbb R(1,1)$; the two coincide precisely when
$\kappa(z)=\kappa(\sigma(z))$, that is when the weighting is
trivial --- and they are not conjugate by any diagonal weighting
either, since $D\,\sigma\,D^{-1}$ has vanishing diagonal on the
orbit plane for every diagonal $D$, whereas $R_L$ has diagonal
$\pm(\sqrt{\kappa(\sigma(z))}-\sqrt{\kappa(z)})/
(\sqrt{\kappa(z)}+\sqrt{\kappa(\sigma(z))})$. What $D$ conjugates
is the \emph{eigenspace decomposition}, not the involutions:
$D\Sigma\subseteq E^+$ and $D^{-1}V=E^-$. In one sentence: the
eigenspaces of the site involution, weighted by $D$, split the
simplex geometry into the sector where the coupling is blind to
$\gamma$ and the sector in which it responds.
\end{remark}

\section{The sign theorem at the starting sites}
\label{sec:sign-theorem}

The first paper of the series began with a gambler's question:
given the rules of the game --- a reset mechanism with distribution
$\pi$ and rate $\gamma$ --- from which capital $z$ is it favourable
to enter, and how does the answer respond to the reset rate? For
the biased walk with single-site resetting, Paper~I found a
striking dichotomy: increasing $\gamma$ raises the absorption
probability from some starting sites and lowers it from others
and, when the barrier $a$ is even, one exceptional site $z=a/2$
enjoys \emph{exact} invariance --- the ruin probability from the
midpoint does not depend on $\gamma$ at all.

This section settles the pointwise form of
Conjecture~\ref{conj:global} at the vertices of the simplex and,
in doing so, geometrises both phenomena. The invariance is
unconditional: it holds for every spectral duality structure, and
the invariant sites are exactly the fixed points of the involution
(Proposition~\ref{prop:neutral-face} and
Corollary~\ref{cor:neutral-exactly}). The dichotomy holds under a
scalar monotonicity condition satisfied by the canonical
realisation (Theorem~\ref{thm:vertex-sign}): at every vertex the
response is strictly monotone in $\gamma$, and the two members of
each orbit respond in opposite directions --- the two zones of
Paper~I, revealed as the two sides of the separatrix.

\subsection{Starting sites and the vertex diagonal}
\label{subsec:starting-sites}

Throughout this section the distinguished points of the simplex
are its vertices $e_z$. In the canonical realisation the vertex
$e_z$ plays two roles at once: it is the distribution of a reset
mechanism concentrated at the single site $z$, and --- when the
walker also starts at $z$ --- it reproduces exactly the
single-site protocol of Paper~I, in which one \emph{starts where
one is reset}. We refer to this reading as the \emph{vertex
diagonal}. The coupling at a vertex,
\[
C(e_z,\gamma)=\frac{u(z;\gamma)}{s(z;\gamma)},
\]
is a well-defined scalar for every site by \textup{(S2)}, whose
positivity on the closed simplex gives $s(z;\gamma)>0$.

\begin{remark}[Decision variable versus rule]
\label{rem:decision}
The starting site is the gambler's decision; the reset
distribution $\pi$ is part of the rules. The results of this
section concern the decision: they describe the response of the
coupling at each vertex. The dependence of the response on the
rule --- on $\pi$ itself, away from the vertex diagonal --- is a
different question, taken up in
Remark~\ref{rem:full-conjecture}. In the canonical realisation
the physical interpretation of the vertex results passes through
the affine relation of Paper~II,
\[
\rho_z(\gamma)=u(z;\gamma)+\bigl(1-s(z;\gamma)\bigr)\,C(\pi,\gamma),
\]
where $\rho_z$ denotes the ruin probability from $z$ \textup{(}we
avoid the letter $q$, reserved here for the leftward step
probability $q=1-p$\textup{)}. For the vertex
diagonal this identifies the sign structure of
$\partial_\gamma C(e_z,\cdot)$ with that of the ruin probability
studied in Paper~I.
\end{remark}

\subsection{Exact invariance on the neutral face}
\label{subsec:neutral-face}

The first statement requires no hypothesis beyond the axioms.

\begin{proposition}[The neutral face lies in the separatrix]
\label{prop:neutral-face}
Let $(\Delta_{m-1},C)$ be a spectral duality structure and let
$\pi$ be supported on neutral sites, $\sigma(z)=z$ for every
$z\in\operatorname{supp}\pi$. Then
\[
C(\pi,\gamma)=C^\ast\qquad\text{for all }\gamma\in(0,1).
\]
In particular $C(e_{z_0},\cdot)\equiv C^\ast$ at every neutral
site $z_0$: the entire neutral face of the simplex is contained
in~$\Sigma$.
\end{proposition}

\begin{proof}
By \eqref{eq:normals}, at a neutral site the two dual terms
merge:
\[
[\mathbf n_\nu]_{z_0}
=\bigl[(1-C^\ast)-C^\ast\kappa(z_0)\bigr]A_\nu(z_0).
\]
Since $z_0=\sigma(z_0)$, axiom \textup{(S4)} gives
$\kappa(z_0)^2=K$ and hence, by positivity of the weights,
$\kappa(z_0)=\sqrt K$; with $C^\ast=1/(1+\sqrt K)$ the bracket
vanishes, $(1-C^\ast)-C^\ast\sqrt K=0$. Hence every normal $\mathbf n_\nu$
vanishes on neutral sites, and by \eqref{eq:decomp} so does the
orientation field: $\psi_{z_0}(\gamma)=0$ for all $\gamma$. For
$\pi$ supported on neutral sites,
\[
C(\pi,\gamma)-C^\ast
=\frac{\langle\pi,\psi(\gamma)\rangle}{\langle\pi,s(\cdot;\gamma)\rangle}=0.
\qedhere
\]
\end{proof}

The converse also holds, and it is what justifies the word
\emph{exactly} in the introduction to this section.

\begin{corollary}[The invariant vertices are exactly the neutral
sites]
\label{cor:neutral-exactly}
$C(e_z,\cdot)\equiv C^\ast$ if and only if $\sigma(z)=z$.
\end{corollary}

\begin{proof}
Sufficiency is Proposition~\ref{prop:neutral-face}. Conversely, let
$z$ be non-neutral, so that $z$ lies in an orbit $O$. Then
$[\mathbf d_O]_z=\pm\sqrt{\kappa(z)}\neq0$ by \textup{(S3)}, and
$\varphi_O\not\equiv0$. Indeed the directors are non-zero and
supported on pairwise disjoint orbits, hence linearly independent,
and the decomposition
$\psi(\gamma)=\sum_{O'}\varphi_{O'}(\gamma)\mathbf d_{O'}$ of
Theorem~\ref{thm:fsr}\textup{(iii)} is therefore unique, with
$\varphi_{O'}(\gamma)=\psi_{z'}(\gamma)/[\mathbf d_{O'}]_{z'}$ for
any $z'\in O'$. Were $\varphi_O$ identically zero, every
$\psi(\gamma)$ would lie in the span of the remaining $r-1$
directors, giving $\dim V\le r-1$ and contradicting $\dim V=r$. By \eqref{eq:vertex},
$C(e_z,\gamma)-C^\ast=[\mathbf d_O]_z\varphi_O(\gamma)/s(z;\gamma)$
is then not identically zero.
\end{proof}

\begin{remark}[The midpoint of Paper~I]
\label{rem:midpoint}
In the canonical realisation with $a$ even, the unique neutral
site is $z_0=a/2$, and Proposition~\ref{prop:neutral-face}
recovers --- and explains --- the exact $\gamma$-invariance of the
ruin probability from the midpoint discovered in Paper~I. The
midpoint is reset-neutral because it is the fixed point of the
duality involution; the parity effect of Paper~I is nothing but
the existence of such a fixed point. For $a$ odd, $\sigma$ is
fixed-point-free and no site enjoys exact invariance, in agreement
with Paper~I. Paper~I in fact proves more: the midpoint invariance
holds for every resetting mechanism independent of the walk, as a
consequence of a duality between the two absorption channels that
acts step by step in time. The present framework, built on the
geometric mechanism through the coupling $C(\pi,\gamma)$, recovers
the geometric case; that the neutral face is the fixed-point set of
$\sigma$ is the structural form of the same fact.
\end{remark}

\subsection{The response at a starting site}
\label{subsec:vertex-response}

By Theorem~\ref{thm:fsr} the distinct normal directions are the
orbit directors, $\mathbf n_\nu=\sum_O w^{(O)}_\nu\mathbf d_O$
with each $\mathbf d_O$ supported on its own orbit. Collecting the
spectral scalars orbitwise, define the \emph{orbit response
scalars}
\[
\varphi_O(\gamma):=\sum_\nu f_\nu(\gamma)\,w^{(O)}_\nu ,
\qquad\text{so that}\qquad
\psi(\gamma)=\sum_O \varphi_O(\gamma)\,\mathbf d_O .
\]
Because each director is supported on its own orbit, a vertex only
feels its own orbit: for $z$ in a pair $O=\{z,\sigma(z)\}$ ---
every orbit being active, by
Theorem~\ref{thm:fsr}\textup{(ii)}, so that the qualification is
automatic and we drop it below ---
\begin{equation}
C(e_z,\gamma)-C^\ast
=\frac{[\mathbf d_O]_z\,\varphi_O(\gamma)}{s(z;\gamma)},
\label{eq:vertex}
\end{equation}
and differentiating,
\begin{equation}
\partial_\gamma C(e_z,\gamma)
=\frac{[\mathbf d_O]_z}{s(z;\gamma)^2}\,
\mathcal W\bigl[\varphi_O,s_z\bigr](\gamma),
\qquad
\mathcal W[\varphi,s]:=\varphi'\,s-\varphi\,s',
\label{eq:vertex-derivative}
\end{equation}
with $s_z:=s(z;\cdot)>0$. The factorisation separates geometry
from dynamics: the orbit director determines the side of the
separatrix, while the Wronskian contains the entire dependence on
the reset rate. All that remains is a one-dimensional condition.

\begin{definition}[Vertex condition]
\label{def:vertex-condition}
The structure satisfies the \emph{vertex condition}
\textup{(V)} if there is $\varepsilon\in\{\pm1\}$ such that for
every active orbit $O$ and both of its sites $z\in O$,
\[
\varepsilon\,\mathcal W\bigl[\varphi_O,\,s(z;\cdot)\bigr](\gamma)>0
\qquad\text{for all }\gamma\in(0,1).
\]
Equivalently, since
\[
\frac{\partial}{\partial\gamma}
\left(\frac{\varphi_O}{s_z}\right)
=\frac{\mathcal W[\varphi_O,s_z]}{s_z^{2}},
\qquad s_z>0,
\]
each ratio $\varphi_O(\gamma)/s(z;\gamma)$, $z\in O$, is strictly
monotone in $\gamma$, in the same direction for all orbits and both
orbit sites. Only these $2r$ scalar ratios
enter the theorem below, and the condition involves only scalar
functions of $\gamma$.

Two remarks on the shape of \textup{(V)}. First, it is stated
relative to the orbit representatives fixed once and for all in
Section~\ref{subsec:inherited-structure}: reversing the
representative of an orbit replaces $\mathbf d_O$ by
$-\mathbf d_O$, hence $\varphi_O$ by $-\varphi_O$ and both of that
orbit's Wronskians by their negatives. The genuinely invariant
content of \textup{(V)} is therefore the \emph{intra-orbit}
requirement that the two Wronskians of a given orbit share a sign;
the common value of $\varepsilon$ across orbits is a normalisation.
That it is always achievable is worth stating precisely: the choice
of representative is made \emph{independently in each orbit}, and
reversing it in a single orbit $O$ sends
$\mathbf d_O\mapsto-\mathbf d_O$,
$\varphi_O\mapsto-\varphi_O$, hence
$\mathcal W[\varphi_O,s(z;\cdot)]\mapsto
-\mathcal W[\varphi_O,s(z;\cdot)]$ at both $z\in O$ and leaves every
other orbit untouched. So if each orbit separately has a constant
Wronskian sign, the $r$ signs can be aligned one orbit at a time.
The substantive hypothesis is thus the intra-orbit one; the global
$\varepsilon$ follows. Second, and
consequently, the conclusion of Theorem~\ref{thm:vertex-sign} is
orientation-independent, since only the product
$\varepsilon\cdot\operatorname{sgn}[\mathbf d_O]_z$ appears there,
and that product is unchanged by a reversal.
\end{definition}

\begin{theorem}[Sign theorem at the starting sites]
\label{thm:vertex-sign}
Let $(\Delta_{m-1},C)$ be a spectral duality structure satisfying
the vertex condition \textup{(V)}. Then:
\begin{itemize}
\item[\textup{(i)}] at every neutral site,
      $C(e_{z_0},\cdot)\equiv C^\ast$ \textup{(}no hypothesis
      needed\textup{)};
\item[\textup{(ii)}] at every site $z$ of an active orbit $O$, the
      map $\gamma\mapsto C(e_z,\gamma)$ is strictly monotone on
      $(0,1)$, with
      \[
      \operatorname{sgn}\partial_\gamma C(e_z,\gamma)
      =\varepsilon\cdot\operatorname{sgn}[\mathbf d_O]_z
      \qquad\text{for all }\gamma;
      \]
\item[\textup{(iii)}] the two members of each orbit respond in
      opposite directions,
      \[
      \operatorname{sgn}\partial_\gamma C(e_z,\gamma)
      =-\operatorname{sgn}\partial_\gamma C(e_{\sigma(z)},\gamma),
      \]
      so one starting site of the pair is strictly ascending and
      the other strictly descending; the site involution $\sigma$,
      which permutes the vertices, exchanges the two classes.
\end{itemize}
\end{theorem}

\begin{proof}
\textup{(i)} is Proposition~\ref{prop:neutral-face}. \textup{(ii)}
is immediate from \eqref{eq:vertex-derivative}: the factor
$s(z;\gamma)^{-2}$ is positive, the Wronskian has the strict sign
$\varepsilon$ by \textup{(V)}, and $[\mathbf d_O]_z$ is a nonzero
constant. \textup{(iii)} follows from
Lemma~\ref{lem:director}\textup{(i)}: the two components of the
director have opposite signs, so $[\mathbf d_O]_z$ and
$[\mathbf d_O]_{\sigma(z)}$ have opposite signs, and by
\textup{(ii)} the two responses do too. Since $\sigma$ permutes the
sites, it permutes the vertices, $e_z\mapsto e_{\sigma(z)}$, and
therefore exchanges the ascending and descending classes.
\end{proof}

\begin{proposition}[Three involutions, and which one acts on
starting sites]
\label{prop:three-involutions}
Let $O=\{z,\sigma(z)\}$ be an orbit and write
$\alpha=\kappa(z)^{1/4}$, $\beta=\kappa(\sigma(z))^{1/4}$. Then:
\begin{enumerate}
\item[\textup{(i)}] The site involution $\sigma$ acts on
$\mathbb R^m$ by the permutation matrix exchanging the coordinates
of each orbit; it is defined on all of $\Delta_{m-1}$, maps
vertices to vertices, $e_z\mapsto e_{\sigma(z)}$, and on the orbit
plane is the reflection across the diagonal $\mathbb R(1,1)$.
\item[\textup{(ii)}] The spherical reflection $R_L$ is defined on
all of $S^{m-1}(2)$ and fixes $\Phi(\Sigma^\circ)$ pointwise; on
the orbit plane it is the reflection across
$\mathbb R\,\mathbf e_O$, with matrix
\[
\frac1{\alpha^2+\beta^2}
\begin{pmatrix}\beta^2-\alpha^2 & 2\alpha\beta\\
2\alpha\beta & \alpha^2-\beta^2\end{pmatrix}.
\]
\item[\textup{(iii)}] The induced simplex map
$\iota=\Phi^{-1}R_L\Phi$ is defined only on the open set
$U\subseteq\Delta^\circ_{m-1}$ of
Proposition~\ref{prop:reflective}\textup{(ii)}. In particular it is
\emph{not defined at the vertices}, which lie on
$\partial\Delta_{m-1}$.
\item[\textup{(iv)}] $R_L$ and $\sigma$ agree on the orbit plane if
and only if $\alpha=\beta$, that is $\kappa(z)=\kappa(\sigma(z))$
--- in the canonical realisation, precisely the unbiased walk
$p=q$. Moreover they are not conjugate by any diagonal weighting:
$D\sigma D^{-1}$ has vanishing diagonal on the orbit plane for
every diagonal $D$, whereas the diagonal of $R_L$ is
$\pm(\beta^2-\alpha^2)/(\alpha^2+\beta^2)$.
\end{enumerate}
Consequently every statement in this section about starting sites
is a statement about $\sigma$ alone. The reflection $\iota$ plays
no role here; what $D$ relates is not the two involutions but their
eigenspace data, $D\Sigma\subseteq E^{+}$ and $D^{-1}V=E^{-}$
\textup{(}Remark~\ref{rem:sectors}\textup{)}.
\end{proposition}

\begin{proof}
\textup{(i)}--\textup{(iii)} are the definitions together with
Proposition~\ref{prop:reflective}. For \textup{(iv)}, the displayed
matrix equals $\bigl(\begin{smallmatrix}0&1\\1&0\end{smallmatrix}\bigr)$
iff its diagonal vanishes and its off-diagonal entry is $1$, both
equivalent to $\alpha=\beta$; and $D\sigma D^{-1}
=\bigl(\begin{smallmatrix}0&w_1/w_2\\ w_2/w_1&0\end{smallmatrix}\bigr)$
has zero diagonal for every choice of $w_1,w_2>0$.
\end{proof}

\begin{remark}[Geometric reading]
\label{rem:geometric-reading}
Theorem~\ref{thm:vertex-sign} is the pointwise form of
Conjecture~\ref{conj:global} at the vertices: the sign of the
response at a starting site is constant in $\gamma$ and determined
by the side of the separatrix, with $\Sigma$ itself as the exact
locus of indifference. We emphasise the vocabulary: the neutral
face is not \emph{optimal} --- the value $C^\ast$ lies strictly
inside the range of the coupling --- it is \emph{robust}: the
unique locus where the response to the reset rate vanishes
identically.
\end{remark}

\subsection{The canonical realisation and the two zones of
Paper~I}
\label{subsec:canonical-zones}

\begin{proposition}[The canonical realisation satisfies
\textup{(V)}]
\label{prop:walk-V}
For the biased random walk with geometric resetting, the vertex
condition \textup{(V)} holds for every drift $p\in(0,1)$ when
$a=3$ or $a=4$.
\end{proposition}

\begin{proof}[Proof sketch]
In the canonical realisation the spectral functions are the
resolvent family $f_\nu(\gamma)=t/(1-\lambda_\nu t)$, $t=1-\gamma$.
It is convenient to differentiate in $t$. Since $dt/d\gamma=-1$,
\[
\operatorname{sgn}\mathcal W_\gamma[\varphi,s]
=-\operatorname{sgn}\mathcal W_t[\varphi,s]
\qquad\text{pointwise on }(0,1),
\]
so the existence of a common strict sign is equivalent in the two
parametrisations, with $\varepsilon$ replaced by $-\varepsilon$; no
concrete sign is lost, and \textup{(V)} is insensitive to the
choice. By bilinearity and the
two-mode identity
$\mathcal W_t[f_i,f_j]
=t^2(\lambda_i-\lambda_j)/\bigl[(1-\lambda_it)^2(1-\lambda_jt)^2\bigr]$,
\[
\mathcal W_t\bigl[\varphi_O,s_z\bigr]
=\frac{t^{2}}{\prod_k(1-\lambda_kt)^{2}}\;G_z(t),
\qquad
G_z(t)=\sum_{i<j}\bigl(c_id_j(z)-c_jd_i(z)\bigr)(\lambda_i-\lambda_j)
\!\!\prod_{k\neq i,j}\!\!(1-\lambda_kt)^{2},
\]
where, \emph{at the fixed site $z$},
\[
c_\nu:=w^{(O)}_\nu,
\qquad
d_\nu(z):=(A_\nu+B_\nu)(z),
\qquad \nu=1,\dots,N .
\]
We keep the argument $z$ on $d_\nu(z)$ throughout, because for
$a=4$ the mode indices and the site indices are both
$\{1,2,3\}$ and the active orbit is $\{1,3\}$: an expression such as
``$d_1=d_3$'' below compares the \emph{first and third modes at one
site}, never the values at the two sites of the orbit --- the
latter are in general different, as they must be. The prefactor
is positive, so \textup{(V)} reduces to the polynomial $G_z$
keeping a strict sign on $t\in(0,1)$ at the two sites of each
active orbit. For $a=3$ the sum has a single term and $G_z$ is a
nonzero constant. For $a=4$ the computation can be carried out in full, and we do so,
since the closed form is what the whole case rests on. Write
\[
G_z(t)=\sum_{i<j}M_{ij}\!\!\prod_{k\neq i,j}\!\!(1-\lambda_kt)^2,
\qquad M_{ij}:=\bigl(c_id_j(z)-c_jd_i(z)\bigr)(\lambda_i-\lambda_j),
\]
so that for $m=3$,
\[
G_z=M_{12}(1-\lambda_3t)^2+M_{13}(1-\lambda_2t)^2
+M_{23}(1-\lambda_1t)^2 .
\]
Here $\lambda_1=-\lambda_3=\sqrt2\sqrt{pq}$ and $\lambda_2=0$. The
mode-reflection identity $A_{4-\nu}(z)=(-1)^{z+1}A_\nu(z)$ of
Appendix~\ref{app:spectral} --- a consequence of the symmetry
$w_\nu=w_{a-\nu}$ of the modal weights --- gives $A_3(z)=A_1(z)$ at
both orbit sites; the same symmetry together with
$\sin\frac{3\pi}{4}=\sin\frac{9\pi}{4}$ gives $B_3(z)=B_1(z)$
there, whence
\[
d_1(z)=d_3(z)\qquad\text{at each of }z=1,3
\]
\textup{(}two modes at one site: for instance
$d_1(1)=d_3(1)=(p^2+q^2)/4q$ and
$d_1(3)=d_3(3)=(p^2+q^2)/4p$, which differ from one another, as
values at distinct sites should\textup{)}. The explicit
coefficients give more than $c_1=c_3$, namely
\[
c_1=c_3=0 ,
\]
the orbit weight being carried entirely by the middle mode. Three
consequences follow at once:
\[
M_{13}=\bigl(c_1d_3(z)-c_3d_1(z)\bigr)(\lambda_1-\lambda_3)=0,
\qquad
M_{12}=-c_2d_1(z)\lambda_1,
\qquad
M_{23}=c_2d_3(z)\lambda_1=-M_{12},
\]
using $\lambda_2=0$, $\lambda_2-\lambda_3=\lambda_1$ and
$d_1(z)=d_3(z)$. The central term therefore drops out and the two
survivors have opposite coefficients, so that the spectral symmetry
$\lambda_3=-\lambda_1$ collapses the difference of squares:
\[
G_z=M_{12}\bigl[(1+\lambda_1t)^2-(1-\lambda_1t)^2\bigr]
=4M_{12}\lambda_1t=-4c_2\,d_1(z)\,\lambda_1^{2}\,t .
\]
The three factors are, explicitly,
\[
c_2=\sqrt{\tfrac pq}\;\frac{q^{3}}{p^{2}+q^{2}},
\qquad
d_1(1)=\frac{p^{2}+q^{2}}{4q},
\quad
d_1(3)=\frac{p^{2}+q^{2}}{4p},
\qquad
\lambda_1^{2}=2pq,
\]
so that the factor $p^{2}+q^{2}$ cancels between $c_2$ and
$d_1(z)$ --- which is why the result is as simple as it is --- and
\[
-4c_2\,d_1(1)\,\lambda_1^{2}=-2p^{3/2}q^{5/2},
\qquad
-4c_2\,d_1(3)\,\lambda_1^{2}=-2p^{1/2}q^{7/2}.
\]
Since $\kappa(z)=(p/q)^{a-z}$, these are precisely
$-2q^{a}(p/q)^{3/2}$ and $-2q^{a}(p/q)^{1/2}$ with $a=4$; that is,
at both sites of the active orbit,
\[
G_z(t)\;=\;-\,2\,q^{a}\,\sqrt{\kappa(z)}\;t ,
\]
of strict sign for \emph{every} drift $p\in(0,1)$. Note that the
outcome is a single monomial: the several spectral contributions
cancel down to a term linear in $t$. Since the two
components of each director have opposite signs, the derivatives at
the two sites of an orbit have opposite signs, whatever orientation
is chosen for $\mathbf d_O$; the absolute pattern of
Corollary~\ref{cor:two-zones} is fixed by the sign of the constant
$-2q^{a}$ in the displayed formula.
\end{proof}

The cases $a=3$ and $a=4$ therefore provide the first nontrivial
examples in which condition \textup{(V)} can be verified
completely in closed form.

\begin{remark}[Numerical evidence beyond the closed-form cases]
\label{rem:walk-V-numeric}
For $5\le a\le 8$ the strict sign of $G_z$ was verified at both
sites of every orbit across the drift range, by locating the real
roots of $G_z$ on $(0,1)$: there are none. This is numerical
evidence, not proof, and we label it as such: no claim beyond
$a=4$ is used in any statement of this paper except through the
explicit hypothesis of Corollary~\ref{cor:two-zones}. In every case
computed, $G_z$ is, up to a global sign, either a constant or $t$
times a polynomial with coefficients of one sign --- a
total-positivity pattern \cite{Karlin1968} for the resolvent family
$\{(1-\lambda t)^{-2}\}$, which we expect to govern the general
case and which we do not pursue here.
\end{remark}

\begin{corollary}[The two zones, geometrised]
\label{cor:two-zones}
In the canonical realisation, for every $a$ and drift for which the
vertex condition \textup{(V)} holds --- in particular for $a=3,4$
by Proposition~\ref{prop:walk-V}, and for $5\le a\le8$ by the
numerical evidence of Remark~\ref{rem:walk-V-numeric} --- the
starting sites split into two consecutive zones about the
midpoint,
\[
\underbrace{+\ \cdots\ +}_{z<a/2}\quad
\underbrace{\bigl[\,0\,\bigr]_{a\ \mathrm{even}}}_{z=a/2}\quad
\underbrace{-\ \cdots\ -}_{z>a/2}\, ,
\]
where $\pm$ denotes the constant sign of
$\partial_\gamma C(e_z,\cdot)$. Two things must be distinguished.
That every orbit $\{z,a-z\}$ straddles the midpoint and contributes
exactly one site to each sign class is forced by the orbit
structure alone, and is independent of the parity of $a$. That the
class below the midpoint is the one labelled $+$, on the other
hand, is not: it depends on the orientation chosen for the
directors and on the global sign $\varepsilon$ of \textup{(V)}, and
is fixed here by the canonical orientation together with the sign
of the constant $-2q^{a}$ computed in
Proposition~\ref{prop:walk-V}. When $a$ is even the midpoint is
itself a site, the unique fixed point of $\sigma$, and it is
exactly reset-neutral; when $a$ is odd the value $a/2$ is not an integer, so no site is a
fixed point of $\sigma$, none is invariant, and the two index
ranges $z<a/2$ and $z>a/2$ exhaust the sites with no neutral index
between them. \textup{(}The ``zones'' are sets of integer sites
throughout, not regions of a continuum.\textup{)} This is the two-zone theorem
of Paper~I, obtained here as a corollary of the geometry: the two
zones are the two sides of the separatrix, the neutral site --- when
it exists --- is its intersection with the vertex set, and the
midpoint is fixed by the site involution $\sigma$, whose fixed
point it is, while the associated spherical reflection $R_L$ fixes
the separatrix pointwise; the two are distinct structures, kept
apart by Proposition~\ref{prop:three-involutions}.
\end{corollary}

\begin{proposition}[Global orientation for $r=1$]
\label{prop:r-one}
Let the structure satisfy \textup{(S1)--(S4)} with $r=1$, let
$\pi^\ast\in\Sigma$ be the reference distribution of
Section~\ref{subsec:orientation}, and write
$\psi(\gamma)=\varphi(\gamma)\,\mathbf n$ for the response line.
Suppose the Wronskians $\mathcal W[\varphi,s(z;\cdot)]$ keep a
common strict sign $\varepsilon$ on $(0,1)$ at \emph{every} site
$z$ --- neutral sites included. \textup{(}This is strictly stronger
than the vertex condition \textup{(V)}, which constrains only the
two sites of each orbit; the proof below averages over all sites of
$\operatorname{supp}\pi$, so the neutral sites cannot be
omitted.\textup{)} Then for every $\pi\in\Delta_{m-1}^\circ$ and
every $\gamma$,
\[
\operatorname{sgn}\partial_\gamma C(\pi,\gamma)
=\varepsilon\,
\operatorname{sgn}\langle\pi-\pi^\ast,\mathbf n\rangle :
\]
the map $\gamma\mapsto C(\pi,\gamma)$ is strictly monotone off
$\Sigma$ and constant on it, with direction determined once and for
all by the side of the separatrix. \textup{(}Here
$\langle\pi-\pi^\ast,\mathbf n\rangle=0$ cuts out a hyperplane of
$\mathbb R^m$, and
$\Sigma=\{\pi\in\Delta_{m-1}:\langle\pi-\pi^\ast,\mathbf n\rangle=0\}$
is its trace on the simplex; it is the latter that is meant
throughout.\textup{)} In particular the qualitative
content of Conjecture~\ref{conj:global} --- a single
$\gamma$-independent orientation of the simplex, organised by
$\Sigma$ --- holds whenever $r=1$.
\end{proposition}

\begin{proof}
Since $C-C^\ast=\langle\pi,\psi\rangle/\langle\pi,s\rangle
=\varphi\,\langle\pi,\mathbf n\rangle/\langle\pi,s\rangle$, the
quotient rule gives
\[
\partial_\gamma C
=\frac{\langle\pi,\mathbf n\rangle}{\langle\pi,s\rangle^{2}}\,
\mathcal W\bigl[\varphi,\langle\pi,s\rangle\bigr].
\]
The Wronskian is linear in its second argument, so
$\mathcal W[\varphi,\langle\pi,s\rangle]
=\sum_z\pi_z\,\mathcal W[\varphi,s(z;\cdot)]$, a combination with
positive weights of terms of the common strict sign $\varepsilon$;
hence it carries the sign $\varepsilon$ for every $\pi$. Note
that $\varphi\not\equiv0$: otherwise $\psi\equiv0$ and
$V=\operatorname{span}\{\psi(\gamma)\}$ would be trivial, contrary
to $\dim V=r=1$ (Theorem~\ref{thm:fsr}\textup{(iii)}). Finally,
$\pi^\ast\in\Sigma$ gives $\langle\pi^\ast,\psi(\gamma)\rangle=0$
for all $\gamma$, which with $\varphi\not\equiv0$ forces
$\langle\pi^\ast,\mathbf n\rangle=0$,
so $\langle\pi,\mathbf n\rangle
=\langle\pi-\pi^\ast,\mathbf n\rangle$, which vanishes exactly on
$\Sigma$.
\end{proof}

In the canonical realisation the hypothesis is verified
analytically. It bears repeating that this is more than
Proposition~\ref{prop:walk-V} delivers: \textup{(V)} constrains only
the two sites of each active orbit, whereas
Proposition~\ref{prop:r-one} averages over
$\operatorname{supp}\pi$ and so needs the neutral sites too. For
$a=3$ there is no neutral site, and the two Wronskians are the
closed form of Proposition~\ref{prop:walk-V}.
For $a=4$ the two orbit sites are again that closed form, and at
the neutral site $z=2$ the same computation gives the even
polynomial
\[
G_2(t)=-\bigl(\alpha+\beta t^{2}\bigr),
\qquad
\alpha=\sqrt p\,q^{5/2},
\qquad
\beta=2\,p^{3/2}q^{7/2},
\]
with $\alpha,\beta>0$ for \emph{every} $p\in(0,1)$; hence $G_2$ has
no zero on $(0,1)$ --- indeed none on $\mathbb R$ --- and keeps the
same strict sign as at the orbit sites. The hypothesis of
Proposition~\ref{prop:r-one} therefore holds for all drifts, not
merely for sampled ones. For these realisations the calibration
$\varepsilon=\operatorname{sgn}\varphi$ of
Remark~\ref{rem:rotation-reversal}\textup{(iii)} holds as well, so
the sign law and the pointwise identification of
Conjecture~\ref{conj:global} coincide: for the canonical family the
identification is a theorem, not a conjecture.

\begin{remark}[The question over the full simplex]
\label{rem:full-conjecture}
Theorem~\ref{thm:vertex-sign} concerns the gambler's decision ---
the starting site --- with the reset rule held fixed along the
vertex diagonal. The dependence of the response on the reset
distribution $\pi$ itself is a \emph{different} question, an
orthogonal section of the same three-variable object: there,
already in the canonical realisation with a single active orbit,
the global identification of
Conjecture~\ref{conj:global} is governed by the invariant $r$ of
Theorem~\ref{thm:fsr}. Indeed, since
$C(\pi,\gamma)-C^\ast=\langle\pi,\psi(\gamma)\rangle/
\langle\pi,s(\cdot;\gamma)\rangle$ and $\psi(\gamma)$ annihilates
$\pi^\ast$, the conjectured identification is equivalent to
$\operatorname{sgn}\partial_\gamma C=\operatorname{sgn}
(C-C^\ast)$, a differential condition which forces
$|C(\pi,\cdot)-C^\ast|$ to increase, hence global monotonicity in
$\gamma$. When $r=1$ the response direction cannot rotate --- all
$\psi(\gamma)$ span the same line --- and
Proposition~\ref{prop:r-one} converts the identification into a
theorem under a one-dimensional sign condition at every site,
verified for the canonical realisations. When $r\ge2$ the situation changes at the outset. Since
$V=\operatorname{span}\{\psi(\gamma):\gamma\in(0,1)\}$ by definition
(Section~\ref{subsec:orientation}), and $\dim V=r$ by
Theorem~\ref{thm:fsr}\textup{(iii)}, the response field is then not
confined to a single line: the vectors
$\{\psi(\gamma)\}_{\gamma\in(0,1)}$ cannot all be collinear, hence
the functionals $\pi\mapsto\langle\pi,\psi(\gamma)\rangle$ are not
all proportional, and therefore at least two of the instantaneous
zero hyperplanes
$\ker(\psi(\gamma)|_{T_{\pi^\ast}\Delta_{m-1}})$ are distinct.
\textup{(}That is the whole of what is claimed. Theorem~\ref{thm:fsr}
gives non-collinearity, hence two distinct hyperplanes; it says
nothing about the hyperplane tracing a curve, moving monotonically
in angle, or any other stronger geometric behaviour, and we assert
none.\textup{)}

Call a \emph{scalar orientation} a fixed linear functional
$\Lambda$ on $\mathbb R^m$, independent of $\gamma$, such that
\[
\operatorname{sgn}\partial_\gamma C(\pi,\gamma)
=\operatorname{sgn}\Lambda(\pi)
\]
for all interior $\pi$ and all $\gamma$; this is the form the
identification takes in
Conjecture~\ref{conj:global}, and what
Proposition~\ref{prop:r-one} supplies when $r=1$. No such $\Lambda$
exists in general once $r\ge2$: the existence of a single interior
$\pi$ for which $\gamma\mapsto C(\pi,\gamma)$ changes its direction
of monotonicity strictly inside $(0,1)$ already contradicts it,
since $\Lambda(\pi)$ would have to carry both signs. That the
conjecture fails in this regime is \emph{not} proved here, and we
separate the two claims carefully. The criterion just given --- that
a single interior $\pi$ with a change of monotonicity refutes any
scalar orientation --- is proved above. That such $\pi$ exist when
$r\ge2$ was established by T.~Newton, who communicated to us
explicit counterexamples in the abstract class; we are indebted to
him for the observation, which is what confines the statements of
this paper to $r=1$ and to the starting sites. Independently, and
by numerical computation rather than in closed form, we have
located such distributions in the canonical realisation itself for
$a=6,7,8$ over a range of drifts, whereas the same search at $r=1$
($a=3,4$) returns none --- as Proposition~\ref{prop:r-one}
requires. No statement elsewhere in this paper depends on either of
these observations.

The pointwise identification is therefore a statement
about the $r=1$ regime, together with the starting sites of
Theorem~\ref{thm:vertex-sign}. What replaces it when $r\ge2$ ---
a combinatorial description of the sign chambers swept by the
rotating response direction --- is open, and belongs to the sequel
(Section~\ref{sec:beyond}).
\end{remark}

Theorem~\ref{thm:vertex-sign} shows that the phenomenon
discovered in Paper~I is not specific to random walks. It is a
structural consequence of spectral duality, with the random walk
providing only one canonical realisation.

\section{The canonical realisation, illustrated}
\label{sec:illustrations}

The results of the preceding sections are structural: they hold for
any spectral duality structure, and the biased random walk enters
only as one realisation among others. It is nevertheless the
realisation that motivated the whole programme, and it is small
enough that the geometry can be drawn. This section does that. Every
quantity displayed below is computed from the walk itself --- from
the discounted absorption functionals $u(z;\gamma)$, $s(z;\gamma)$
of Appendix~\ref{app:spectral} --- and not from the abstract
axioms; the figures therefore provide numerical tests of the
theory, rather than merely illustrations of its notation.

\subsection{The separatrix and its spherical image}
\label{subsec:fig-geometry}

Take $a=4$ and $p=0.65$. The interior sites are $\{1,2,3\}$, the
involution is $\sigma(z)=4-z$, and the orbit structure is the
simplest one that is not trivial: a single active orbit $\{1,3\}$
and a single neutral site $z_0=2$. The compatibility constant is
$K=(p/q)^4$ and the invariant value is
$C^\ast=1/(1+\sqrt K)=0.2248$.

Proposition~\ref{prop:sigma-exact} predicts that $\Sigma$ is the
segment cut out by the single ratio constraint
$\pi_1/\pi_3=\sqrt{\kappa(3)/\kappa(1)}=q/p$, with the neutral
weight $\pi_2$ free. Figure~\ref{fig:geometry}(a) shows it. The
segment runs from the point $(0.35,0,0.65)$ on the edge
$\overline{e_1e_3}$ --- where the whole mass sits on the active
orbit --- to the neutral vertex $e_2$, where the whole mass sits on
the fixed point of $\sigma$. Along that segment, and only there,
$C(\pi,\cdot)$ is constant: numerically, the variation of
$C(\pi^\ast,\gamma)$ over $\gamma\in(0,1)$ is $1.4\times10^{-16}$,
at the level of machine precision --- a check on the computation,
the constancy itself being Proposition~\ref{prop:sigma-exact}. As reference point we take
$\pi^\ast=\pi^{\natural}\propto\kappa^{-1/2}$, the distinguished
element of Remark~\ref{rem:sectors}; here it is
$(0.237,0.323,0.440)$, and it lies on the segment, as it must.
Because this realisation has a single active orbit, the orbitwise
relation
$\pi^\ast_z\propto\sqrt{\mu(z)}$ of Remark~\ref{rem:pi-mu} happens
to hold across all three sites at once; that is a consequence of
having a single active orbit --- the proportionality constant may
differ from orbit to orbit --- and not a general identity.

The two shaded regions are the two sides of the separatrix, and the
signs at the vertices are those of Theorem~\ref{thm:vertex-sign}:
positive at $e_1$, zero at $e_2$, negative at $e_3$. The neutral
vertex is not merely one where the response happens to vanish; by
Proposition~\ref{prop:neutral-face} it lies on $\Sigma$
unconditionally, and it does so because $\sigma$ fixes it.

Panel (b) is the same picture after the square-root embedding
$\Phi(\pi)=2\sqrt\pi$ of Section~\ref{sec:fisher-rao}. The simplex
becomes the positive octant of the sphere of radius $2$, the
vertices become the three coordinate points on the sphere, and the
straight
separatrix of panel (a) becomes an arc. \textup{(}The panel is drawn
in a planar projection; only the angles are to scale.\textup{)} Theorem~\ref{thm:great-subsphere} asserts that this
arc is a great-circle arc --- a totally geodesic subsphere. The
shaded region lies in the two-dimensional plane through the origin
spanned by the two endpoint vectors of the arc, and the arc is the
intersection of that plane with the sphere --- which is what being
a great circle means. We verified this independently: the
plane normal computed from any two points of the image is the same
to machine precision.

\begin{figure}[t]
\centering
\begin{tikzpicture}[scale=1.0]
\begin{scope}
  \fill[blue!8]  (0,0) -- (2.6,0) -- (2,3.464) -- cycle;
  \fill[red!8]   (2.6,0) -- (4,0) -- (2,3.464) -- cycle;
  \draw[thick] (0,0) -- (4,0) -- (2,3.464) -- cycle;
  \draw[very thick,black] (2.6,0) -- (2,3.464);
  \fill (2.406,1.119) circle (2.2pt);
  \node[right=2pt] at (2.406,1.119) {$\pi^\ast$};
  \fill (0,0) circle (1.6pt); \fill (4,0) circle (1.6pt);
  \fill[white,draw=black,thick] (2,3.464) circle (2.6pt);
  \node[below left]  at (0,0)     {$e_1$};
  \node[below right] at (4,0)     {$e_3$};
  \node[above]       at (2,3.464) {$e_2$};
  \node at (0.15,-0.42) {$+$};
  \node at (3.85,-0.42) {$-$};
  \node at (2,3.95) {$0$};
  \node[blue!60!black] at (1.25,1.05) {$\partial_\gamma C>0$};
  \node[red!60!black]  at (2.95,0.75) {$\partial_\gamma C<0$};
  \node[below] at (2.6,-0.05) {\scriptsize$(0.35,0,0.65)$};
  \node[rotate=-80] at (2.42,2.25) {\scriptsize$\Sigma$};
  \node at (2,-1.50) {(a) the simplex $\Delta_2$};
\end{scope}
\begin{scope}[xshift=7.4cm,yshift=1.5cm,scale=2.6]
  \fill[black!12] (0,0) -- (-0.1518,0.5707) -- (0.0000,-0.8165) -- cycle;
  \draw[gray,thick]
    (0.7071,0.4082) -- (0.6088,0.4580) -- (0.5000,0.5000) -- (0.3827,0.5334)
    -- (0.2588,0.5577) -- (0.1305,0.5724) -- (0.0000,0.5774) -- (-0.1305,0.5724)
    -- (-0.2588,0.5577) -- (-0.3827,0.5334) -- (-0.5000,0.5000)
    -- (-0.6088,0.4580) -- (-0.7071,0.4082);
  \draw[gray,thick]
    (0.7071,0.4082) -- (0.7011,0.2982) -- (0.6830,0.1830) -- (0.6533,0.0647)
    -- (0.6124,-0.0547) -- (0.5610,-0.1732) -- (0.5000,-0.2887)
    -- (0.4305,-0.3992) -- (0.3536,-0.5030) -- (0.2706,-0.5981)
    -- (0.1830,-0.6830) -- (0.0923,-0.7562) -- (0.0000,-0.8165);
  \draw[gray,thick]
    (-0.7071,0.4082) -- (-0.7011,0.2982) -- (-0.6830,0.1830) -- (-0.6533,0.0647)
    -- (-0.6124,-0.0547) -- (-0.5610,-0.1732) -- (-0.5000,-0.2887)
    -- (-0.4305,-0.3992) -- (-0.3536,-0.5030) -- (-0.2706,-0.5981)
    -- (-0.1830,-0.6830) -- (-0.0923,-0.7562) -- (-0.0000,-0.8165);
  \draw[black!45,dashed] (0,0) -- (-0.1518,0.5707);
  \draw[black!45,dashed] (0,0) -- (0.0000,-0.8165);
  \fill[black!55] (0,0) circle (0.6pt);
  \node[black!55,left=1pt,scale=0.34] at (0,0) {$O$};
  \draw[very thick,black]
    (-0.1518,0.5707) -- (-0.1453,0.3107) -- (-0.1385,0.1876) -- (-0.1314,0.0860)
    -- (-0.1239,-0.0055) -- (-0.1159,-0.0912) -- (-0.1073,-0.1738)
    -- (-0.0980,-0.2552) -- (-0.0876,-0.3372) -- (-0.0759,-0.4218)
    -- (-0.0620,-0.5124) -- (-0.0438,-0.6170) -- (0.0000,-0.8165);
  \fill (-0.1249,0.0056) circle (0.9pt);
  \node[right=1pt,scale=0.34] at (-0.1249,0.0056) {$\Phi(\pi^\ast)$};
  \fill (0.7071,0.4082) circle (0.7pt);
  \fill (-0.7071,0.4082) circle (0.7pt);
  \fill[white,draw=black] (0.0000,-0.8165) circle (1.0pt);
  \node[above right=0pt,scale=0.34] at (0.7071,0.4082)  {$\Phi(e_1)$};
  \node[above left=0pt,scale=0.34]  at (-0.7071,0.4082) {$\Phi(e_3)$};
  \node[below=1pt,scale=0.34]       at (0.0000,-0.8165) {$\Phi(e_2)$};
  \node[scale=0.30,rotate=-84] at (-0.175,0.30) {$\Phi(\Sigma)$};
  \node at (0,-1.16)
       {(b) the spherical octant, $\Phi(\pi)=2\sqrt{\pi}$};
\end{scope}
\end{tikzpicture}
\caption{The separatrix for the biased walk with $a=4$, $p=0.65$:
one active orbit $\{1,3\}$ and one neutral site $z_0=2$. (a) In the
simplex, $\Sigma$ is the segment fixed by the single ratio
constraint $\pi_1/\pi_3=q/p$, joining the critical point of the
orbit edge to the neutral vertex; it separates the two zones of
strict monotonicity, and the vertex signs are $[+,0,-]$ as in
Theorem~\ref{thm:vertex-sign}. (b) Under the square-root embedding
$\Phi(\pi)=2\sqrt\pi$
the simplex becomes the positive octant and $\Sigma$ becomes a
great-circle arc: it is the intersection of the sphere with the
shaded plane through the origin, which is the content of
Theorem~\ref{thm:great-subsphere}.}
\label{fig:geometry}
\end{figure}
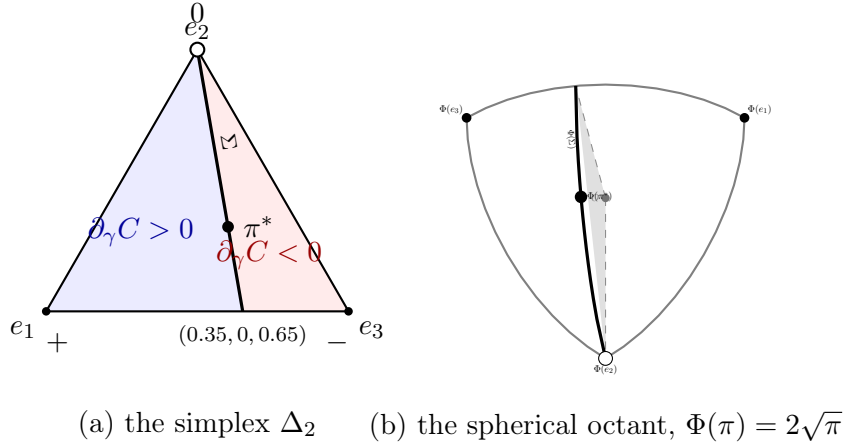

\subsection{The response at the starting sites}
\label{subsec:fig-vertices}

Figure~\ref{fig:response}(a) shows the sign theorem at work for
$a=6$, $p=0.65$, where the interior sites are $\{1,\dots,5\}$, the
orbits are $\{1,5\}$ and $\{2,4\}$, and $z_0=3$ is neutral. The five
curves are the responses $\gamma\mapsto C(e_z,\gamma)$ at the five
vertices of $\Delta_4$; by the discussion of
Section~\ref{sec:sign-theorem} each of them is also the ruin
probability of a gambler who starts at $z$ and resets to $z$.

The picture is exactly the dichotomy of
Corollary~\ref{cor:two-zones}. The two sites below the midpoint
respond upwards, the two above respond downwards, the midpoint does
not respond at all, and each orbit contributes one site to each
zone. The flat line is $C(e_3,\cdot)\equiv C^\ast=0.1350$, constant
to machine precision; it is not a numerical coincidence but
Proposition~\ref{prop:neutral-face}. The behaviour as $\gamma\to1$
is worth reading correctly. In that limit the survival factor
$t=1-\gamma$ tends to zero, so the leading contribution comes from
trajectories absorbed in the minimum possible number of steps: in
the limit, $C$ is decided by whichever boundary is nearer in
lattice distance. From $z=1$ and
$z=2$ the origin is nearer ($1$ and $2$ steps against $5$ and $4$),
so those two curves saturate at $1$; from $z=4$ and $z=5$ the far
boundary is nearer and they saturate at $0$. The midpoint $z_0=3$ is
equidistant --- three steps either way --- and there the two
shortest paths carry weights $q^3$ and $p^3$, so that
\[
C(e_3,\gamma)\;\longrightarrow\;\frac{q^{3}}{q^{3}+p^{3}}
=\frac{1}{1+(p/q)^{3}}=C^\ast,
\]
which is $0.1350$ at $p=0.65$. The same computation at general even
$a$ gives $q^{a/2}/(q^{a/2}+p^{a/2})=1/(1+\sqrt K)=C^\ast$: the
midpoint value is recovered from shortest paths alone, without any
spectral input. The limiting values are therefore \emph{not} the
reset-free ruin probabilities, that is the values
$C(e_z,0)=u(z;0)$ of the same functionals at $\gamma=0$, which for
these five sites are $(0.527,0.272,0.135,0.061,0.021)$; in this
realisation the two coincide only at the midpoint, and that
coincidence is the midpoint invariance of Paper~I.

\subsection{Where the global picture breaks: the invariant $r$}
\label{subsec:fig-rotation}

Theorem~\ref{thm:fsr} attaches to each realisation the integer
$r=\dim V$, the dimension of the span of the response functionals,
and Section~\ref{sec:finite-reduction} identifies it with the number
of active orbits. Remark~\ref{rem:full-conjecture} explains why this
integer, and not the drift, decides whether the response admits a
global scalar orientation. For $r=1$ all response functionals lie
on a single line --- which by itself leaves their orientation free,
since a sign change keeps a vector on its line.
Proposition~\ref{prop:r-one}, together with the one-dimensional
condition holding at every site, neutral sites included, which the
canonical realisations with $r=1$ satisfy, fixes that orientation
and yields the corresponding global sign law. That law is stated with respect to the fixed
director $\mathbf n$ and the constant $\varepsilon$, not with
respect to $\psi(\gamma)$ itself; the two agree exactly when
$\varepsilon=\operatorname{sgn}\varphi$, which holds throughout the
canonical family. For $r\ge2$, Theorem~\ref{thm:fsr} gives that the family
$\{\psi(\gamma)\}_{\gamma\in(0,1)}$ spans a space of dimension at
least two; hence the nonzero response vectors cannot all be
collinear, and the direction does not remain constant over the whole
interval.
Interior distributions may then reverse.

Figure~\ref{fig:response}(b) exhibits such a reversal. The
realisation is $a=6$, $p=0.9$, so $r=2$; the distribution is
$\pi=(0.0117,0.1462,0.2621,0.2558,0.3242)$, which is interior and
has no special structure --- it was drawn at random. Its response
rises until $\gamma^\ast=0.65180\ldots$ and falls thereafter. No boundary in the space of distributions is crossed
here, and there is no small-denominator artefact: the curve is
computed from the linear system that defines $u$ and $s$, its
derivative is evaluated in closed form, and both lobes are of
comparable size. The reversal does \emph{not} occur at a crossing of
the neutral value. At $\gamma^\ast$ one has
$C(\pi,\gamma^\ast)=0.0043567\ldots$, whereas
$C^\ast=1/(1+3^{6})=0.0013699\ldots$, and $C(\pi,\gamma)-C^\ast$
stays strictly positive on the whole of $(0,1)$, between
$2.1\times10^{-3}$ and $3.0\times10^{-3}$; we thank T.~Newton
for this observation. By
\eqref{eq:C-minus-Cstar}, $L_\gamma(\pi)$ therefore keeps its sign
throughout, and what changes sign at $\gamma^\ast$ is the Wronskian
$\mathcal W\bigl[L_\gamma(\pi),\langle\pi,s(\cdot;\gamma)\rangle\bigr]$
that governs $\partial_\gamma C$ --- a rate condition, not a level
condition. The rotation of $\psi(\gamma)$ is what makes this
possible: for $r=1$ the same Wronskian factors as
$\langle\pi,\mathbf n\rangle\,\mathcal W[\varphi,\langle\pi,s\rangle]$
and, under the hypothesis of Proposition~\ref{prop:r-one}, cannot
vanish, whereas for $r\ge2$ it is a combination
$\sum_O t_O\,\mathcal W[\varphi_O,\langle\pi,s\rangle]$ of
independent orbit contributions, $t_O=\langle\pi|_O,\mathbf d_O\rangle$,
whose sign is no longer fixed. What fails is not the theory but the
attempt to summarise a rotating direction by a single sign. For
$a=4$, where $r=1$, we found no such distribution: in a numerical
check, $200$ randomly sampled interior distributions at each of four
drift values were all strictly monotone, as the reduction above
predicts.

\begin{figure}[t]
\centering
\begin{tikzpicture}
\begin{axis}[
  width=7.4cm, height=5.6cm,
  xlabel={$\gamma$}, ylabel={$C(e_z,\gamma)$},
  xmin=0, xmax=1, ymin=-0.03, ymax=1.05,
  title={\small (a) vertex responses, $a=6$, $p=0.65$},
  legend style={at={(0.97,0.55)},anchor=east,draw=none,
                fill=none,font=\scriptsize,cells={anchor=west}},
  tick label style={font=\scriptsize},
  label style={font=\small}, every axis title/.style={above,at={(0.5,1.0)},yshift=2pt},
]
\addplot[blue,thick] coordinates {(0.020,0.5620)(0.050,0.6122)(0.080,0.6592)(0.110,0.7028)(0.140,0.7427)(0.170,0.7789)(0.200,0.8115)(0.230,0.8404)(0.260,0.8659)(0.290,0.8882)(0.320,0.9075)(0.350,0.9241)(0.380,0.9383)(0.410,0.9503)(0.440,0.9604)(0.470,0.9688)(0.500,0.9757)(0.530,0.9814)(0.560,0.9859)(0.590,0.9895)(0.620,0.9924)(0.650,0.9946)(0.680,0.9963)(0.710,0.9975)(0.740,0.9984)(0.770,0.9990)(0.800,0.9995)(0.830,0.9997)(0.860,0.9999)(0.890,1.0000)(0.920,1.0000)(0.950,1.0000)(0.980,1.0000)};
\addlegendentry{$z=1$}
\addplot[blue!55,thick] coordinates {(0.020,0.2869)(0.050,0.3095)(0.080,0.3327)(0.110,0.3566)(0.140,0.3810)(0.170,0.4061)(0.200,0.4318)(0.230,0.4581)(0.260,0.4848)(0.290,0.5120)(0.320,0.5396)(0.350,0.5675)(0.380,0.5956)(0.410,0.6239)(0.440,0.6523)(0.470,0.6806)(0.500,0.7087)(0.530,0.7365)(0.560,0.7637)(0.590,0.7904)(0.620,0.8162)(0.650,0.8410)(0.680,0.8647)(0.710,0.8870)(0.740,0.9077)(0.770,0.9268)(0.800,0.9440)(0.830,0.9591)(0.860,0.9720)(0.890,0.9826)(0.920,0.9907)(0.950,0.9964)(0.980,0.9994)};
\addlegendentry{$z=2$}
\addplot[black,very thick,dashed] coordinates {(0.020,0.1350)(0.500,0.1350)(0.980,0.1350)};
\addlegendentry{$z=3$ (neutral)}
\addplot[red!55,thick] coordinates {(0.020,0.0571)(0.050,0.0516)(0.080,0.0466)(0.110,0.0421)(0.140,0.0381)(0.170,0.0344)(0.200,0.0311)(0.230,0.0280)(0.260,0.0252)(0.290,0.0227)(0.320,0.0204)(0.350,0.0182)(0.380,0.0163)(0.410,0.0145)(0.440,0.0128)(0.470,0.0113)(0.500,0.0099)(0.530,0.0086)(0.560,0.0075)(0.590,0.0064)(0.620,0.0055)(0.650,0.0046)(0.680,0.0038)(0.710,0.0031)(0.740,0.0025)(0.770,0.0019)(0.800,0.0014)(0.830,0.0010)(0.860,0.0007)(0.890,0.0004)(0.920,0.0002)(0.950,0.0001)(0.980,0.0000)};
\addlegendentry{$z=4$}
\addplot[red,thick] coordinates {(0.020,0.0186)(0.050,0.0152)(0.080,0.0124)(0.110,0.0102)(0.140,0.0084)(0.170,0.0069)(0.200,0.0056)(0.230,0.0046)(0.260,0.0038)(0.290,0.0031)(0.320,0.0025)(0.350,0.0020)(0.380,0.0016)(0.410,0.0013)(0.440,0.0010)(0.470,0.0008)(0.500,0.0006)(0.530,0.0005)(0.560,0.0003)(0.590,0.0003)(0.620,0.0002)(0.650,0.0001)(0.680,0.0001)(0.710,0.0001)(0.740,0.0000)(0.770,0.0000)(0.800,0.0000)(0.830,0.0000)(0.860,0.0000)(0.890,0.0000)(0.920,0.0000)(0.950,0.0000)(0.980,0.0000)};
\addlegendentry{$z=5$}
\end{axis}
\begin{scope}[xshift=8.2cm]
\begin{axis}[
  width=7.4cm, height=5.6cm,
  xlabel={$\gamma$}, ylabel={$10^{3}\,C(\pi,\gamma)$},
  xmin=0, xmax=1, ymin=3.4, ymax=4.55,
  title={\small (b) a rotating distribution, $a=6$, $p=0.9$, $r=2$},
  tick label style={font=\scriptsize}, label style={font=\small},
  every axis title/.style={above,at={(0.5,1.0)},yshift=2pt},
]
\addplot[black,very thick] coordinates {(0.020,3.5497)(0.050,3.6125)(0.080,3.6739)(0.110,3.7338)(0.140,3.7920)(0.170,3.8483)(0.200,3.9027)(0.230,3.9549)(0.260,4.0047)(0.290,4.0520)(0.320,4.0967)(0.350,4.1384)(0.380,4.1771)(0.410,4.2126)(0.440,4.2446)(0.470,4.2730)(0.500,4.2976)(0.530,4.3182)(0.560,4.3346)(0.590,4.3466)(0.620,4.3540)(0.650,4.3567)(0.680,4.3545)(0.710,4.3473)(0.740,4.3349)(0.770,4.3171)(0.800,4.2939)(0.830,4.2652)(0.860,4.2308)(0.890,4.1907)(0.920,4.1448)(0.950,4.0930)(0.980,4.0355)};
\addplot[only marks,mark=*,mark size=1.6pt,black] coordinates {(0.650,4.3567)};
\node[font=\scriptsize,anchor=south] at (axis cs:0.65,4.375) {$\partial_\gamma C=0$};
\node[font=\scriptsize,blue!60!black,anchor=north] at (axis cs:0.28,4.05) {$\partial_\gamma C>0$};
\node[font=\scriptsize,red!60!black,anchor=north] at (axis cs:0.88,4.28) {$\partial_\gamma C<0$};
\end{axis}
\end{scope}
\end{tikzpicture}
\caption{(a) The sign theorem at the starting sites for $a=6$,
$p=0.65$: the two sites below the midpoint respond upwards, the two
above respond downwards, and the neutral site $z_0=3$ does not
respond at all, its curve sitting at $C^\ast=0.1350$ to machine
precision. Each orbit contributes one site to each zone, which is
the two-zone theorem of Paper~I recovered as
Corollary~\ref{cor:two-zones}. (b) With $r\ge2$ the response
direction rotates and the global picture changes: for $a=6$,
$p=0.9$ the interior distribution
$\pi=(0.0117,0.1462,0.2621,0.2558,0.3242)$ increases up to
$\gamma\approx0.65$ and decreases afterwards. Both lobes are
macroscopic; the derivative is evaluated in closed form.}
\label{fig:response}
\end{figure}
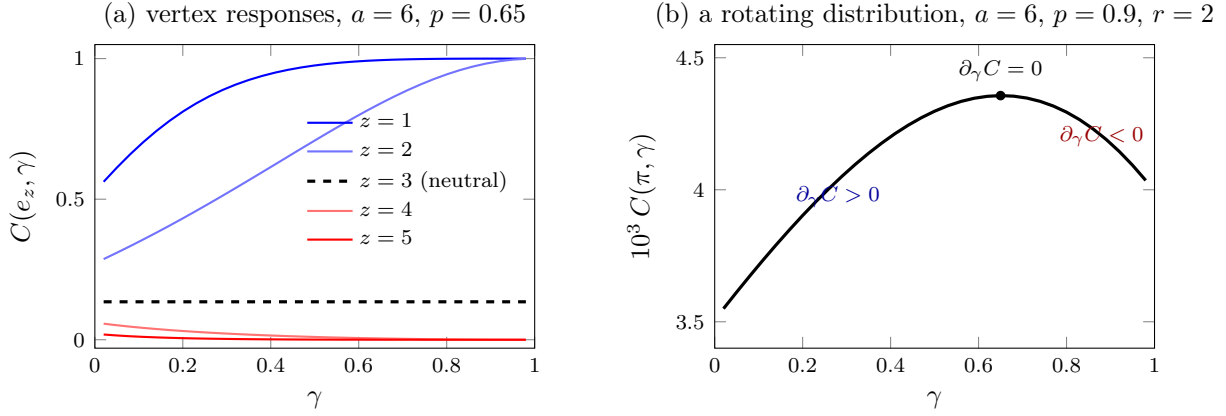
The rotation itself can be drawn, and
Figure~\ref{fig:rotation} does so for the realisation of
Figure~\ref{fig:response}(a). By Theorem~\ref{thm:fsr} every
response functional $\psi(\gamma)$ lies in the plane $V$ spanned by
the two orbit directors, so it has well-defined coordinates
$(\varphi_{O_1}(\gamma),\varphi_{O_2}(\gamma))$ in the director
basis; we computed them from the walk and checked that the
component along the neutral site and the two consistency relations
within the orbits vanish to $10^{-16}$, which is
a numerical check of Theorem~\ref{thm:fsr} on the physical object,
the theorem itself being already proved. As $\gamma$
runs from $0$ to $1$ the direction of $\psi(\gamma)$ turns towards
the director of the outer orbit $\{1,5\}$. Two angles must be
distinguished here. The directors of distinct orbits are orthogonal,
having disjoint supports, but they are not unit vectors:
$\|\mathbf d_O\|^2=\kappa(z)+\kappa(\sigma(z))$, which here gives
$\|\mathbf d_{O_1}\|=4.894$ and $\|\mathbf d_{O_2}\|=3.917$. The
\emph{coordinate} angle read off the pair
$(\varphi_{O_1},\varphi_{O_2})$ therefore differs from the Euclidean
angle in $V$. Concretely, if
$\psi=\varphi_{O_1}\mathbf d_{O_1}+\varphi_{O_2}\mathbf d_{O_2}$,
the Euclidean direction is determined by the weighted pair
\[
\bigl(\|\mathbf d_{O_1}\|\,\varphi_{O_1},\;
      \|\mathbf d_{O_2}\|\,\varphi_{O_2}\bigr),
\]
not by the raw coordinates. The coordinate angle changes from
$25.1^\circ$ to $0.5^\circ$, whereas the corresponding Euclidean
direction in $V$ rotates by $20.1^\circ$. \emph{The $25^\circ$ arc
drawn in Figure~\ref{fig:rotation} refers to the coordinate angle,
not to the Euclidean angle in $V$.} In this realisation the
turn happens to be monotone; that is an observation about the
computed curve, not a consequence of any result above, which asserts
only that the direction cannot be constant. A distribution $\pi$ pairs with this
moving direction through
$L_\gamma(\pi)=\langle\pi-\pi^\ast,\psi(\gamma)\rangle$. A crossing
of the hyperplane orthogonal to $\pi-\pi^\ast$ would be a zero of
$L_\gamma(\pi)$, that is a passage of $C(\pi,\gamma)$ through
$C^\ast$; it is a change of level, not of monotonicity, and it is
not what happens in Figure~\ref{fig:response}(b), where $L_\gamma(\pi)$
never vanishes. What the rotation does is remove the factorisation
available for $r=1$: the rate $\partial_\gamma C$ depends on how the
orbit coefficients $\varphi_O(\gamma)$ change relative to one
another, and not on the sign of $L_\gamma(\pi)$ alone.

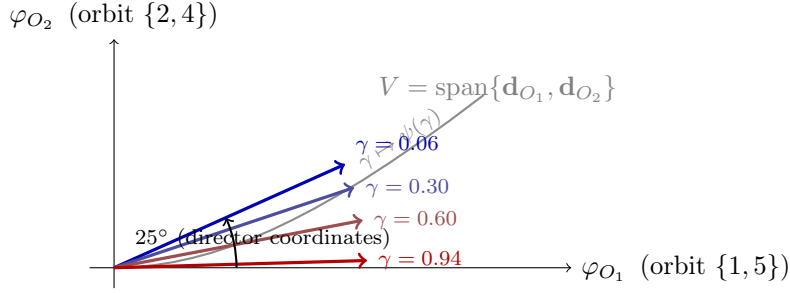
\begin{figure}[t]
\centering
\begin{tikzpicture}[scale=5.4]
  \draw[->] (-0.06,0) -- (1.12,0)
    node[right,font=\small] {$\varphi_{O_1}$\, (orbit $\{1,5\}$)};
  \draw[->] (0,-0.05) -- (0,0.56)
    node[above,font=\small] {$\varphi_{O_2}$\, (orbit $\{2,4\}$)};
  \node[font=\small,gray] at (0.94,0.44) {$V=\operatorname{span}\{\mathbf d_{O_1},\mathbf d_{O_2}\}$};
  \draw[black!45,thick]
    (0.9059,0.4235) -- (0.8231,0.3612) -- (0.7470,0.3064) -- (0.6764,0.2581)
    -- (0.6107,0.2155) -- (0.5490,0.1781) -- (0.4908,0.1451) -- (0.4356,0.1164)
    -- (0.3830,0.0913) -- (0.3325,0.0698) -- (0.2838,0.0514) -- (0.2367,0.0361)
    -- (0.1908,0.0237) -- (0.1458,0.0139) -- (0.1016,0.0068) -- (0.0579,0.0022)
    -- (0.0144,0.0001);
  \node[black!45,font=\scriptsize,rotate=33] at (0.70,0.315) {$\gamma\mapsto\psi(\gamma)$};
  \draw[->,very thick,blue!70!black]  (0,0) -- (0.5658,0.2536);
  \draw[->,very thick,blue!45!black!70] (0,0) -- (0.5881,0.1964);
  \draw[->,very thick,red!45!black!70]  (0,0) -- (0.6090,0.1162);
  \draw[->,very thick,red!70!black]   (0,0) -- (0.6198,0.0177);
  \node[blue!70!black,font=\scriptsize,anchor=south west] at (0.567,0.253) {$\gamma=0.06$};
  \node[blue!45!black!70,font=\scriptsize,anchor=west] at (0.592,0.196) {$\gamma=0.30$};
  \node[red!45!black!70,font=\scriptsize,anchor=west] at (0.612,0.116) {$\gamma=0.60$};
  \node[red!70!black,font=\scriptsize,anchor=west] at (0.623,0.018) {$\gamma=0.94$};
  \draw[->,thick] (0.30,0.0) arc (0:24.1:0.30);
  \node[font=\scriptsize] at (0.365,0.070) {$25^\circ$ (director coordinates)};
\end{tikzpicture}
\caption{The rotating orientation field for $a=6$, $p=0.65$ --- the
realisation of Figure~\ref{fig:response}(a) seen from inside the
response plane. Every $\psi(\gamma)$ lies in the plane $V$ spanned
by the two orbit directors (Theorem~\ref{thm:fsr}; checked
numerically here to $10^{-16}$ on the physical object), with
coordinates
$(\varphi_{O_1},\varphi_{O_2})$ in the director basis. The grey
curve is the trajectory of $\psi$, normalised to its largest value;
the arrows are its directions at four resetting rates. The axes are
the director coordinates, in which the sweep reads $25^\circ$; since
the directors are orthogonal but not normalised
($\|\mathbf d_{O_1}\|=4.894$, $\|\mathbf d_{O_2}\|=3.917$), the
Euclidean angle swept in $V$ is $20.1^\circ$. The turn is monotone
here and towards the outer-orbit director as $\gamma\to1$. It is
this rotation that rules out a single scalar orientation for
$r\ge2$; the reversal of Figure~\ref{fig:response}(b) is a zero of
$\partial_\gamma C$ and occurs without any crossing of the neutral
value $C^\ast$.}
\label{fig:rotation}
\end{figure}

\subsection{Summary of the verified data}
\label{subsec:table}

Table~\ref{tab:realisations} collects the structural data of the
canonical realisation for the domain sizes we have examined. The
column $r$ reports a numerical rank: the field $\psi(\gamma)$ is
sampled at $60$ equally spaced values of $\gamma\in(0.05,0.95)$, the
sampled vectors are taken as the rows of a matrix, and its rank is
read from the singular values, those exceeding
$10^{-9}\sigma_{\max}$ being counted as nonzero. No normalisation is
applied beforehand. The result agrees in every case with the exact
value $r=n_{\mathrm{pair}}$ predicted by Theorem~\ref{thm:fsr}, with
$\sigma_r/\sigma_{r+1}>10^{13}$ throughout --- $\sigma_{r+1}$
sitting at the level of machine precision, so that the count is
insensitive to the tolerance. It is equally insensitive to the
sampling: taking $40$ or $200$ points instead of $60$, widening the
interval, or normalising the rows beforehand changes neither the
rank nor the order of magnitude of the gap. The
last column records whether the vertex condition
\textup{(V)} holds at every site of every orbit across the drift
range $p\in[0.55,0.95]$; it does, in all cases examined. For $a=3$
and $a=4$ this is also proved in closed form
(Proposition~\ref{prop:walk-V}).

\begin{table}[H]
\centering
\caption{Structural data of the canonical realisation at
$p=0.65$. Here $N=a-1$ is the number of spectral modes,
$n_{\mathrm{pair}}$ the number of active orbits, $n_0\in\{0,1\}$ the
number of neutral sites, $r$ the numerically computed dimension of
the response span, and $C^\ast=1/(1+\sqrt K)$ the invariant value.
The identity $r=n_{\mathrm{pair}}$ is Theorem~\ref{thm:fsr}, proved
there; the column reports its numerical confirmation. The last
column reports the verification of \textup{(V)} across
$p\in[0.55,0.95]$.}

\vspace{0.3cm}
\label{tab:realisations}
\begin{tabular}{ccccccc}
\hline\hline
$a$ & $N$ & $n_{\mathrm{pair}}$ & $n_0$ & $r$ & $C^\ast$ & \textup{(V)} \\
\hline
$3$ & $2$ & $1$ & $0$ & $1$ & $0.2832$ & yes \\
$4$ & $3$ & $1$ & $1$ & $1$ & $0.2248$ & yes \\
$5$ & $4$ & $2$ & $0$ & $2$ & $0.1754$ & yes \\
$6$ & $5$ & $2$ & $1$ & $2$ & $0.1350$ & yes \\
$7$ & $6$ & $3$ & $0$ & $3$ & $0.1028$ & yes \\
$8$ & $7$ & $3$ & $1$ & $3$ & $0.0775$ & yes \\
\hline\hline
\end{tabular}
\end{table}

The table displays a sharp structural dichotomy: the rank is
exactly the number of active orbits --- Theorem~\ref{thm:fsr} made
visible, column by column --- while the appearance of a neutral
site records only the parity of $a$. Parity plays no further role:
it decides whether the midpoint is a site, and hence the unique
fixed point of $\sigma$, not whether the two zones exist, since the
zones are forced by the orbit structure alone. And since $r$ grows
with $a$, the realisations that are large enough to be interesting
are precisely those in which the response direction rotates. The two-site case of Paper~III,
where the whole simplex is governed by a single sign, is the
exception rather than the rule; the generic $r\ge2$ setting admits
the reversal mechanism illustrated in
Figure~\ref{fig:response}(b) --- admits, not forces, since whether a
given $\pi$ reverses depends on $\pi$ --- and the
systematic description of which distributions do is the subject of
the sequel.

\section{Outlook}
\label{sec:beyond}

The results of this paper describe the reset response of a spectral
duality structure through the geometry it induces on the simplex:
the separatrix $\Sigma$ as a critical manifold, the invariant value
$C^\ast$, the response span $V$ and its dimension $r$, and the
orientation field $\psi(\gamma)$ that rotates within $V$. Several of
these threads are, by their nature, the beginning of a longer story,
and we indicate here where they lead.

\subsection*{The rotating regime}

For $r=1$ the reset response is globally oriented: a single
$\gamma$-independent side of $\Sigma$ decides the sign of
$\partial_\gamma C$, and Proposition~\ref{prop:r-one} makes this a
theorem. For $r\ge2$ the response span has dimension at least two,
so the direction of $\psi(\gamma)$ cannot stay constant, and the
simplex is no longer split into two monotone regions but organised
into finer chambers whose walls the moving hyperplane
$\ker\psi(\gamma)$ sweeps out. Figure~\ref{fig:rotation} shows the
field turning; what it does not show is the combinatorial object
that records how the chambers fit together. These chambers record
the sign of $C-C^\ast$; as Figure~\ref{fig:response}(b) shows, for
$r\ge2$ that sign no longer decides the sign of $\partial_\gamma C$,
so a complete account must also track the rate condition.

That object is a sign arrangement. Each site contributes a wall,
each orbit a direction in $V$, and the chambers are the cells into
which these walls cut the simplex; their incidences form a
partially ordered set graded by the invariant $r$. The two-zone
picture of Section~\ref{sec:sign-theorem} is the case where this
poset has a single generator; the rotating case is where its full
structure appears. A systematic account --- the lattice of sign
chambers, its rank, and the way $\Sigma$ sits inside it as the flat
of codimension $r$ --- is the subject of a sequel, where the
combinatorics can be developed on its own terms rather than glimpsed
through a single realisation.

\subsection*{The projective and metric structure}

The square-root embedding of Section~\ref{sec:fisher-rao} sends the
simplex to the positive octant of a sphere and the separatrix to a
totally geodesic subsphere (Theorem~\ref{thm:great-subsphere}). Two
further structures are visible from there but not pursued here.
First, two distinct involutions act on the sphere and should not be
run together. The site involution $\sigma$ acts as the reflection
exchanging the coordinates of each orbit; its fixed set in the
simplex is the symmetric family $\{\pi:\pi_z=\pi_{\sigma(z)}\}$,
which contains the neutral vertices but is larger than them, and
which is \emph{not} contained in $\Sigma$ --- indeed $\Sigma$ is not
$\sigma$-invariant unless $\kappa(z)=\kappa(\sigma(z))$. Separately,
extending the octant to the ambient sphere and passing to the
antipodal quotient $x\sim-x$ gives a projective object, and
$\Phi(\Sigma^\circ)$, being the trace of a linear subspace, does
descend to it. It is this second identification, not $\sigma$, that
produces the projective picture; what a joint treatment of the two
would look like is precisely what we have not pursued here. Second, the same square root that carries
measures to the Hilbert sphere is the one underlying the
Bhattacharyya coefficient \cite{Bhattacharyya1943} and, through it,
the fidelity of quantum information geometry \cite{Uhlmann1976}: the reset-neutral value $C^\ast$ and the
distance from $\Sigma$ acquire, in that language, a metric reading.
We have kept the present paper on the real, classical side of this
correspondence; the projective and information-geometric
development belongs elsewhere.

\subsection*{The representation as the primary object}

Finally, Remark~\ref{rem:bridge} recast the coupling functional as a
derived observable: the primary datum is the linear representation
$\mathcal R:\Delta_{m-1}\to\mathcal H$ that carries a distribution
to its spectral coordinates, and $C$ is a rational projection of it.
Read this way, a spectral duality structure is not a functional on a
simplex but a bridge between two universal geometries --- the
spectral Hilbert space of the underlying operator and the
Fisher--Rao manifold of reset distributions --- joined by a
model-dependent correspondence. The structural conditions
(S1)--(S4) become properties of that bridge. Developing this
viewpoint into a definition, and asking which pairs of geometries
arise from genuine Markov processes, is the natural continuation of
the programme and the frame in which the operator-theoretic origin
of the duality, promised in Paper~IV, is most naturally posed.

\medskip

In this perspective the invariant $r$ measures the intrinsic
complexity of the response geometry: it is at once the number of
independent response directions, the codimension of the local
separatrix, and the parameter controlling the transition from
global orientation to genuine rotation. None of the directions
sketched here is required to read the results above: this paper
stands on the four structural conditions and the geometry they
force. They indicate, rather, that the reset response geometry
studied here is one face of a larger object, and that the biased
random walk --- the model that began the programme --- is the
smallest window onto it.

\appendix
\section{Spectral realisation for the random walk}
\label{app:spectral}

Theorem~\ref{thm:fsr} was proved in the body by an exact algebraic
collapse of the orbits, a route that owes nothing to any particular
realisation. For the biased random walk the same reduction can be
obtained a second time, and from a different direction: through the
spectral decomposition of the symmetrised generator and the
resolvent that produced it. This appendix records that route. It is
worth keeping for three reasons.

First, two structurally distinct derivations of the same
finite-dimensional reduction --- the abstract one from
\textup{(S2)--(S4)}, the spectral one below for the canonical
realisation --- are evidence of its robustness. They are not
logically independent: the source \eqref{eq:xi-def}, the form
$\psi=\mathcal G(t)\xi$ and the canonical modes all come from the
realisation, and the modes are quoted from Paper~II. What differs is
the route --- orbit collapse against spectral parity --- and it is
that difference which the agreement tests.

Second, the resolvent-invariance principle the spectral route rests
on (\S\ref{subsec:resolvent-invariance}) owes nothing to the
resetting mechanism, holds for any diagonalisable operator, and may
be of independent interest.

Third, it is through the resolvent that the reduction connects with
the operator-theoretic framework of Paper~IV \cite{PaperIV}, a
connection the abstract route leaves invisible.

The orientation field $\psi(\gamma)$, a priori an infinite family
indexed by the continuous reset parameter $\gamma\in(0,1)$, spans a
subspace of \emph{finite} dimension equal to an intrinsic invariant
of the process, with all $\gamma$-dependence confined to a finite set
of scalar coefficients. The finite
reduction for our process (\S\ref{subsec:parity-count}) is the
application of that principle to the symmetrised generator, with the
count fixed by the spatial symmetry.

\subsection{Spectral directions are invariant under the resolvent}
\label{subsec:resolvent-invariance}

The first result is purely linear-algebraic --- finite dimensions
throughout, no analysis --- and we state it with its own parameter
$\tau$: in the application below $\tau$ will be the discount
$t=1-\gamma$, not the reset rate $\gamma$ itself. Let $A$ be a
diagonalisable operator on a finite-dimensional space with spectral
projectors $\{P_k\}$, $A=\sum_k\lambda_k P_k$, and let $J$ be any
nonempty open interval on which $I-\tau A$ is invertible, so that
$R(\tau)=(I-\tau A)^{-1}$ is defined there. \textup{(}Here $I$ is
the identity, as elsewhere in this paper; the interval is $J$.\textup{)}

\begin{lemma}[Resolvent invariance of spectral directions]
\label{lem:resolvent-invariance}
Assume the active eigenvalues $\{\lambda_k : P_k\xi\neq 0\}$ are
pairwise distinct. Then
\begin{equation}
\operatorname{span}\{R(\tau)\xi : \tau\in J\}
= \bigoplus_{k:\,P_k\xi\neq 0}\operatorname{span}\{P_k\xi\}.
\label{eq:resolvent-invariance}
\end{equation}
The conclusion is unchanged if $R(\tau)$ is multiplied by any
nowhere-vanishing scalar function of $\tau$, since each individual
vector $R(\tau)\xi$ is then merely rescaled by a nonzero scalar.
\end{lemma}

\begin{proof}
From $R(\tau)=\sum_k(1-\tau\lambda_k)^{-1}P_k$ one has
$R(\tau)\xi=\sum_k (1-\tau\lambda_k)^{-1}P_k\xi$, so the left side
is contained in the right. For the reverse inclusion, let
$k_1,\dots,k_s$ index the active projectors ($P_{k_j}\xi\neq0$) with
distinct $\lambda_{k_j}$. Choose distinct
$\tau_1,\dots,\tau_s\in J$ and form the matrix
$M_{ij}=(1-\tau_i\lambda_{k_j})^{-1}$, which is well defined since
$\tau_i\in J$. We claim $M$ is invertible. Suppose
$\sum_j c_j M_{ij}=0$ for every $i$, and put
\[
Q(\tau):=\sum_j c_j\!\!\prod_{j'\neq j}\!\bigl(1-\tau\lambda_{k_{j'}}\bigr),
\]
a polynomial of degree at most $s-1$ obtained by clearing
denominators. It vanishes at the $s$ distinct points
$\tau_1,\dots,\tau_s$, hence $Q\equiv0$. Evaluating at
$\tau=1/\lambda_{k_j}$ for any $\lambda_{k_j}\neq0$ kills every term
but the $j$th and leaves
$c_j\prod_{j'\neq j}(1-\lambda_{k_{j'}}/\lambda_{k_j})=0$, whose
product is nonzero because the $\lambda_{k_{j'}}$ are distinct; so
$c_j=0$. At most one $\lambda_{k_j}$ vanishes, and for that index
$Q\equiv0$ then reads $c_j\prod_{j'\neq j}(1-\tau\lambda_{k_{j'}})
\equiv0$, forcing $c_j=0$ as well. Hence $M$ is invertible, and the
linear system
$R(\tau_i)\xi=\sum_j M_{ij}\,P_{k_j}\xi$ can be solved for each
$P_{k_j}\xi$ as a linear combination of $\{R(\tau_i)\xi\}_i$.
Therefore every $P_{k_j}\xi$ lies in
$\operatorname{span}\{R(\tau)\xi\}$, and both spans coincide. For
the final sentence, replacing $R(\tau)\xi$ by $c(\tau)R(\tau)\xi$
with $c(\tau)\neq0$ replaces each member of the spanning family by a
nonzero multiple of itself, which leaves the span unchanged.
\end{proof}

The content is that \emph{the resolvent creates no new spectral
directions}: it only rescales each spectral projection of $\xi$ by the
scalar $(1-\tau\lambda_k)^{-1}$. Since the projectors are fixed, the
only directions available are those already present in the spectral
decomposition of $\xi$. It is for this reason that the reduction is
insensitive to how the discount is parametrised: reparametrising the
interval $J$ bijectively, or multiplying
$R$ by a nonvanishing scalar, changes neither side of
\eqref{eq:resolvent-invariance}. This is the feature of the resolvent structure
of Paper~IV that the present paper makes visible; it holds for any
diagonalisable operator and does not refer to resetting. In the reset
model the eigenvalues $\lambda_k=2\sqrt{pq}\cos(\pi k/a)$ are simple
(the cosine is injective on $(0,\pi)$), so the distinctness hypothesis
is automatic.

\subsection{The reduction space and its dimension}
\label{subsec:parity-count}

We now specialise to the reset problem. Recall from
Section~\ref{sec:inherited} that the orientation field is
\begin{equation}
\psi(\gamma) = \mathcal G(t)\,\xi,
\qquad
\mathcal G(t):=t\,(I-tP)^{-1},
\qquad
t:=1-\gamma,
\label{eq:psi-def-app}
\end{equation}
where $P$ is the interior generator and $t$ is the per-step
survival factor, so that
$u(\cdot;\gamma)=\mathcal G(t)b_0$ and
$s(\cdot;\gamma)=\mathcal G(t)(b_0+b_a)$, in agreement with the
spectral scalars $f_\nu(\gamma)=t/(1-\lambda_\nu t)$ used throughout
the body. We write $S$ for the spatial reflection $(Sv)_z=v_{a-z}$
--- reserving $\mathcal R$ for the representation map of
Section~\ref{sec:inherited}, with which it has nothing to do ---
and
\begin{equation}
\xi = b_0 - C^\ast(b_0+b_a),
\quad (b_0)_z=q\,\delta_{z,1},\ (b_a)_z=p\,\delta_{z,a-1}.
\label{eq:xi-def}
\end{equation}
Set $V := \operatorname{span}\{\psi(\gamma):\gamma\in(0,1)\}$. Note
that $\psi$ is \emph{not} defined by antisymmetrising: no projector
is applied to $\mathcal G(t)\xi$. That the field turns out to be
antisymmetric --- in the symmetrised coordinates introduced below,
not in the original ones --- is a consequence of
Lemma~\ref{lem:parity} together with $[\widetilde P,S]=0$, and is
the substance of Theorem~\ref{thm:fsr-spectral} rather than an
assumption of it.

For the biased walk $P$ is not symmetric but is symmetrisable: with
$D=\operatorname{diag}(\sqrt{\kappa(z)})$, $\kappa(z)=(p/q)^{a-z}$, the
operator $\widetilde P=D^{-1}PD$ is symmetric.

\begin{lemma}
\label{lem:sym}
$\widetilde P$ is the symmetric tridiagonal matrix with constant
off-diagonal $\sqrt{pq}$; it commutes with $S$. Its eigenpairs are
\begin{equation}
v_k(z)=\sin(\pi k z/a),\qquad
\lambda_k = 2\sqrt{pq}\,\cos(\pi k/a),\qquad k=1,\dots,a-1,
\label{eq:eigenpairs}
\end{equation}
and each eigenvector has definite parity under the reflection,
\begin{equation}
S v_k = (-1)^{k+1} v_k,
\label{eq:eigen-parity}
\end{equation}
so $v_k$ is odd for $k$ even and even for $k$ odd.
\end{lemma}

\begin{proof}
Symmetry and the tridiagonal form are the identity
\[
P_{z,z+1}\sqrt{\kappa_{z+1}/\kappa_z}=\sqrt{pq}
=P_{z+1,z}\sqrt{\kappa_z/\kappa_{z+1}} ;
\]
\eqref{eq:eigenpairs} is classical. For \eqref{eq:eigen-parity},
$v_k(a-z)=\sin(\pi k-\pi kz/a)=(-1)^{k+1}\sin(\pi kz/a)$, and
$[\widetilde P,S]=0$ follows.
\end{proof}

\subsection*{The canonical realisation: modes and duality}

The eigenvectors $v_k$ of $\widetilde P$ are not themselves the
spectral modes of the walk: the modes carry the Doob weight
restored by the conjugation $D$. Explicitly, for the biased walk
with multi-site geometric resetting the discounted absorption
functionals admit the finite decompositions
$u(z;\gamma)=\sum_\nu f_\nu(\gamma)A_\nu(z)$ and
$s(z;\gamma)=\sum_\nu f_\nu(\gamma)\bigl(A_\nu+B_\nu\bigr)(z)$ with
$f_\nu(\gamma)=(1-\gamma)/(1-\lambda_\nu(1-\gamma))$ and
\begin{equation}
A_\nu(z)=w_\nu\Bigl(\frac qp\Bigr)^{z/2}\sin\frac{\pi\nu z}{a},
\qquad
B_\nu(z)=w_\nu\Bigl(\frac pq\Bigr)^{(a-z)/2}
\sin\frac{\pi\nu(a-z)}{a},
\qquad
w_\nu=\frac{2\sqrt{pq}}{a}\,\sin\frac{\pi\nu}{a},
\label{eq:canonical-modes}
\end{equation}
as established in Paper~II. The Doob factors $(q/p)^{z/2}$ and
$(p/q)^{(a-z)/2}$ encode the bias asymmetry, and the modal weights
$w_\nu$ are independent of the site.

\begin{lemma}[Modes, symmetrised eigenvectors, and duality]
\label{lem:canonical-modes}
With $\kappa(z)=(p/q)^{a-z}$, $K=(p/q)^a$ and $\sigma(z)=a-z$:
\begin{enumerate}
\item[(i)] $A_\nu=(w_\nu/\sqrt K)\,D\,v_\nu$; that is, the modes are
the eigenvectors of $P$ obtained by undoing the symmetrising
conjugation, up to the mode-dependent normalisation
$w_\nu/\sqrt K$, which is independent of the site.
\item[(ii)] $B_\nu(z)=\kappa(z)\,A_\nu(\sigma(z))$ for all
$\nu$ and $z$, so \textup{(S3)} holds with $\nu$-independent
weights $\kappa$; and $\kappa(z)\kappa(\sigma(z))=K$, so
\textup{(S4)} holds.
\item[(iii)] $w_\nu=w_{a-\nu}$, and consequently the modes satisfy
the reflection identity
\begin{equation}
A_{a-\nu}(z)=(-1)^{z+1}A_\nu(z).
\label{eq:mode-reflection}
\end{equation}
\end{enumerate}
\end{lemma}

\begin{proof}
(i) Since $\sqrt{\kappa(z)}=(p/q)^{(a-z)/2}=\sqrt K\,(q/p)^{z/2}$,
we have $(q/p)^{z/2}=\sqrt{\kappa(z)}/\sqrt K$, and
\eqref{eq:canonical-modes} reads
$A_\nu(z)=(w_\nu/\sqrt K)\sqrt{\kappa(z)}\,v_\nu(z)$.
(ii) $\kappa(z)A_\nu(a-z)
=(p/q)^{a-z}w_\nu(q/p)^{(a-z)/2}\sin(\pi\nu(a-z)/a)
=w_\nu(p/q)^{(a-z)/2}\sin(\pi\nu(a-z)/a)=B_\nu(z)$; the second
claim is $(p/q)^{a-z}(p/q)^{z}=(p/q)^a$.
(iii) $w_{a-\nu}\propto\sin(\pi(a-\nu)/a)=\sin(\pi\nu/a)$. Hence
$A_{a-\nu}(z)=w_\nu(q/p)^{z/2}\sin(\pi z-\pi\nu z/a)
=(-1)^{z+1}A_\nu(z)$, using $\sin(\pi z-x)=(-1)^{z+1}\sin x$ for
integer $z$.
\end{proof}

\begin{remark}
The reflection identity \eqref{eq:mode-reflection} is the
mechanism behind the closed forms of
Section~\ref{sec:sign-theorem}: at a site of an active orbit it
pairs the modes $\nu$ and $a-\nu$, and the symmetry $w_\nu=w_{a-\nu}$
of the modal weights is what makes the pairing exact. It should not
be confused with the parity \eqref{eq:eigen-parity} of the
symmetrised eigenvectors, which reflects the \emph{site} $z\mapsto
a-z$ rather than the \emph{mode} $\nu\mapsto a-\nu$; both hold, and
\textup{(i)} above converts either into the other.
\end{remark}

Work henceforth in symmetrised coordinates,
$\widetilde\psi=D^{-1}\psi$, $\widetilde\xi=D^{-1}\xi$, so that
\[
\widetilde\psi(\gamma)=t\,(I-t\widetilde P)^{-1}\widetilde\xi ,
\]
and $\widetilde P=\sum_k\lambda_k P_k$ is a genuine spectral
resolution with orthogonal projectors
$P_k=(2/a)\,v_kv_k^{\!\top}$, normalised by
$\|v_k\|^2=a/2$. Note that $\psi$ itself is antisymmetric only when
$\kappa(z)=\kappa(\sigma(z))$, that is for the unbiased walk; it is
$\widetilde\psi$ that is antisymmetric for every drift.

\begin{lemma}[Parity of the source]
\label{lem:parity}
$S\,\widetilde\xi=-\widetilde\xi$: the symmetrised source is odd.
\end{lemma}

\begin{proof}
Only $z=1$ and $z=a-1=\sigma(1)$ carry source. With
$C^\ast=1/(1+\sqrt K)$, $K=(p/q)^a$,
\[
\xi_1=q(1-C^\ast)=\frac{q\sqrt K}{1+\sqrt K},\qquad
\xi_{a-1}=-C^\ast p=\frac{-p}{1+\sqrt K}.
\]
Oddness means $\xi_{a-z}/\sqrt{\kappa(a-z)}=-\xi_z/\sqrt{\kappa(z)}$;
the only nontrivial instance is $z=1$. With $\kappa(1)=(p/q)^{a-1}$,
$\kappa(a-1)=p/q$ and $\sqrt K=(p/q)^{a/2}$,
\[
\frac{\xi_{a-1}}{\sqrt{\kappa(a-1)}}=\frac{-\sqrt{pq}}{1+\sqrt K}
=-\frac{\xi_1}{\sqrt{\kappa(1)}},
\]
using $\sqrt K\,(q/p)^{(a-1)/2}=(p/q)^{1/2}$.
\end{proof}

\begin{lemma}[Activation of the odd spectral sector]
\label{lem:activation}
Every odd eigenmode of $\widetilde P$ has nonzero overlap with the
symmetrised source:
\begin{equation}
\langle\widetilde\xi, v_k\rangle = 2c\,\sin(\pi k/a),
\qquad k \text{ even},
\qquad
c = \frac{\sqrt p\,q^{(a+1)/2}}{p^{a/2}+q^{a/2}} > 0.
\label{eq:activation}
\end{equation}
Consequently every odd spectral projector is active,
$P_k\widetilde\xi\neq 0$ for all $k$ even.
\end{lemma}

\begin{proof}
By Lemma~\ref{lem:parity} the source is supported on $\{1,a-1\}$ with
$\widetilde\xi_1=-\widetilde\xi_{a-1}=c$, and
$c=\xi_1/\sqrt{\kappa(1)}
=\sqrt p\,q^{(a+1)/2}/(p^{a/2}+q^{a/2})>0$. Using
$\sin(\pi k(a-1)/a)=-(-1)^k\sin(\pi k/a)$,
\[
\langle\widetilde\xi, v_k\rangle
= c\,[\sin(\tfrac{\pi k}{a})-\sin(\tfrac{\pi k(a-1)}{a})]
= c\,[1+(-1)^k]\sin(\tfrac{\pi k}{a}),
\]
which equals $2c\sin(\pi k/a)$ for $k$ even and $0$ for $k$ odd.
Since $0<k/a<1$, $\sin(\pi k/a)>0$, so the overlap is nonzero for
every even $k$.
\end{proof}

Let $E_{\mathrm{odd}}=\operatorname{span}\{v_k : k \text{ even}\}$ be
the odd sector of the reflection $S$, of dimension
$r=\lfloor(a-1)/2\rfloor$.

\begin{theorem}[Finite spectral reduction]
\label{thm:fsr-spectral}
The reduction space, \emph{in symmetrised coordinates}, is exactly
the odd sector:
\begin{equation}
\widetilde V:=D^{-1}V = E_{\mathrm{odd}},
\qquad\text{equivalently}\qquad
V = D\,E_{\mathrm{odd}} .
\label{eq:V-eq-Eodd}
\end{equation}
The distinction matters: $D$ does not commute with $S$ unless
$p=q$, so $D\,E_{\mathrm{odd}}\neq E_{\mathrm{odd}}$ for a biased
walk --- consistently with the fact, noted above, that it is
$\widetilde\psi$ and not $\psi$ that is antisymmetric. The space
admits the canonical decomposition
\begin{equation}
\psi(\gamma)=\sum_{k\ \mathrm{even}} f_k(\gamma)\,\mathbf m_k,
\qquad
\mathbf m_k := D\,P_k\,D^{-1}\xi = D\,P_k\,\widetilde\xi,
\qquad
f_k(\gamma)=\frac{t}{1-\lambda_k t},
\quad t=1-\gamma .
\label{eq:fsr}
\end{equation}
The directions $\mathbf m_k$ are nonzero and linearly independent.
As for canonicity, the symmetrised directions $P_k\widetilde\xi$ are
determined by $(\widetilde P,S,\widetilde\xi)$; their
original-coordinate images are determined by $(P,S,\xi)$ alone. For
the present irreducible tridiagonal $P$ the positive diagonal
symmetriser is unique up to an overall scalar --- the ratios
$D_{z+1}/D_z$ being fixed by the neighbour relations --- and the
conjugation $D P_k D^{-1}$ is unchanged under $D\mapsto cD$; so no
normalisation of the symmetriser need be fixed. In particular $\dim V = r$, and all $\gamma$-dependence is
carried by the elementary scalars $f_k$.

The notation is deliberate: these $\mathbf m_k$ are \emph{not} the
normal directions $\mathbf n_\nu$ of Section~\ref{sec:finite-reduction}.
There are $N$ of the latter, one per spectral mode, spanning a space
of dimension $r$ and therefore dependent; there are exactly $r$ of the
former, indexed by the even $k$ alone, and they are independent. Both
families span $V$, which is all the two routes are claimed to share.
\end{theorem}

\begin{proof}
We argue in symmetrised coordinates and transport at the end. The
two inclusions are each supplied by one lemma.

First, $\widetilde V\subseteq E_{\mathrm{odd}}$. By
Lemma~\ref{lem:parity} the source $\widetilde\xi$ is odd, so
$\langle\widetilde\xi,v_k\rangle=0$ and hence
$P_k\widetilde\xi=0$ for every \emph{odd} index $k$, i.e.\ whenever
$v_k$ is even. Since
$\widetilde\psi(\gamma)=t\sum_k(1-\lambda_kt)^{-1}P_k\widetilde\xi$,
only the even indices survive and $\widetilde\psi(\gamma)\in
E_{\mathrm{odd}}$ for every $\gamma$. Note that this is a
consequence, not a construction: no antisymmetriser is applied
anywhere.

Second, $E_{\mathrm{odd}}\subseteq\widetilde V$: every such
projector is active (Lemma~\ref{lem:activation}) and the eigenvalues
are simple (Lemma~\ref{lem:sym}), so
Lemma~\ref{lem:resolvent-invariance}, applied with $A=\widetilde P$,
$\xi=\widetilde\xi$ and $J=(0,1)$ --- legitimate because
$\rho(\widetilde P)=2\sqrt{pq}\,\cos(\pi/a)<1$, since $pq\le1/4$ by
AM--GM, so that $I-t\widetilde P$ is invertible for every
$t\in(0,1)$ --- and $t$ ranges over $(0,1)$ as $\gamma$
does --- the scalar prefactor $t$ being harmless by the last clause
of that lemma --- gives
$\operatorname{span}\{\widetilde\psi(\gamma)\}
=\bigoplus_{k\ \mathrm{even}}\operatorname{span}\{P_k\widetilde\xi\}
=\bigoplus_{k\ \mathrm{even}}\operatorname{span}\{v_k\}
=E_{\mathrm{odd}}$. Hence $\widetilde V=E_{\mathrm{odd}}$ and
$\dim V=\dim\widetilde V=r$, the conjugation by $D$ being invertible.

For the decomposition, each even index contributes the single
elementary coefficient $f_k(\gamma)=t/(1-\lambda_kt)$. Passing back
via $\psi=D\widetilde\psi$ and setting
$\mathbf m_k=D\,P_k\widetilde\xi$ yields \eqref{eq:fsr}; the
$\mathbf m_k$ are images under the invertible $D$ of nonzero vectors
lying in pairwise orthogonal ranges, hence independent.
\end{proof}

\begin{remark}
The identity \eqref{eq:V-eq-Eodd} is the structural heart of the
result: in symmetrised coordinates the family
$\{\widetilde\psi(\gamma)\}$ generates \emph{exactly} the odd sector
of the reflection, no more and no less; equivalently, in the
original coordinates $V=D\,E_{\mathrm{odd}}$. The dimension count
$\dim V=r$ is an immediate corollary, the conjugation by $D$ being
invertible. The value
$r=\lfloor(a-1)/2\rfloor$ arises not from counting eigenvalue pairs but
because $\dim E_{\mathrm{odd}}=r$ and the source reaches every one of
its modes.
\end{remark}

\begin{remark}
The apparent $\{\lambda_k,-\lambda_k\}$ pairing of the spectrum is not
the organising structure: when both $\pm\lambda$ happen to be odd and
active (which occurs for some $a$), they contribute two independent
directions, each with its own elementary coefficient, not one direction
with a combined rational coefficient. The neutral eigenvalue
$\lambda=0$ (present iff $a$ is even) needs no special treatment: its
eigenvector is odd or even according to \eqref{eq:eigen-parity} and is
included in or excluded from $E_{\mathrm{odd}}$ automatically.
\end{remark}

\subsection{Positivity of the coefficients (model-specific)}
\label{subsec:positivity}

The following is a property of the reset model, logically separate from
the reduction itself.

\begin{lemma}
\label{lem:positivity}
For the biased walk, $f_k(\gamma)=t/(1-\lambda_k t)>0$ for all
$\gamma\in(0,1)$ and all $k$, where $t=1-\gamma$.
\end{lemma}

\begin{proof}
By AM--GM with $p+q=1$ we have $2\sqrt{pq}\le1$, with equality
exactly at $p=q$; and $|\cos(\pi k/a)|\le\cos(\pi/a)<1$ for
$k=1,\dots,a-1$. Hence
\[
|\lambda_k|=2\sqrt{pq}\,|\cos(\pi k/a)|\le\cos(\pi/a)<1,
\]
the strict inequality holding for every drift, the unbiased case
included. Therefore $1-\lambda_kt\ge1-|\lambda_k|>0$ for
$t\in[0,1]$, while $t>0$ on $\gamma\in(0,1)$.
\end{proof}

\begin{remark}
Positivity is not part of Theorem~\ref{thm:fsr-spectral}: the
reduction holds for any diagonalisable generator, on the range of
$t$ where the resolvent is defined. For spectral radius
$\rho(P)\ge1$ one restricts to $t<1/\rho(P)$; the sign of $f_k$ then
plays no role in the reduction, only later in the orientation
theorem of Section~\ref{sec:sign-theorem}.
\end{remark}

\begin{remark}[Structural hypotheses]
\label{rem:hypotheses-spectral}
Theorem~\ref{thm:fsr-spectral} --- the spectral route of this
appendix, not the abstract Theorem~\ref{thm:fsr} of the body, which
rests on \textup{(S2)--(S4)} and on no generator at all --- uses, in
the present finite-dimensional setting, three structural
ingredients: \emph{(i)} a symmetrisable generator, \emph{(ii)} an
involutive symmetry commuting with it, and \emph{(iii)} a source of
definite parity under that involution. Together with
diagonalisability and the separation of the active eigenvalues ---
both automatic here, since $\widetilde P$ is symmetric tridiagonal
with simple spectrum --- these yield a finite reduction whose
dimension is the number of active odd spectral directions. The
specific spectrum \eqref{eq:eigenpairs} and the value
$r=\lfloor(a-1)/2\rfloor$ are features of this family.

Two scope remarks. The resolvent lemma
\textup{(}Lemma~\ref{lem:resolvent-invariance}\textup{)} is where the
independence from resetting really lies: it is a statement about any
diagonalisable operator and refers to no reset mechanism at all. The
theorem above is not independent in that sense, since the form
$\psi=\mathcal G(t)\xi$ and the source \eqref{eq:xi-def} come from
the geometric-resetting realisation. And the ingredients
\emph{(i)--(iii)} are structural, not sufficient by themselves:
without spectral separation the count of directions is no longer the
number of odd eigenvectors met by the source.
\end{remark}

\section*{Acknowledgements}

We are grateful to Tony Newton for a careful reading of the earlier
papers in this series and for sending us the analysis
\cite{Newton}, whose contribution to the present work is
substantive and specific. The global orientation principle
of Conjecture~\ref{conj:global} does not follow from the structural
axioms even when the response span is one-dimensional and the
vertex condition holds: a further calibration between two
orientations is required, and it can fail. That this is a genuine
third requirement, and not a consequence of the first two, is his
observation, and it is what
Remark~\ref{rem:rotation-reversal}\textup{(iii)} records. Without
it the discussion of Conjecture~\ref{conj:global} in
Section~\ref{sec:inherited} would have been incomplete in a way we
had not detected.

His analysis also contains, obtained independently and by the same
argument, the total geodesy of the separatrix under the
Fisher--Rao metric proved here as
Theorem~\ref{thm:great-subsphere}. That two routes arrived at the
same proof suggests it is the natural one.

The behaviour of the response field when $r\ge2$ --- the regime his
examples open, and the one this paper leaves as the combinatorial
geometry of the rotating case --- remains the natural next problem. We also thank him for pointing out
that the reversal displayed in Figure~\ref{fig:response}(b) occurs
without any crossing of the neutral value $C^\ast$, which corrected
our description of its mechanism in
Section~\ref{subsec:fig-rotation}.

We also thank the colleagues who commented on earlier drafts of
this series; their questions shaped the presentation more than they
may realise.

\end{document}